\documentclass{article}

\usepackage{amsmath}
\usepackage{amsthm}
\usepackage{amsfonts}
\usepackage{amsbsy}
\usepackage{amssymb}
\usepackage[initials]{amsrefs}
\newtheorem{theorem}{Theorem}
\newtheorem{proposition}{Proposition}
\newtheorem{corollary}{Corollary}
\newtheorem{definition}{Definition}
\newtheorem{lemma}{Lemma}

\newcommand{\R}{{\mathbb R}}
\newcommand{\Z}{{\mathbb Z}}
\newcommand{\set}[2]{ \left\{ #1 \ \left| \ #2 \right. \right\}}

\usepackage[dvipsnames]{xcolor}

\newcommand{\mstar}{m_*}
\newcommand{\domain}{\Omega}
\newcommand{\basis}{\mathbf{e}}
\newcommand{\vecspan}{\mathop{\operatorname{span}}}

\newcommand{\dist}{\operatorname{dist}}
\newcommand{\SL}{\operatorname{SL}}
\newcommand{\GL}{\operatorname{GL}}

\title{Generalized Sublevel Estimates for Form-Valued Functions and Related Results for Radon-like Transforms}
\author{Philip T. Gressman\footnote{Partially supported by NSF grants DMS-2054602 and DMS-2348384.}}
\date{\today}

\begin{document}
\maketitle

\begin{abstract}
Motivated by the testing condition for Radon-Brascamp-Lieb multilinear functionals established in \cite{testingcond}, this paper is concerned with identifying local conditions on smooth maps $u(t)$ with values in the space of decomposable $p$-forms on some real vector space $V$ which guarantee uniform integrability of $||u(t)||^{-\tau}$ over a certain natural, noncompact family of norms. One can loosely regard this problem as a higher-dimensional analogue of establishing uniform bounds for the size of a sublevel set of a function in terms of the size of its derivatives. The resulting theorem relies extensively on ideas from Geometric Invariant Theory to understand what appropriate derivative bounds look like in this context. Several examples and applications are presented, including a new local characterization of so-called ``model'' Radon-like transforms in terms of the semistability of a natural curvature functional (giving an equivalent but rather different criterion than the one first established in \cite{gressmancurvature}).
\end{abstract}

\tableofcontents

\section{Introduction}
\label{introduction}

Suppose $\domain \subset \R^d$ is open and that $V$ is an $q$-dimensional real vector space which is equipped with a nonzero alternating $q$-linear form called its determinant functional. Given smooth functions $u^1,\ldots,u^{p}$ from $\domain$ into $V^*$ with $1 \leq p \leq q$ and linearly-independent vectors $\{\omega^j\}_{j=1}^q$ in $V$, one can define an associated norm for $p$-forms $u^1 \wedge \cdots \wedge u^p$ with respect to the basis $\{\omega^j\}_{j=1}^q$ by means of the formula
\begin{equation} ||u(t)||_{\{\omega^j\}_{j=1}^q}^2 := \frac{1}{p!} \sum_{j_1=1}^q \cdots \sum_{j_p=1}^q \left| \det \begin{bmatrix} u^1(t) \cdot \omega^{j_1} & \cdots & u^1(t) \cdot \omega^{j_p} \\ \vdots & \ddots & \vdots \\ u^p(t) \cdot \omega^{j_1} & \cdots & u^p(t) \cdot \omega^{j_p} \end{bmatrix} \right|^2, 
\label{normdef} \end{equation}
where $\cdot$ in this case indicates the natural pairing between elements of $V^*$ and $V$. When no confusion arises, the shortened notation $||u(t)||_{\omega}$ will be used to indicate $||u(t)||_{\{\omega^j\}_{j=1}^q}$.

Motivated by \cite{testingcond}, the goal of the present paper is to identify local or geometric conditions on $u(t)$ under which there exists a nonnegative weight function $w$ on $U$ and some exponent $\tau > 0$ such that the integral
\begin{equation} \int_{\domain} \frac{w(t) dt}{\left[ || u(t)||_{\omega} \right]^\tau} \label{theintegral} \end{equation}
is uniformly bounded above for all choices of $\{\omega^j\}_{j=1}^q$ which span a parallelepiped of volume one, i.e., $|\det (\omega^1,\ldots,\omega^q)| = 1$, where $\det$ is the given determinant functional on $V$. As such, this paper is aligned with many other previous results in various settings which seek to make precise the informal observation that smooth objects with large derivatives cannot be globally small in magnitude. While many robust results of this sort exist in one dimension (including Arhipov, Karacuba, and \v{C}ubarikov \cite{akc1979}; Phong, Stein, and Sturm \cite{pss1999}; Rogers \cite{rogers2005}; and Stein and Wainger \cite{sw2001}), working in higher dimensions presents a host of new challenges. A variety of interesting and important approaches have been developed to address these challenges. The works of Carbery, Christ, and Wright \cite{ccw1999}; Carbery and Wright \cite{cw2002}; Greenblatt \cite{greenblatt2006}; and Phong, Stein, and Sturm \cite{pss2001} are particularly noteworthy and showcase very different techniques and ways of thinking about the problem.

One lesson learned very early is that, when proving general sublevel set inequalities, imposing lower bounds on several mixed partial derivatives of a function yields results which are no better than what is obtained by imposing each individual constraint in isolation (see the concluding remarks of Phong, Stein, and Sturm \cite{pss2001} for an example). What this appears to mean in practice is that one must develop a more nuanced and geometric understanding of how partial derivatives relate to one another in order to robustly quantify the behavior of smooth functions in higher dimensions. An example of such an approach is the result of Carbery \cite{carbery2009}, which establishes uniform sublevel set bounds for convex functions under the condition that their Hessian determinant is bounded below; related results are of use in Monge-Amp\`{e}re theory \cite{lmt2017}. The Hessian determinant is a particularly natural differential operator to use when studying functions of several variables because it is invariant under all linear coordinate changes which preserve volume. This particularly large symmetry group may be one reason why lower bounds for the Hessian determinant imply uniform sublevel set estimates in some natural cases despite the fact that lower bounds for the Laplacian do not (see, e.g., Carbery, Christ, and Wright \cite{ccw1999} or Steinerberger \cite{steinerberger2021}).

Rather than identifying or constructing nonlinear polynomials like the Hessian determinant which govern the behavior of the integral \eqref{theintegral}, the approach to be taken here is somewhat different and is heavily influenced by Geometric Invariant Theory and real Kempf-Ness theory specifically \cites{kn1979,bl2021}. Roughly speaking, one looks at the Taylor series jet of the function when expressed in essentially arbitrary coordinates and requires that the jets are always quantitatively nontrivial. In practice, this means imposing lower bounds on the norm of the jet as it ranges over the noncompact (but still linear) group of coordinate transformations that naturally arise in the problem. An approach of this sort was successful in characterizing Oberlin's affine Hausdorff measure \cite{gressman2019}, and it turns out in the current context to again be a quite successful and robust way to use local differentiability properties of $u(t)$ to establish uniform bounds on the integrals \eqref{theintegral}.

\subsection{Main Result}

The main quantitative result revolves around a sort of generalized row and column reduction which can be applied to the form $u(t)$ to progressively eliminate low-order terms in Taylor series jets. The outcome of such a process is most succinctly described in terms of decreasing families of subspaces and a differential condition which applies to pairings of elements in these subspaces. The precise structure is described in the following sequence of definitions.
\begin{definition}
For each $t \in \domain$, let $W(t)$ be some $r$-dimensional subspace of $V$. This $W(t)$ is called a smooth family of subspaces of $V$ when for every $t_0 \in \domain$, there is a neighborhood $\domain_0 \subset \Omega$ of $t_0$ and a smooth nonvanishing $w_0(t)$ defined on $\domain_0$ with values in the exterior algebra $\Lambda^r(V)$ such that at all $t \in \domain_0$, there exist linearly independent $w^1,\ldots,w^r \in W(t)$ for which $w^1 \wedge \cdots \wedge w^r = w_0(t)$. Any such pair $(\domain_0,w_0)$ is called a local parametrization of $W(t)$.
\end{definition}

Note in the above definition that the $r$-form $w_0(t)$ is assumed to be smooth but it is not assumed \textit{a priori} that the individual vectors $w^1,\ldots,w^r$ possess regularity of any kind in the parameter $t$.

\begin{definition}
Given families of smooth subspaces $V \supset W_0(t) \supset W_1(t) \supset \cdots \supset W_{m}(t) \supset W_{m+1}(t) := \{0\}$ for $t \in \domain$ and some smooth $w(t)$ on $\domain$ with values in $\Lambda^r(V)$ such that $(\domain,w)$ is a local parametrization of $W_0(t)$, we say that $V$-valued functions $\{\overline{w}^{i,i'}(t)\}$ for $i \in \{0,\ldots,m\}, i' \in \{1,\ldots,\dim W_i - \dim W_{i+1}\}$ are an adapted factorization of $w(t)$ when the following conditions hold:
\begin{itemize}
\item For each $t \in \domain$,
\[ \bigwedge_{i=0}^{m} \left[ \bigwedge_{i'=1}^{\dim W_i - \dim W_{i+1}} \overline{w}^{i,i'}(t) \right]  = w(t). \]
\item For each $i \in \{0,\ldots,m\}$, the vectors $\{\overline{w}^{i,i'}(t)\}_{i'=1}^{\dim W_i - dim W_{i+1}}$ belong to $W_i(t)$ and are linearly independent modulo $W_{i+1}(t)$.
\end{itemize}
\end{definition}

As an important side note, for any real vector field $V$ of dimension $q$ which comes equipped with a determinant functional, it is always the case that $v^1 \wedge \cdots \wedge v^q = w^1 \wedge \cdots \wedge w^q$ when $\det (v^1,\ldots,v^q) = \det (w^1,\ldots,w^q) = 1$. A collection $\{\overline{v}^{j,j'}\}_{j,j'}$ will be called an adapted factorization of $V$ when it is an adapted factorization of $v^1 \wedge \cdots \wedge v^q$ for any volume one basis $v^1,\ldots,v^q$ (where $m=0$ and the $v^i$ are regarded as constant functions of $t$).

\begin{definition}
Let $U_0(t)$ be a smooth family of $p$-dimensional subspaces of $V^*$ for $t \in \domain$. Suppose there exist smooth families of subspaces $U_0(t) \supset \cdots \supset U_{\mstar}(t) \supset U_{\mstar+1}(t) := \{0\}$ and $V =: V_0(t) \supset \cdots \supset V_{m}(t) \supset V_{m+1}(t) := \{0\}$ and nonnegative integers $D_{ij}$ for $i \in \{0,\ldots,\mstar\},j \in \{0,\ldots,m\}$, such that for each $t_0$, there is a neighborhood $\domain_0$ of $t_0$ and local parametrizations $(\domain_0,u_i)$ and $(\domain_0,v_j)$ of $U_i(t)$ and $V_j(t)$, respectively, such that for all $t \in \domain_0$,
\begin{equation} \overline{u} \cdot (\partial^\gamma_t v_j(t)) = 0 = (\partial^\gamma_t u_i(t)) \cdot \overline{v} \label{reductiondef} \end{equation}
when $\overline{u} \in U_i(t)$, $\overline{v} \in V_j(t)$, and $\gamma$ is any multiindex such that the order $|\gamma|$ is strictly less than $D_{ij}$.
 Here $\cdot$ indicates the interior product.
 Such a family of nested subspaces is called a block decomposition of $(U_0,V)$. The integers $D_{ij}$ will be called formal degrees of pairings of $U_i(t)$ and $V_j(t)$.
 \end{definition}
 It will be assumed throughout that $D_{ij}$ is nondecreasing as a function of $i$ and $j$ individually (i.e., nondecreasing in $i$ for each fixed $j$ and vice-versa). It will also be assumed without loss of generality that the subspaces $U_i(s)$ and $V_j(t)$ are strictly decreasing in dimension as a function of $i$ and $j$. Given a block decomposition, it will be standard to define $p_i := \dim U_{i} - \dim U_{i+1}$ for each $i \in \{0,\ldots,\mstar\}$, and $q_j := \dim V_j - \dim V_{j+1}$ for $j \in \{0,\ldots,m\}$.

For a simple example of a block decomposition, consider the vector space $V = \R^5$, regarded as column vectors, with dual space $V^* = \R^5$, regarded as row vectors. Given the matrix
\begin{equation} \begin{bmatrix} 1 & 0 & 0 & s & 0  \\ 0 & 1 & 0 & 0 & s \\  -s & -\frac{s^2}{2} & 1 & -\frac{s^2}{2} & -\frac{s^3}{3} \end{bmatrix}  \label{matrixU} \end{equation}(where $s$ here is a real parameter), let $U_0(s)$ be the span of all three rows (with local parametrization $u_0(s)$ given by the wedge of the three rows) and let $U_1(s)$ be the span of the final row. Likewise, given the matrix 
\begin{equation} \begin{bmatrix} 1 & 0 & 0 & -t & t^2 \\ 0 & 1 & 0 & 0 & -t \\ 0 & 0 & 1 & - \frac{t^2}{2} & \frac{t^3}{3} \\ 0 & 0 & 0 & 1 & -t \\ 0 & 0 & 0 & 0 & 1 \end{bmatrix},  
\label{matrixV} \end{equation}
let $V_0(t) = \R^5$, $V_1(t)$ be the span of the two rightmost columns, and let $V_2(t)$ be the span of the final rightmost column. The product of matrices \eqref{matrixU} and \eqref{matrixV} is
\[ \left[\begin{array}{ccc|c|c} 1 & 0 & 0 & s-t & t(t-s) \\ 0 & 1 & 0 & 0 & s-t \\ \hline -s & -\frac{s^2}{2} & 1 & - \frac{(s-t)^2}{2} & - \frac{(s-t)^3}{3} \end{array} \right]. \]
Here the vertical lines mark the transitions from $V_0$ to $V_1$ and $V_1$ to $V_2$, respectively, while the horizontal line marks the transition from $U_0$ to $U_1$. The computations from Proposition \ref{blockprop} in Section \ref{adaptsec} allow one to read off viable formal degrees by simply finding the order of vanishing along the diagonal $t=s$ of each block of this matrix. In this case, $U_0,U_1$ and $V_0,V_1,V_2$ form a block decomposition of $(U_0,V)$ with formal degrees $D_{00} = D_{10} = 0$, $D_{01} = D_{02} = 1$, $D_{11} = 2$, and $D_{12} = 3$. Block decompositions can informally be regarded as a bookkeeping system for recording the rates at which various key subspaces of $U_0(s)$ vary as $s$ varies and the directions in which they vary. This, of course, does not yet illuminate how to \textit{construct} a block decomposition for a given $(U_0,V)$, but Section \ref{examplesec} introduces an algorithm involving joint row and column reduction of a matrix like \eqref{matrixU} which generically produces such a block decomposition.

One final piece of necessary information concerns the regularity constraints imposed on $u$. Throughout most of this paper, it will be assumed that $u$ exhibits Nash regularity. More details are available in Section \ref{nash}; roughly speaking, this means that $u$ belongs to a class of real analytic functions which contains and extends polynomials. Nash functions are equipped with an integer $K$ called complexity which generalizes the notion of degree. To simplify formulas somewhat, when a function or other object is described as having complexity at most $K$, it will always be assumed that $K \geq 2$. Because the dependence of constants on the Nash complexity parameter $K$ has important implications (allowing the passage from Nash to smooth Radon-like transforms in \cite{gressmancurvature}), this dependence will be tracked throughout the paper.

With the definitions in place, the main result of this paper is as follows.

\begin{theorem}
Let $\domain \subset \R^d$ be open and $u^1(t),\ldots,u^p(t)$ on $\domain$ be smooth $V^*$-valued functions such that $u(t) := u^1(t) \wedge \cdots \wedge u^p(t)$ is nonvanishing and is Nash of complexity at most $K$.  Let $U_0(t)$ be the family of subspaces parametrized by $(\domain,u)$ and suppose that $(U_0,V)$ admits a block decomposition with subspaces $U_0(t) \supset \cdots \supset U_{\mstar}(t) \supset \{0\}$ and $V =: V_0(t) \supset \cdots \supset V_m(t) \supset \{0\}$ with formal degrees $D_{ij}$. \label{bigtheorem1} 

Suppose there exists a nonnegative measurable function $w(t)$ on $\domain$ and some $\sigma > 0$ such that at every $t \in \Omega$, all adapted factorizations $\{\overline{u}^{i,i'}\}_{i,i'}$ and $\{\overline{v}^{j,j'}\}_{j,j'}$ of $u(t)$ and $V$, respectively, satisfy the inequality
\begin{equation} |\det M|^{\sigma} (w(t))^\sigma \leq \left( \sum_{i=0}^{\mstar} \sum_{i'=1}^{p_i} \sum_{j=0}^m \sum_{j'=1}^{q_j} \sum_{|\alpha| = D_{ij}} \frac{ \left|\overline{u}^{i,i'}(t) \cdot [ (M \partial)^\alpha \overline{v}^{j,j'}(t)] \right|^2 }{\alpha!} \right)^{\frac{1}{2}} \label{nondegenhyp} \end{equation}
for any $d \times d$ invertible matrix $M$ (see \eqref{matchangedef} for the definition of $M \partial$; here $\alpha$ is a multiindex in $\mathbb{Z}_{\geq 0}^d$).
Then there exists a constant $C$ depending only on $p,q,d$, and the $D_{ij}$ such that
\[ \int_\domain \frac{w(t) dt}{\left[ ||u(t)||_{\omega}\right]^{1/(p \sigma)}} \leq C K^{C} \]
for every choice of basis $\omega^1,\ldots,\omega^q$ of $V$ satisfying $|\det(\omega^1,\ldots,\omega^q)| = 1$.
\end{theorem}

When viewing Theorem \ref{bigtheorem1} as a computational tool, the challenge is, of course, to determine whether any nontrivial weight function $w$ exists satisfying \eqref{nondegenhyp}. Meeting the challenge can indeed be difficult in some cases. To ease the burden in this regard, Section \ref{usefulsec} gives a number of secondary technical results directed at understanding the condition \eqref{nondegenhyp}. In particular, Theorem \ref{usefulthm} relates the inequality \eqref{nondegenhyp} to the tools and results of Geometric Invariant Theory to give a number of ways to verify that the hypothesis does indeed hold. Theorem \ref{usefulradon} applies these results to the general setting of Radon-like transforms to establish a class of uniform $L^p$-improving estimates for Radon-like transforms equipped with measures which generalize affine arc length and affine hypersurface measure.

\subsection{Principal Application: Model Radon-Like Transforms}

Theorem \ref{bigtheorem1} has important implications Radon-like transforms thanks to the testing condition established in \cite{testingcond}. The result is that it is possible to obtain a second (and independent) local criterion which is both necessary and sufficient for a Radon-like transform to be a so-called ``model operator,'' meaning that it exhibits best-possible $L^p$-improving properties. Establishing boundedness of model operators and understanding the underlying curvature condition has been a topic of interest for some time (see, for example, Ricci \cite{ricci1997} and Oberlin \cite{oberlin2008}). In particular, with Theorem \ref{bigtheorem1}, it is possible to give a more succinct and natural description of the necessary curvature condition than was possible in \cite{gressmancurvature}.

The key condition will be formulated for Radon-like transforms written in double fibration form:
\begin{equation} \int_{\mathcal M} g(\pi_2(z)) f(\pi_1(z)) a(z) d \sigma(z). \label{doublefibration} \end{equation}
Here $f$ is a function on $\R^{n_1}$, $g$ is a function on $\R^{n}$, $\mathcal M$ is some manifold of dimension $n_1 + n -k$ (where it is assumed that the codimension $k$ is less than $\min\{n,n_1\}$), $\sigma$ is some measure of smooth density on $\mathcal M$, and $a$ is some continuous function of compact support in $\mathcal M$. The maps $d \pi_1$ and $d \pi_2$ are assumed to have surjective differentials on a neighborhood of the support of $a$. Let $X^1,\ldots,X^{n-k}$ be smooth vector fields spanning the kernel of $d \pi_1$ near some point $z_0 \in \mathcal M$ in the support of $a$ and let $Y^1,\ldots,Y^{n_1-k}$ be smooth vector fields spanning the kernel of $d \pi_2$ near $z_0$. Let $X$ be any smooth linear combination of the $X^i$ and let $Y$ be any smooth linear combination of the $Y^i$. The map
\begin{equation} (X,Y) \mapsto [X,Y] \label{theqmap} \end{equation}
(where $[\cdot,\cdot]$ denotes the Lie bracket) sends $X$ and $Y$ to a new vector field on $\mathcal M$. Modulo $\ker d \pi_1 + \ker d \pi_2$, the value of $[X,Y]$ depends only on the pointwise values of $X$ and $Y$ since if $X = X'$ at a point $z_0$, then
\[ X - X' = \sum_{i=1}^{n_1-k} c_i(z) X^{i} \]
for smooth functions $c_i$ vanishing at $z=z_0$, which forces
\[ [X - X',Y] = \sum_{i=1}^{n_1-k} c_i(z) [X^{i},Y] + \sum_{i=1}^{n_1-k} (Y c_i(z)) X^i; \]
the first sum must vanish at $z = z_0$, meaning that $[X-X',Y] \in \ker d \pi_{1}$ at $z_0$. An analogous argument establishes that $[X,Y'-Y] \in \ker d \pi_2$ when $Y'$ and $Y$ are vector fields in the kernel of $d \pi_2$ which agree at $z_0$. Thus at every point $z \in \mathcal M$, there is a canonical bilinear map:
\begin{equation} Q_{z} : \ker d \pi_1|_z \times \ker d \pi_2|_z \rightarrow T_{z}(\mathcal{M}) / (\ker d \pi_1 |_z + \ker d \pi_2 |_z) \label{qdescribed} \end{equation}
given by \eqref{theqmap}. 
The main result of \cite{gressmancurvature} was a local condition on this map $Q_z$ phrased in terms of an associated convex polytype reminiscent of the Newton polytope. An interesting and important consequence of Theorem \ref{bigtheorem1} is that it allows one to understand model operators in terms of semistability, which is a concept already familiar in the context of Geometric Invariant Theory.

\begin{definition} Fix any norm $|| \cdot ||$ on the vector space of bilinear maps to which $Q_z$ belongs. We say that $Q_z$ is semistable when there is some positive constant $c$ such that every map of the form
$(X,Y) \mapsto C Q_z(A X, B Y)$ has norm at least $c$ for all linear operators $A, B,$ and $C$ on $\ker d \pi_1 |_z$, $\ker d \pi_2 |_z$ and $T_{z}(\mathcal{M}) / (\ker d \pi_1 |_z + \ker d \pi_2 |_z)$, respectively, which have determinant $1$.
\end{definition}

In Section \ref{neccsec}, it is established that \eqref{doublefibration} cannot give rise to a model operator when the kernels of $d \pi_1$ and $d \pi_2$ fail to be transverse. One can thus assume without loss of generality that the map $z \mapsto (\pi_2(z), \pi_1(z))$ is locally an embedding. This implies (reordering the coordinates of $\pi_1(z)$ if necessary) the existence of coordinates $z(x,t)$ on $\mathcal M$ defined near $z_0 \in \mathcal M$ such that $x \in \R^n$, $t \in \R^{n_1-k}$, and $(\pi_2(z(x,t)),\pi_1(z(x,t))) = (x,(t,\phi(x,t)))$ for some smooth $\phi(x,t)$. We say that \eqref{doublefibration} is Nash of complexity at most $K$ when $x \mapsto \phi(x,t)$ is Nash with complexity at most $K$ for each $t$ and $t \mapsto \frac{\partial \phi}{\partial x}(x,t)$ is Nash with complexity at most $K$ for each $x$. If, for example, $\phi$ is a polynomial, then the regularity condition will be satisfied by simply taking $K$ to be its degree.

\begin{theorem}
Fix  $z_0 \in \mathcal{M}$. Suppose that near $z_0$, \eqref{doublefibration} has $\ker d \pi_1 \cap \ker d \pi_2 = \{0\}$ and it admits $\phi(x,t)$ as just described. Then there exists a smooth function $a$ defined on a neighborhood of $z_0$ in $\mathcal M$ with $a(z_0) \neq 0$ such that \label{characterize}
\begin{equation} \left| \int_{\mathcal M} g(\pi_2(z)) f(\pi_1(z)) a(z) d \sigma(z) \right| \leq C ||g||_{L^{\frac{k(n_1-k)}{n_1(n-k)}+1}} ||f||_{L^{\frac{k (n-k)}{n(n_1-k)}+1}} \label{thebnd} \end{equation}
for some finite constant $C$ and all $f$ and $g$ in their respective Lebesgue spaces if and only if $Q_{z_0}$ is semistable. The constant $C$ grows at most like a fixed power (depending only on dimensions) of the Nash complexity $K$.
\end{theorem}

Although the details will not be pursued here, the methods of Section 7 of \cite{gressmancurvature} can be applied to Theorem \ref{characterize} to show that it continues to hold in the $C^\infty$ category at the cost of $\epsilon$ losses in the Lebesgue exponents. 

\subsection{Notation and Summary}

We will make extensive use of multiindex notation: any $\alpha := (\alpha_1,\ldots,\alpha_d) \in \Z^{d}$ with nonnegative entries will be called a multiindex on $\R^d$. The order or degree of $\alpha$ will be denoted $|\alpha|$ and equals $\alpha_1 + \cdots + \alpha_d$. Similarly $\alpha! := \alpha_1! \cdots \alpha_d!$ and $x^{\alpha} := \prod_{i=1}^d x_i^{\alpha_i}$ when $x := (x_1,\ldots,x_d) \in \R^d$. Regarding derivatives, we define
\[ \partial^\alpha := \prod_{i=1}^d \left( \frac{\partial}{\partial x_i} \right)^{\alpha_i}. \]
When there are several sets of variables and potential confusion may arise, a subscript will denote those to which the derivatives apply (i.e., $\partial_x^{\alpha}$ indicates partial derivatives in the $x$ variables only).
If $M$ is any $d \times d$ real matrix with entries $M_{ij}$ in the standard basis, we let
\begin{equation} (M \partial)_k := \sum_{k'=1}^d M_{kk'} \frac{\partial}{\partial x_{k'}} \label{matchangedef} \end{equation}
and note that the chain rule dictates that
\[ (M \partial)_k f(0) = \partial_k|_{z=0}\left[ f(M^T z) \right], \]
where $M^T$ is the transpose of $M$. When multiindex and matrix notation are mixed, one defines
\[ (M \partial)^{\alpha} := \prod_{i=1}^d \left( (M \partial)_i \right)^{\alpha_i}. \]

For each positive integer $k$, the symbol $\mathbf{1}_k \in \R^k$ denotes the vector whose coordinates are all $1$: $\mathbf{1}_k := (1,\ldots,1)$. The symbol $\delta_{kk'}$ is the Kronecker $\delta$ and is defined to equal $1$ when $k = k'$ and zero when $k \neq k'$. Each vector space $\R^k$ is also assumed to be equipped with a standard basis $\{\basis^i\}_{i=1}^k$ defined as usual so that $\basis^i$ has coordinate $1$ in position $i$ and zeros elsewhere. When dealing with long tuples of real numbers, semicolons will occasionally be used as separators in place of the more familiar comma; in most cases, this is done to highlight product structure. For example, the notation $(k^{-1} \mathbf{1}_k ; n^{-1} \mathbf{1}_n)$ will denote the element of $\R^k \times \R^n$ whose first $k$ entries equal $k^{-1}$ and whose final $n$ entries equal $n^{-1}$.

Given two integers $i_L \leq i_R$ (in particular, possibly equal), the notation $[i_L,i_R]$ will be used to indicate the interval in $\Z$ with these endpoints:  $[i_L,i_R] := \set{i \in \Z}{i_L \leq i \leq i_R}$.  The symbol $\mathbf{1}_{[i_L,i_R]}$ will be understood to be the indicator function of $[i_L,i_R]$.

The paper is structured as follows. Section \ref{background} is devoted to the proofs of a number of basic propositions and other results which are needed for Theorems \ref{bigtheorem1} and \ref{characterize}. Of particular note is Section \ref{gitsec}, where the reader will find elementary, standalone proofs for a number of results from Geometric Invariant Theory which are needed when establishing that the hypothesis \eqref{nondegenhyp} holds for some nontrivial weight $w$. Many of these results are already known in much greater generality than the forms appearing in Section \ref{gitsec}, but it is useful nevertheless to approach Geometric Invariant Theory from the perspective of analysis rather than algebra to see why these objects and ideas yield effective tools for quantifying derivatives in higher dimensions and ultimately for bounding the integral \eqref{theintegral}. Section \ref{nash} is a very rapid review of the properties of Nash functions; see \cite{gressmancurvature} for more details. Section \ref{kernelsec} proves the main results (e.g., Proposition \ref{kernelparam}) which are needed in Section \ref{adaptsec} to establish the key properties of block decompositions.

At their core, block decompositions should be thought of as a sort of generalization of matrix row and column reduction which apply to matrices whose entries are functions rather than scalars. The underlying symmetry properties of forms $u(t)$ allow one to engage in a limited sort of elimination process for these variable matrices. In general, while it isn't possible to eliminate entries entirely, it will be possible to eliminate low order terms of the Taylor series expansions of the entries and be left with functions which vanish to increasingly high orders. The condition \eqref{reductiondef} in the definition of block decompositions records the exact criterion which quantifies what is needed---as one moves to smaller subspaces $U_0(t) \supset U_1(t) \supset \cdots$ and $V_0(t) \supset V_1(t) \supset \cdots$, the reward is that the pairings \eqref{reductiondef} vanish to increasingly high orders. One question which will not be deeply explored in the present paper is the extent to which one can fully automate the process of generating block decompositions. Proposition \ref{blockprop2} is a first step in this direction, but is limited to only that which is needed for the proof of Theorem \ref{bigtheorem1}. In all other cases explored here, it is easier to use ad-hoc methods to produce block decompositions than it would be to invoke heavier machinery.

Section \ref{proofsec} contains the proof of Theorem \ref{bigtheorem1}. Aside from the tools developed in Section \ref{background}, Theorem 1 also relies heavily on the geometric differential inequalities which have appeared in various forms in other recent results \cites{gressman2019,gressmancurvature}. The key features of these inequalities are recalled in Theorem \ref{mainineq} in Section \ref{gdineqsec}. Section \ref{actualproof} contains the proof of Theorem 1.

Section \ref{usefulsec} further explores the relationship between Geometric Invariant Theory and Theorem \ref{bigtheorem1}. Theorem \ref{usefulthm} of this section establishes a sufficient condition for proving the inequality \eqref{nondegenhyp} which involves a certain Newton-like polytope (one related to but still simpler than the one in \cite{gressmancurvature}). Theorem \ref{usefulradon} connects Theorem \ref{usefulthm} to the motivating application of Radon-like transforms. The resulting theorem is in the same general spirit as the work of Christ, Dendrinos, Stovall and Street \cite{cdss2020}, the key differences being that \cite{cdss2020} is more general in terms of the H\"{o}rmander condition configurations it admits, while Theorem \ref{usefulradon} is more restricted in the H\"{o}rmander sense (Theorem \ref{usefulradon} deals only with ``one-sided'' H\"{o}rmander spanning sets) but applies to averages beyond curves.

Section \ref{examplesec} includes a number of examples of varying complexity to demonstrate the power and flexibility of the tools developed in Sections \ref{background} and \ref{usefulsec}. 

Section \ref{modelsec} contains the proof of Theorem \ref{characterize}. Sufficiency of semistability is shown in Section \ref{suffsec} as an application of Theorem \ref{usefulsec}. Section \ref{neccsec} establishes necessity and is based on a family of Knapp-type examples. Finally, Section \ref{remarksec} contains some concluding remarks.

\section{Preparatory Tools and Computations}
\label{background}
\subsection{Tools from Geometric Invariant Theory}
\label{gitsec}

This section introduces the basic set of tools, ideas, and results from Geometric Invariant Theory which will play a role in the estimation of the integral \eqref{theintegral}. For the most part, the propositions and corollaries found here are special cases of much more general results, but there is significant utility in giving elementary, standalone proofs of the facts we will use. In particular, the presentation found here, while limited, does not require that the reader have any previous knowledge of the theory of linearly reductive real algebraic groups.

The main objects of interest are matrices whose entries are real polynomials. Specifically, suppose that $P(x,y,z)$ is bilinear in the variables $x := (x_1,\ldots,x_p) \in \R^p$ and $y := (y_1,\ldots,y_q) \in \R^q$ with values which are polynomials of degree at most $D$ in the variable $z \in \R^d$, i.e.,
\[ P(x,y,z) := \sum_{i=1}^p \sum_{j=1}^q \sum_{|\alpha| \leq D} c_{ij\alpha}  x_i y_j z^{\alpha}. \]
We fix a norm of Hilbert-Schmidt type on the vector space of such objects:
\begin{equation} ||P||^2 := \sum_{i=1}^p \sum_{j=1}^q \sum_{|\alpha| \leq D} \frac{1}{\alpha!} \left| \partial^{\alpha}_z P(\basis^i,\basis^j,z)|_{z=0} \right|^2. \label{stdnorm} \end{equation}
There is a natural representation of the group $\SL(p,\R) \times \SL(q,\R) \times \GL(d,\R)$ acting on $P$ which is given by
\[ (\rho_{(A,B,C)} P)(x,y,z) := P(A^T x, B^T y, C^T z) \]
for any $(A,B,C) \in \SL(p,\R) \times \SL(q,\R) \times \GL(d,\R)$. For any $\sigma > 0$, a fundamental question which will be repeatedly encountered in the computations that follow is whether the infimum
\begin{equation} ||| P|||_{\sigma} := \inf_{A \in \SL(p,\R)} \inf_{B \in \SL(q,\R)} \inf_{C \in \GL(d,\R)} |\det C|^{-\sigma} || \rho_{(A,B,C)} P|| \label{gitnorm} \end{equation}
is strictly positive or not. Informally, strict positivity of $|||P|||_\sigma$ means that there do not exist bases of $\R^p, \R^q$, and $\R^d$ in which $P$ looks ``abnormally small.'' The parameter $\sigma$ allows some nuance in quantifying exactly what ``abnormal'' means in this context. When $P$ happens to have entries which are all homogeneous of degree $D$, a simple scaling argument shows that $|||P|||_{\sigma} = 0$ whenever $\sigma \neq D d^{-1}$. The particular case $\sigma = D d^{-1}$ reduces to one in which real Kempf-Ness theory readily applies; in particular, $|||P|||_{\sigma} > 0$ exactly when $P$ is semistable with respect to the representation $\rho$ (see \cite{bl2021}, for example). For general (i.e., not necessarily homogeneous) $P$, the range of $\sigma$ for which $|||P|||_{\sigma}$ is not identically zero will be contained in the interval $[0,D d^{-1}]$.

The first and most fundamental fact to observe regarding \eqref{gitnorm} is that one always has
\begin{equation} |||P|||_{\sigma} = \big|\big|\big| |\det C|^{-\sigma} \rho_{(A,B,C)} P \big|\big|\big|_{\sigma} \label{groupprop} \end{equation}
for any $\SL(p,\R) \times \SL(q,\R) \times \GL(d,\R)$, which follows rather directly from the fact that $\rho$ is a representation (and that $|\det C_1 C_2|^{-\sigma} = |\det C_1|^{-\sigma} |\det C_2|^{-\sigma}$).
Another equally fundamental observation is that the standard norm \eqref{stdnorm} exhibits an extremely convenient invariance property: $||P|| = ||\rho_{(O_1,O_2,O_3)} P||$ for all orthogonal matrices $O_1,O_2,O_3$. Specifically for $O_1$ and $O_2$, the proof is an elementary consequence of the definition \eqref{stdnorm} and the definition of orthogonal matrices; invariance in $O_3$ is a consequence of the following proposition.
\begin{proposition}
Suppose $P$ is a homogeneous polynomial of degree $D$ on $\R^d$. Then for any $d \times d$ orthogonal matrix $O$, \label{proportho}
\[ \sum_{|\alpha|=D} \frac{1}{\alpha!} |\partial^\alpha P(O z)|_{z=0}|^2 = \sum_{|\alpha|=D} \frac{1}{\alpha!} |\partial^\alpha P(0)|^2. \]
\end{proposition}
\begin{proof}
The key observation is that
\[ \sum_{|\alpha| = D} \frac{D!}{\alpha!} |\partial^\alpha P(z)|^2 = \sum_{i_1=1}^d \cdots \sum_{i_D=1}^d |\partial_{i_1} \cdots \partial_{i_D} P(z)|^2 \]
by virtue of the fact that there are exactly $D!/\alpha!$ ways to represent $\partial^\alpha$ as an iterated derivative $\partial_{i_1} \cdots \partial_{i_d}$.
But now
\begin{align*}
\sum_{i_1,\ldots,i_D=1}^d & |\partial_{i_1} \cdots \partial_{i_D} P(O z)|_{z=0}|^2 \\ & = \sum_{\substack{i_1,\ldots,i_D \\ j_1,\ldots,j_D \\ j'_1,\ldots,j'_D}=1}^d \left[ \prod_{\ell=1}^D O_{j_\ell i_\ell} O_{j'_\ell i_\ell} \right]  \left[ (\partial_{j_1} \cdots \partial_{j_{D}} P)(0) \right]  \left[ (\partial_{j'_1} \cdots \partial_{j'_{D}} P)(0) \right]; 
\end{align*}
summing over $i_1,\ldots,i_D$ first and using orthogonality of $O$ annihilates all terms except those where $j_\ell=j'_\ell$ for each $\ell=1,\ldots,D$, which is exactly the desired identity.
\end{proof}

From the standpoint of the analysis which will follow, the most important thing to know about the property of nondegeneracy is that it is stable as a function of $P$.

\begin{lemma}
Suppose $||\cdot||$ is the norm \eqref{stdnorm} on bilinear maps $P$ on $\R^p \times \R^q$ mapping into real polynomials of degree at most $D$ on $\R^d$. Then the map $P \mapsto ||| P|||_{\sigma}$
is continuous as a function of $P$ for each $\sigma \geq 0$. \label{continuity}
\end{lemma}
\begin{proof}
The infimum of any family of nonnegative continuous maps is always upper semicontinuous, so it suffices to show that the map $P \mapsto ||| P|||_{\sigma}$ is lower semicontinuous. Because $||P|| = ||\rho_{(O_1,O_2,O_3)} P||$ for any orthogonal matrices $O_1,O_2,O_3$, by the Singular Value Decomposition, if one defines the function
\[ F(P) := \inf_{D_1,D_2,D_3} |\det D_3|^{-\sigma} || \rho_{(D_1,D_2,D_3)} P|| \]
where $D_1,D_2,D_3$ are diagonal matrices with nonnegative entries and $\det D_1 = \det D_2 = 1$, it follows by the Singular Value Decomposition that
\begin{equation} |||P|||_{\sigma} = \inf_{O'_1,O'_2,O'_3} F(\rho_{(O'_1,O'_2,O'_3)} P)  \label{twostep} \end{equation}
where $O_1',O_2',O_3'$ are again orthogonal matrices. If $F$ is shown to be continuous, then \eqref{twostep} ensures that $P \mapsto |||P|||_{\sigma}$ will be as well by virtue of compactness of the orthogonal groups. This is because given any $a$ and $P$ such that $|||P|||_{\sigma} > a$, one can cover $\{P\} \times O(p) \times O(q) \times O(d)$ by open sets on which $F(\rho_{(O'_1,O'_2,O'_3)} P') > a$ for all $(P',O_1',O_2',O_3')$ in the open set, and the compactness implies that there is some neighborhood $U$ of $P$ and some $\epsilon > 0$ such that $F(\rho_{(O'_1,O'_2,O'_3)} P') > a + \epsilon$ for all $P' \in U$ and all orthogonal $O_1',O_2',O_3'$, which forces $|||P'|||_{\sigma} \geq a + \epsilon$ for all $P' \in U$.

After these reductions, it suffices to show that $F$ is lower semicontinuous. To that end, observe that
\begin{equation} || \rho_{(D_1,D_2,D_3)} P||^2 = \sum_{i=1}^p \sum_{j=1}^q \sum_{|\alpha|\leq D} \left[(D_1)_{ii} (D_2)_{jj} D_3^\alpha\right]^2 \frac{|\partial^\alpha P(\basis^i,\basis^j,z)|_{z=0}|^2}{\alpha!} . \label{issmall} \end{equation}
Here $(D_1)_{ii}$ is the $i$-th diagonal entry of $D_1$. The quantities $(D_2)_{jj}$ and $(D_3)_{\ell \ell}$ are defined similarly, and for any multiindex $\alpha := (\alpha_1,\ldots,\alpha_d)$, one defines
\[ D_3^{\alpha} := \prod_{\ell=1}^d \left[ (D_3)_{\ell \ell} \right]^{\alpha_\ell}. \]
If $P= 0$, there is nothing to prove since $F(P) > a$ implies that $a$ is negative and it is already known that $F > a$ holds everywhere when $a$ is negative. So let $c$ be the minimum of $(\alpha!)^{-1/2} |\partial^\alpha P(\basis^i,\basis^j,z)|_{z=0}|$ among those $(i,j,\alpha)$ for which this quantity is nonzero. If $||P-P'|| < \epsilon c$ for some $\epsilon \in (0,1)$, then
\[  (\alpha!)^{-1/2} |\partial^\alpha P'(\basis^i,\basis^j,z)|_{z=0}| \geq (1-\epsilon) (\alpha!)^{-1/2} |\partial^\alpha P(\basis^i,\basis^j,z)|_{z=0}|, \]
for each $\alpha,i$, and $j$ (since the inequality is trivially true when $\partial^{\alpha} P (\basis^i,\basis^j,z) = 0$). It follows for each such $P'$ that
$|| \rho_{(D_1,D_2,D_3)} P'|| \geq (1-\epsilon) || \rho_{(D_1,D_2,D_3)} P||$ for all nonnegative diagonal $D_1,D_2,D_3$.
But this means that for all $\epsilon \in (0,1)$, there is a neighborhood of $P$ on which $F(P') \geq (1-\epsilon) F(P)$ for all $P'$ in the neighborhood. Since we may assume without loss of generality that $F(P) > 0$, if $F(P) > a > 0$, then for any small $\epsilon$ such that $(1-\epsilon) F(P) > a$, there will be a neighborhood of $P$ on which $F > a$ as desired.
\end{proof}

A useful corollary is the following result, known in Geometric Invariant Theory as the Hilbert-Mumford criterion (see \cites{mfk1994,dc1970,sury2000}). In the language of analysis, it asserts that $|||P|||_{\sigma} = 0$ if and only if there is a rotation of the standard coordinate system and a one-parameter family of (presumably nonisotropic) rescalings of $P$ in the new coordinates which sends $P$ to $0$.
\begin{corollary} For any $P$ as above, $|||P|||_{\sigma} = 0$ if and only if there exist orthogonal matrices $O_1',O_2',O_3'$ and diagonal matrices $D_1',D_2',D_3'$ with $D_1'$ and $D_2'$ traceless such that 
\[ \lim_{t \rightarrow \infty} |\det e^{t D_3'} |^{-\sigma} \rho_{(e^{tD_1'},e^{t D_2'},e^{t D_3'})} \rho_{(O_1',O_2',O_3')} P = 0. \]
\end{corollary}
\begin{proof} For convenience, let $P_{ij}(z) := P(\basis^i,\basis^j,z)$ for each $i \in \{1,\ldots,p\}$ and $j \in \{1,\ldots,q\}$.
Since $F(\rho_{(O_1',O_2',O_3')} P)$ is continuous as a function of $O_1',O_2',O_3',$ and $P$, the infimum over $(O_1',O_2',O_3')$ is always attained and consequently $|||P|||_{\sigma} = 0$ implies that $F(\rho_{(O_1',O_2',O_3')} P) = 0$ for some $(O_1',O_2',O_3')$. Let $P' := \rho_{(O_1',O_2',O_3')}P$. Since $F(P') = 0$, there must be nonnegative diagonal matrices $D_1,D_2,D_3$ with $\det D_1 = \det D_2 = 1$ such that $|\det D_3|^{-\sigma} || \rho_{(D_1,D_2,D_3)} P'||$
is strictly less than the smallest positive value of $|\partial^{\alpha} P'_{ij}(0)| (\alpha!)^{-1/2}$ as $(i,j,\alpha)$ range over all possible elements of $\{1,\ldots,p\} \times \{1,\ldots,q\} \times \set{\alpha}{|\alpha| \leq D}$. It follows from the identity \eqref{issmall} that
\[ |\det D_3|^{-\sigma} (D_1)_{ii} (D_2)_{jj} D_3^{\alpha}  < 1 \]
for all $(i,j,\alpha)$ such that $|\partial^{\alpha} P_{ij}(0)| \neq 0$. Now let $D_1',D_2',D_3'$ be chosen so that $e^{t D_k'} = D_k$ when $t=1$ for each $k=1,2,3$. For all $t > 0$,
\[ |\det e^{tD_3'}|^{-\sigma} (e^{tD_1'})_{ii} (e^{tD_2'})_{jj} (e^{t D_3'})^{\alpha} = \left[|\det D_3|^{-\sigma} (D_1)_{ii} (D_2)_{jj} D_3^{\alpha} \right]^{t} \rightarrow 0 \]
as $t \rightarrow \infty$ when $\partial^{\alpha}P_{ij}(0) \neq 0$, so all coordinates of $\rho_{(e^{tD_1'},e^{tD_2'},e^{tD_3'})} P'$ tend to zero as $t \rightarrow \infty$.

The reverse direction is immediate from \eqref{twostep} since there will by assumption exist $O_1',O_2',O_3'$ such that $F(\rho_{(O_1',O_2',O_3')} P) = 0$.
\end{proof}

The next corollary presents a somewhat different way of understanding degeneracy. The construction is reminiscent of the Newton polytope and is similar to the condition appearing in \cite{gressmancurvature}.

\begin{corollary}
For any $P$ as above, let $\mathcal{N}(P)$ be the convex hull of all points $(\basis^i;\basis^j;\alpha) \in \R^{p} \times \R^{q} \times \R^{d}$ for which $\partial^{\alpha}_zP (\basis^i,\basis^j,z)|_{z=0}$ is nonzero. Then \label{newtoncor}
$|||P|||_\sigma = 0$ if and only if there exist orthogonal matrices $O_1',O_2',O_3'$ such that
\[ \left( p^{-1} \mathbf{1}_p; q^{-1} \mathbf{1}_q; \sigma \mathbf{1}_d \right) \not \in \mathcal{N}(\rho_{(O_1',O_2',O_3')}P). \]
\end{corollary}
\begin{proof}
Let $P' := \rho_{(O'_1,O_2',O_3')} P$ and suppose that  there exist $O_1',O_2',O_3'$ such that $\left( p^{-1} \mathbf{1}_p; q^{-1} \mathbf{1}_q; \sigma \mathbf{1}_d \right) \not \in \mathcal{N}(P')$.  The Separating Hyperplane Theorem guarantees the existence of $(w_1;w_2;w_3) \in \R^{p} \times \R^{q} \times \R^{d}$ and some real constant $c$ such that
\[ c < (w_1;w_2;w_3) \cdot \left( p^{-1} \mathbf{1}_p; q^{-1} \mathbf{1}_q; \sigma \mathbf{1}_d \right) \]
and
\[  (w_1;w_2;w_3) \cdot (\basis^i;\basis^j;\alpha) < c \]
whenever $\partial^{\alpha}P'_{ij}(0) \neq 0$. Because $p^{-1} \mathbf{1}_p \cdot \mathbf{1}_p = \basis^i \cdot \mathbf{1}_p = 1$ for any $i$, when one subtracts a multiple of $\mathbf{1}_p$ from $w_1$, the quantities $w_1 \cdot p^{-1} \mathbf{1}_p$ and $w_1 \cdot \basis^i$ decrease by the same amount. Thus one may assume without loss of generality (changing the value of $c$ if necessary) that both $w_1$ and $w_2$ have entries which sum to zero, i.e., $w_1 \cdot \mathbf{1}_p = w_2 \cdot \mathbf{1}_q = 0$. In this case, it follows that
\begin{equation} (w_1;w_2;w_3) \cdot (\basis^i;\basis^j;\alpha) < c < \sigma w_3 \cdot \mathbf{1}_d \label{meansneg} \end{equation}
whenever $\partial^{\alpha} P'_{ij}(0) \neq 0$. Now let $D_1',D_2',D_3'$ be diagonal matrices whose entries are given by $w_1,w_2,$ and $w_3$, respectively. Using this definition, one has the identity
\begin{align*}
|\det e^{tD_3'}|^{-\sigma} (e^{tD_1'})_{ii} (e^{tD_2'})_{jj} (e^{t D_3'})^{\alpha} = e^{t (\basis^i; \basis^j; \alpha - \sigma \mathbf{1}_d) \cdot (w_1;w_2;w_3)}
\end{align*}
for each $i$, $j$, and $\alpha$. The inequality \eqref{meansneg} implies that this exponential decays to zero whenever $\partial^{\alpha} P'_{ij}(0) \neq 0$. Thus by the previous corollary, $|||P|||_\sigma = 0$.

For the converse, if $D_1',D_2',D_3'$ as described in the previous corollary exist so that $|\det e^{t D_3'}|^{-\sigma} \rho_{(e^{tD_1'},e^{tD_2'},e^{tD_3'})} P' \rightarrow 0$ as $t \rightarrow \infty$, letting $w_1,w_2,$ and $w_3$ be vectors whose coordinates agree with the diagonal entries of $D_1, D_2,$ and $D_3$, respectively, similarly implies that
\[ (w_1;w_2;w_3) \cdot \left( \basis^i; \basis^j; \alpha - \sigma \mathbf{1}_d \right) < 0 \]
whenever $\partial^{\alpha} P'_{ij}(0) \neq 0$, meaning that
\begin{equation} (w_1;w_2;w_3) \cdot \left( \basis^i; \basis^j; \alpha  \right) <  (w_1;w_2;w_3) \cdot \left( p^{-1} \mathbf{1}_p; q^{-1} \mathbf{1}_q; \sigma \mathbf{1}_d \right) \label{almostdone} \end{equation}
because $D_1'$ and $D_2'$ are traceless. If $\left( p^{-1} \mathbf{1}_p; q^{-1} \mathbf{1}_q; \sigma \mathbf{1}_d \right)$ happened to belong to $\mathcal{N}(P')$, one could take an appropriate convex linear combination of the inequalities \eqref{almostdone} to deduce that
\[ (w_1;w_2;w_3) \cdot \left( p^{-1} \mathbf{1}_p; q^{-1} \mathbf{1}_q; \sigma \mathbf{1}_d \right) < (w_1;w_2;w_3) \cdot \left( p^{-1} \mathbf{1}_p; q^{-1} \mathbf{1}_q; \sigma \mathbf{1}_d \right), \]
which is clearly impossible.
\end{proof}

From a computational standpoint, the practical work of determining whether or not $|||P|||_{\sigma}$ is zero for a given $P$ and $\sigma$ is difficult and there are few known shortcuts. The proposition below establishes perhaps one of the most simplest criteria by which it may be deduced that $|||P|||_{\sigma}$ is strictly positive.
\begin{proposition}
Suppose that $P(z)$ is a $p \times q$ matrix whose entries are polynomials of degree at most $D$ in the variable $z \in \R^{d}$. If $P$ satisfies the equations
\begin{align}
\sum_{j=1}^q \sum_{|\alpha|\leq D}  \frac{1}{\alpha!} \left[ \partial^\alpha P_{i j}(z) \partial^\alpha P_{i'j} (z)|_{z=0} \right] = p^{-1} ||P||^2 \delta_{ii'}, \label{foc1} \\
\sum_{i=1}^p \sum_{|\alpha|\leq D} \frac{1}{\alpha!} \left[ \partial^\alpha P_{ij}(z)  \partial^\alpha P_{ij'}(z)|_{z=0} \right] = q^{-1} ||P||^2 \delta_{jj'}, \label{foc2} \\
\sum_{i=1}^p \sum_{j=1}^q \sum_{\substack{|\alpha|,|\alpha'| \leq D\\ \alpha + \basis^{k'} = \alpha' + \basis^k}} \sqrt{ \frac{\alpha_{k} \alpha'_{k'}}{\alpha! \alpha'!}} \left[ \partial^\alpha P_{ij} (z) \partial^{\alpha'} P_{ij} (z) |_{z=0} \right]  = \sigma ||P||^2 \delta_{kk'}, \label{foc3}
\end{align}
where $\delta$ is the Kronecker $\delta$, then $|||P|||_\sigma = ||P||$. Conversely, if the infimum \eqref{gitnorm} is attained for some triple $(A,B,C)$, then the matrix $\tilde P := |\det C|^{-\sigma} \rho_{(A,B,C)} P$ satisfies \eqref{foc1}--\eqref{foc3}. \label{critprop}
\end{proposition}
\begin{proof}
By the Singular Value Decomposition, every matrix $A$ admits a decomposition of the form $O_1' e^{D_1} O_1$, where $O_1$ and $O_1'$ are orthogonal and $D_1$ is a diagonal matrix. If $|\det A| = 1$, the matrix $D_1$ must furthermore be traceless so that $|\det O_1' e^{D_1} O_1| = 1$ as well. Since $||\rho_{(O_1,O_2,O_3)} P|| = ||P||$, it follows that
\begin{align*} |\det C|^{-2 \sigma} &  || \rho_{(A,B,C)} P ||^2  = e^{-2 \sigma D_3 \cdot \mathbf{1}_d} || \rho_{(O_1',O_2',O_3')} \rho_{(e^{D_1},e^{D_2},e^{D_3})} \rho_{(O_1,O_2,O_3)} P||^2 \\
 & = e^{-2 t \sigma D_3 \cdot \mathbf{1}_d}\left. || \rho_{(e^{tD_1},e^{tD_2},e^{tD_3})} \rho_{(O_1,O_2,O_3)} P||^2 \right|_{t=1} \\
 & = \left. \sum_{i,j,\alpha} \frac{ \left| e^{t (D_{1i} + D_{2j} + D_3 \cdot (\alpha-\sigma \mathbf{1}_d))}\partial^\alpha [\rho_{(O_1,O_2,O_3)} P]_{ij}(z) |_{z=0}\right|^2}{\alpha!} \right|_{t=1},
 \end{align*}
where $D_{ki}$ denotes the $i$-th diagonal element of $D_k$ and $D_3 \cdot (\alpha-\sigma \mathbf{1}_d) = \sum_{j} D_{3j} (\alpha_j - \sigma)$. To simplify notation, let $m_{ij\alpha}(t) := e^{t (D_{1i} + D_{2j} + D_3 \cdot (\alpha-\sigma \mathbf{1}_d))}$ and let $\tilde P_{ij}(z) := [\rho_{(O_1,O_2,O_3)}P]_{ij}(z)$. Now
\begin{align*}
\frac{d}{dt} & \sum_{i,j,\alpha} \frac{(m_{ij\alpha}(t))^2}{\alpha!} | \partial^{\alpha} \tilde P_{ij}(z) |_{z=0}|^2 = \\
& \sum_{i,j,\alpha}  \frac{2(D_{1i} + D_{2j} + D_3 \cdot (\alpha-\sigma \mathbf{1}_d)) m_{ij\alpha}(2t)}{\alpha!} | \partial^{\alpha}  \tilde P_{ij}(z) |_{z=0}|^2, \\
\frac{d^2}{dt^2} &\sum_{i,j,\alpha} \frac{(m_{ij\alpha}(t))^2}{\alpha!} | \partial^{\alpha} \tilde P_{ij}(z) |_{z=0}|^2 =  \\
& \sum_{i,j,\alpha} \frac{4(D_{1i} + D_{2j} + D_3 \cdot (\alpha-\sigma \mathbf{1}_d))^2 m_{ij\alpha}(2t)}{\alpha!} | \partial^{\alpha}  \tilde P_{ij}(z) |_{z=0}|^2  \geq 0.
\end{align*}
An immediate consequence is that the function 
\[ |\det e^{t D_3}|^{-2 \sigma} ||\rho_{(e^{t D_1},e^{tD_2},e^{tD_3})} \rho_{(O_1,O_2,O_3)} P||^2 \]
 is a convex function of $t$ and therefore any critical point is necessarily a global minimum. If $P$ happens to have the property that
\begin{equation} \sum_{i,j,\alpha}  \frac{(D_{1i} + D_{2j} + D_3 \cdot (\alpha-\sigma \mathbf{1}_d)}{\alpha!} | \partial^{\alpha} [\rho_{(O_1,O_2,O_3)} P]_{ij}(z) |_{z=0}|^2 = 0 \label{cpeq} \end{equation}
for all traceless diagonal $D_1,D_2$, all diagonal $D_3$, and all orthogonal $O_1,O_2,O_3$, then one will have that the function $|\det C|^{-\sigma} ||\rho_{(A,B,C)} P||$ attains a minimum at the identity and consequently $|||P|||_{\sigma} = ||P||$. Likewise, if the minimum \eqref{gitnorm} is attained for some triple $(A,B,C)$, then the matrix $|\det C|^{-\sigma} \rho_{(A,B,C)} P$ has a critical point at the identity (because the infimum \eqref{gitnorm} applied to this matrix will attain its infimum at the identity). 

The equation \eqref{cpeq} is equivalent to the system
\begin{align}
\sum_{i,j,\alpha} \frac{\basis^i-p^{-1} \mathbf{1}_p}{\alpha!}  | \partial^{\alpha} [\rho_{(O_1,O_2,O_3)} P]_{ij}(z) |_{z=0}|^2 = 0,\label{cpeq1} \\
\sum_{i,j,\alpha} \frac{\basis^j - q^{-1} \mathbf{1}_q }{\alpha!} | \partial^{\alpha} [\rho_{(O_1,O_2,O_3)} P]_{ij}(z) |_{z=0}|^2 = 0,\label{cpeq2} \\
\sum_{i,j,\alpha} \frac{\alpha - \sigma \mathbf{1}_d}{\alpha!} | \partial^{\alpha} [\rho_{(O_1,O_2,O_3)} P]_{ij}(z) |_{z=0}|^2 = 0, \label{cpeq3} 
\end{align}
where the $\basis^{i}$ and $\basis^j$ are standard basis vectors. Equivalence follows from the fact that the left-hand side of \eqref{cpeq} is simply a sum of dot products: the dot product of diagonal elements of $D_1$ with \eqref{cpeq1}, the dot product of diagonal elements of $D_2$ with \eqref{cpeq2}, and the dot product of diagonal elements of $D_3$ with \eqref{cpeq3}. For this reason, \eqref{cpeq1}--\eqref{cpeq3} immediately imply \eqref{cpeq}. On the other hand, if any one of \eqref{cpeq1}--\eqref{cpeq3} failed, a suitable diagonal matrix could be constructed to violate \eqref{cpeq} as well (simply let the corresponding $D_i$ have entries on the diagonal which give a nonzero value when dotted with any one of \eqref{cpeq1}--\eqref{cpeq3} which happens not to be zero; since every vector $\basis^i - p^{-1} \mathbf{1}_p$ in \eqref{cpeq1} and $\basis^j - q^{-1} \mathbf{1}_q$ in \eqref{cpeq2} has entries which sum to zero, subtracting off a suitable constant from each diagonal entry of $D$ can ensure that it is traceless without changing the fact that its dot product with either \eqref{cpeq1} or \eqref{cpeq2} is nonzero).

The system \eqref{cpeq1}--\eqref{cpeq3} further simplifies upon observing that the left-hand side of \eqref{cpeq1} is independent of $O_2$ and $O_3$, and likewise \eqref{cpeq2} depends only on $O_2$ and \eqref{cpeq3} depends only on $O_3$. But then
\begin{align} \sum_{i,j,\alpha} & \frac{\basis^i - p^{-1} \mathbf{1}_p}{\alpha!} \left| \partial^{\alpha} P(O_1^{T} \basis^i,\basis^j,z)|_{z=0}\right|^2 \nonumber \\ 
= & \sum_{i,i',i''} (\basis^i-p^{-1} \mathbf{1}_p) (O_1)_{ii'} (O_1)_{ii''}\sum_{j,\alpha} \frac{ \partial^{\alpha} P_{i'j}( z)  \partial^{\alpha} P_{i''j} (z) |_{z=0} }{\alpha!}. \label{diageigen} 
\end{align}
Certainly \eqref{diageigen} will be zero for all orthogonal $O_1$ if
\begin{equation} \sum_{j,\alpha}  \frac{1}{\alpha!} \partial^{\alpha} P_{i'j}(z)  \partial^{\alpha} P_{i''j}(z)|_{z=0} = \frac{||P||^2}{p} \delta_{i'i''},\label{diagonalize} \end{equation}
and if this condition does not hold, since the left-hand side of \eqref{diagonalize} is self-adjoint when regarded as a matrix indexed by pairs $(i',i'')$, it will always be possible to choose a matrix $O_1$ diagonalizing it so that for this $O_1$, \eqref{diageigen} equals
\[ \sum_{i} (\basis^i - p^{-1}\mathbf{1}_p) \lambda_i = (\lambda_1,\ldots,\lambda_p) - \overline{\lambda} \mathbf{1}_p \neq 0 \]
where $\lambda_1,\ldots,\lambda_p$ are the eigenvalues of the left-hand side of \eqref{diagonalize} and $\overline{\lambda}$ is the average of the eigenvalues. In particular, note that the sum of the eigenvalues must always equal $||P||^2$, so indeed \eqref{diageigen} is nonzero for this $O_1$ whenever the eigenvalues are not all equal to $p^{-1} ||P||^2$.
The same argument establishes that \eqref{cpeq2} holds for all $O_1,O_2,O_3$ exactly when
\begin{equation}
\sum_{i,\alpha} \frac{1}{\alpha!} \partial^\alpha P_{ij'}(0) \partial^\alpha P_{ij''}(0) = \frac{||P||^2}{q} \delta_{j'j''}. \label{cpeq15}
\end{equation}

Understanding the condition \eqref{cpeq3} requires a bit more work. The first step is essentially to repeat the argument of Proposition \ref{proportho} to express the left-hand side of \eqref{cpeq3} in a way that reduces and clarifies its dependence on $O_3$. First observe that
\begin{align*}
\sum_{i,j,|\alpha|=m} & \frac{\alpha}{\alpha!} |\partial^\alpha P_{ij} (O_3^T z)|_{z=0}|^2  \\
 & = \frac{1}{m!} \sum_{i,j,k_1,\ldots,k_m} (\basis^{k_1} + \cdots + \basis^{k_m}) |\partial_{k_1} \cdots \partial_{k_m} P_{ij}(O_3^T z)|_{z=0}|^2 \\
& = \frac{1}{(m-1)!} \sum_{i,j,k_1,\ldots,k_m} \basis^{k_1}  |\partial_{k_1} \cdots \partial_{k_m} P_{ij}(O_3^T z)|_{z=0}|^2
\end{align*}
because there are exactly $m!/\alpha!$ tuples $(k_1,\ldots,k_m)$ such that $\partial^{m}_{k_1\cdots k_m} = \partial^{\alpha}$, and for each such tuple, $\alpha = \basis^{k_1} + \cdots + \basis^{k_m}$. By symmetry of this sum in the parameters $k_1,\ldots,k_m$, one can retain only the $\basis^{k_1}$ term of the sum at the cost of multiplying the entire expression by a factor of $m$. Continuing along the same lines of Proposition \ref{proportho} (where we write $O_3$ as just $O$ to simplify cumbersome notation),
\begin{align*}
\frac{1}{(m-1)!} & \sum_{i,j,k_1,\ldots,k_m} \basis^{k_1}  |\partial_{k_1} \cdots \partial_{k_m} P_{ij}(O^T z)|_{z=0}|^2 \\
= &\frac{1}{(m-1)!} \sum_{k_1,k_1',k_1''} \basis^{k_1} O_{k_1 k_1'} O_{k_1 k_1''} \\
&   \qquad \qquad \cdot \sum_{\substack{i,j,k_2,\ldots,k_m\\k_2',\ldots,k_m'\\k_2'',\ldots,k_m''}} \left[ \prod_{\ell=2}^m O_{k_\ell k_\ell''} O_{k_\ell,k_\ell'} \right] \partial^m_{k_1' \cdots k_m'} P_{ij}(z) \partial^m_{k_1'' \cdots k_m''} P_{ij}(z) |_{z=0} \\
 = & \frac{1}{(m-1)!} \sum_{k_1,k_1',k_1''} \basis^{k_1} O_{k_1 k_1'} O_{k_1 k_1''}  \\
& \qquad \qquad \cdot \sum_{i,j,k_2,\ldots,k_m} \partial_{k_1'} \partial^{m-1}_{k_2 \cdots k_m} P_{ij}(z) \partial_{k_1''} \partial^{m-1}_{k_2 \cdots k_m} P_{ij}(z) |_{z=0} \\
= & \sum_{k_1,k_1',k_1''} \basis^{k_1} O_{k_1 k_1'} O_{k_1 k_1''} \sum_{i,j,|\beta|={m-1}} \frac{\partial_{k_1'} \partial^{\beta} P_{ij}(z) \partial_{k_1''} \partial^{\beta} P_{ij}(z) |_{z=0}}{\beta!}.
\end{align*}
By the same reasoning used for \eqref{cpeq1} and \eqref{cpeq2}, the condition \eqref{cpeq3} will hold for all $O_1,O_2,O_3$ precisely when
\[  \sum_{i,j,|\beta|\leq{D-1}} \frac{\partial_{k_1'} \partial^{\beta} P_{ij}(z) \partial_{k_1''} \partial^{\beta} P_{ij}(z) |_{z=0}}{\beta!} = \sigma ||P||^2 \delta_{k_1'k_1''}.\]
Now the multiindices $\beta$ with $|\beta| = m-1$ are in bijection with those multiindices $\alpha$ with $|\alpha| = m$ and $\alpha_{k_1'} > 0$ (since we may map $\alpha$ to $\alpha - \basis^{k_1}$ and the map is clearly a bijection). Since $1/\beta! = \alpha_{k_1'} / \alpha!$ in this correspondence, it follows that
\begin{equation*}
\begin{split}\sum_{\substack{i,j\\|\beta| \leq D-1}}& \frac{\left[ \partial_{k_1'} \partial^\beta P_{ij}(z) \right] \left[  \partial_{k_1''} \partial^\beta P_{ij} (z)\right]|_{z=0}}{\beta!}
\\ & = \sum_{i,j, \alpha,\alpha'} \sqrt{ \frac{\alpha_{k_1'} \alpha'_{k_1''}}{\alpha! \alpha'!}} \partial^\alpha P_{ij} (0) \partial^{\alpha'} P_{ij} (0) \delta_{\alpha + \basis^{k_2''} = \alpha' + \basis^{k_1'}}, 
\end{split}
\end{equation*}
which establishes the final condition \eqref{foc3}. (Note that the final sum is taken over all $|\alpha| \leq D$ since any $\alpha$ with $\alpha_{k'_1} = 0$ or $\alpha_{k_1''} = 0$ do not contribute to the sum.)
\end{proof}

We come now to the main and final result of this section, which gives a simple procedure to determine for sufficiently sparse matrices $P$ when $|||P|||_{\sigma} > 0$.
\begin{lemma}
Suppose that $P$ is a $p \times q$ matrix whose entries are real polynomials of degree at most $D$ in the variable $z \in \R^d$. Suppose also that $P$ has the following properties: \label{standardform}
\begin{enumerate}
\item For each multiindex $\alpha$ with $|\alpha| \leq D$, the matrix $\partial^\alpha P(0)$ has at most one nonzero entry in each row and each column.
\item For each pair $(i,j)$, let $E_{ij}$ be the collection of those multiindices $\alpha$ such that $\partial^\alpha P_{ij}(0) \neq 0$. Then for all $\alpha,\alpha' \in E_{ij}$, if there exist $k, k' \in \{1,\ldots,d\}$ such that $\alpha + \basis^{k'} = \alpha' + \basis^k$, it must be the case that $\alpha = \alpha'$.
\item Let $E$ be the set of $(i,j,\alpha)$ such that $\partial^{\alpha} P_{ij}(0) \neq 0$. For each $(i,j,\alpha) \in E$,  there is a nonnegative real number $\theta_{ij\alpha}$ such that $\sum_{(i,j,\alpha) \in E} \theta_{ij\alpha} = 1$ and 
\begin{equation} \sum_{(i,j,\alpha) \in E} \theta_{ij\alpha} (\basis^i;\basis^j;\alpha) = \left( p^{-1} \mathbf{1}_p; q^{-1} \mathbf{1}_q; \sigma \mathbf{1}_d \right).\label{convex} \end{equation}
\end{enumerate}
Then 
\[ |||P|||_\sigma = \inf_{A  \in \SL(p,\R)} \inf_{B \in \SL(q,\R)} \inf_{C \in \GL(d,\R)} |\det C|^{-\sigma} ||\rho_{(A,B,C)} P|| > 0. \] 
If all $\theta_{ij\alpha}$ are all strictly positive, then the infimum is attained for some diagonal matrices $A,B,$ and $C$. Note that if \eqref{convex} fails to hold for any suitable choices of $\theta_{ij\alpha}$, then $|||P|||_{\sigma} = 0$ by Corollary \ref{newtoncor}.
\end{lemma}
\begin{proof}
Consider the vector space $W$ of vectors $w := (w_p;w_q;w_d) \in \R^p \times \R^q \times \R^d$ such that $w_p \cdot \mathbf{1}_p = w_q \cdot \mathbf{1}_q = 0$. Let $w^1,\ldots,w^N$ be any maximal collection of linearly independent vectors in $W$ with the property that the span of $w^1,\ldots,w^N$ contains no nontrivial vectors $w = (w_p;w_q;w_d)$ such that $w_p \cdot \basis^i + w_q \cdot \basis^j + w_d \cdot (\alpha-\sigma \mathbf{1}_d) = 0$ for all $(i,j,\alpha) \in E$. Let the span of $w^1,\ldots,w^N$ be called $\tilde W$. If $N=0$, i.e., if every $w \in W$ satisfies $w \cdot (\basis^i;\basis^j; \alpha - \sigma \mathbf{1}_d) = 0$ for all $(i,j,\alpha) \in E$, then let $\tilde W := \{0\}$. When $\tilde W$ is nontrivial, every $w \in W$ admits a unique $\tilde w \in \tilde W$  such that $w - \tilde w$ vanishes when dotted with $(\basis^i;\basis^j;\alpha-\sigma \mathbf{1}_d)$ for each $(i,j,\alpha) \in E$. This implies that 
\begin{equation} \sum_{(i,j,\alpha) \in E} \! e^{2 w \cdot (\basis^i;\basis^j;\alpha-\sigma \mathbf{1}_d)} \frac{|\partial^\alpha P_{ij}(0)|^2}{\alpha!} = \! \! \sum_{(i,j,\alpha) \in E} \! e^{2 \tilde{w} \cdot (\basis^i;\basis^j;\alpha-\sigma \mathbf{1}_d)} \frac{|\partial^\alpha P_{ij}(0)|^2}{\alpha!} \label{genscale} \end{equation}
for each $w \in W$. A consequence is that if the right-hand side of \eqref{genscale} attains a minimum as $\tilde w$ ranges over $\tilde W$, then the left-hand side must also attain a minimum as $w$ ranges over $W$. If $\tilde W = \{0\}$, then the left-hand side of \eqref{genscale} is simply independent of $w$ and consequently it trivially attains a minimum for any $w \in W$.

We first establish a fundamental sufficiency criterion: if the left-hand side of \eqref{genscale} attains a global minimum as a function of $w \in W$, then it must be the case that $|||P|||_\sigma > 0$ with the infimum in the definition of $|||P|||_\sigma$ being attained for some diagonal matrices $A, B$, and $C$. To see this, Let $w$ denote any vector at which the minimum is attained. This $w$ splits into a triple $(w_p;w_q;w_d)$, and differentiating any one of $w_p$ or $w_q$ in the direction $\basis^{k} - \basis^{k'}$ for $k \neq k'$ and differentiating $w_d$ in any direction $\basis^k$ gives that
\begin{align*} 0 & = \sum_{(i,j,\alpha) \in E} \basis^i \cdot( \basis^k - \basis^{k'}) e^{2 w \cdot (\basis^i;\basis^j;\alpha-\sigma \mathbf{1}_d)} \frac{|\partial^\alpha P_{ij}(0)|^2}{\alpha!} \\ & = \sum_{(i,j,\alpha) \in E} \basis^j  \cdot( \basis^k - \basis^{k'})e^{2 w \cdot (\basis^i;\basis^j;\alpha-\sigma \mathbf{1}_d)} \frac{|\partial^\alpha P_{ij}(0)|^2}{\alpha!} \\ & = \sum_{(i,j,\alpha) \in E} (\alpha-\sigma \mathbf{1}_d) \cdot \basis^k e^{2 w \cdot (\basis^i;\basis^j;\alpha-\sigma \mathbf{1}_d)} \frac{|\partial^\alpha P_{ij}(0)|^2}{\alpha!}. \end{align*}
Let 
\[ \tilde P_{ij}(z_1,\ldots,z_d) := e^{- \sigma w_d \cdot \mathbf{1}_d} e^{(w_p)_i} e^{(w_q)_j} P_{ij} (e^{(w_d)_1} z_1,\ldots,e^{(w_d)_d} z_d). \]
(Here $(v)_i$ indicates the $i$-th coordinate of the vector $v$ in the standard basis.)
This $\tilde P$ may be understood to exactly equal $|\det D_3|^{-\sigma} \rho_{(D_1,D_2,D_3)} P$, where $D_1,D_2,D_3$ are diagonal matrices whose entries are the exponentials of the coordinates of $w_p,w_q,$ and $w_d$, respectively. The critical point equations just established require that there must exist real $\lambda_1$ and $\lambda_2$ such that
\[ \sum_{j,\alpha} \frac{|\partial^\alpha \tilde P_{ij}(z)|^2}{\alpha!} = \lambda_1 \text{ for each } i \in \{1,\ldots,p\}, \]
\[ \sum_{i,\alpha} \frac{|\partial^\alpha \tilde P_{ij}(z)|^2}{\alpha!} = \lambda_2 \text{ for each } j \in \{1,\ldots,q\}, \text{ and }\]
\[ \sum_{i,j,\alpha} \alpha \frac{|\partial^\alpha \tilde P_{ij}(0)|^2}{\alpha!} = \sigma ||\tilde P||^2 \mathbf{1}_d. \]
By summing, it is easy to check that $\lambda_1$ must equal $p^{-1} ||\tilde P||^2$ and $\lambda_2$ must equal $q^{-1} ||\tilde P||^2$.
But the hypotheses of this lemma guarantee that
\[ \sum_{j,\alpha} \frac{1}{\alpha!} [ \partial^\alpha \tilde P_{ij}(0) ] [ \partial^{\alpha} \tilde P_{i'j}(0) ] = 0 \]
when $i \neq i'$ because each column $j$ of each matrix $\partial^\alpha P(0)$ admits at most one row $i$ in which it is nonzero. Likewise
\[ \sum_{i,\alpha} \frac{1}{\alpha!} [ \partial^\alpha \tilde P_{ij}(0)] [ \partial^{\alpha} \tilde P_{ij'}(0) ] = 0 \]
when $j \neq j'$. Lastly, $[\partial^\alpha \tilde P_{ij}(0)] [\partial^{\alpha'} \tilde P_{ij}(0)] \delta_{\alpha + \basis^{k'} = \alpha' + \basis^k}$ is identically zero as a function of $i$ and $j$ when $\alpha \neq \alpha'$. Thus the equations \eqref{foc1}--\eqref{foc3} are satisfied for $\tilde P$ and $|||P|||_\sigma = ||| \tilde P|||_\sigma = ||\tilde P|| > 0$. Since $\tilde P = |\det C|^{-\sigma} \rho_{(A,B,C)} P$ for diagonal $A$, $B$, and $C$, the infimum in the definition of $|||P|||_\sigma$ must be attained for some diagonal matrices $A$, $B$, and $C$.

By the above argument, if $\tilde W$ is trivial, then the lemma follows because the sum in \eqref{genscale} trivially attains a minimum. It may consequently be assumed henceforth that $\tilde W$ is not trivial. The proof will proceed by induction on the number of indices $(i,j,\alpha) \in E$ for which $\theta_{ij\alpha} = 0$.

In the case of strict positivity (i.e., $\theta_{ij\alpha} > 0$ for all $(i,j,\alpha) \in E$), fix any norm $|| \cdot ||$ on $\tilde W$; we will show that there must be some positive $c$ such that $\max_{(i,j,\alpha) \in E} \tilde w \cdot (\basis^i;\basis^j;\alpha-\sigma \mathbf{1}_d) \geq c ||\tilde w||$ for all $\tilde w \in \tilde W$. This is a standard compactness argument: by rescaling, one may assume that $||\tilde w|| = 1$. If there were no such $c$, there would be a sequence $\{\tilde w^n\}_{n=1}^\infty$ on the unit sphere of the norm such that $\max_{(i,j,\alpha) \in E} \tilde w^n \cdot (\basis^i;\basis^j;\alpha-\sigma \mathbf{1}_d)  \leq n^{-1}$. Because $E$ is finite and the unit sphere is compact, there must be a convergent subsequence with limit $\tilde w$ in the unit sphere such that $\max_{(i,j,\alpha) \in E} \tilde w \cdot (\basis^i;\basis^j;\alpha-\sigma \mathbf{1}_d) \leq 0$. But
\begin{equation} \sum_{(i,j,\alpha) \in E} \theta_{ij\alpha} \tilde w \cdot (\basis^i;\basis^j;\alpha-\sigma \mathbf{1}_d) = \tilde w \cdot \left( p^{-1} \mathbf{1}_p; q^{-1} \mathbf{1}_q; (\sigma - \sigma) \mathbf{1}_d \right) = 0 \label{thesum} \end{equation}
by \eqref{convex}.
Since all $\theta_{ij\alpha}$ are strictly positive and each term of the sum \eqref{thesum} is nonpositive, each term must actually equal zero (since if any of the terms were strictly negative, the sum \eqref{thesum} would have to be strictly negative). But this means that $\tilde w$ vanishes when dotted with $(\basis^i;\basis^j;\alpha-\sigma \mathbf{1}_d)$ for each $(i,j,\alpha) \in E$, which is not possible because $\tilde w$ is a nonzero vector in $\tilde W$.

Supposing in general that $\max_{(i,j,\alpha) \in E} \tilde w \cdot (\basis^i;\basis^j;\alpha-\sigma \mathbf{1}_d) \geq c ||\tilde w||$ for all $\tilde w \in \tilde W$ (i.e., regardless of whether one is in the strict positivity situation or not), it will follow that the left-hand side of \eqref{genscale} attains a minimum and consequently that $|||P|||_{\sigma} > 0$ with its infimum being attained for diagonal $A, B$, and $C$. For any $\tilde w \in \tilde W$, one may bound the right-hand side of \eqref{genscale} below by any single term which maximizes the exponential $e^{2 \tilde w \cdot (\basis^i;\basis^j;\alpha-\sigma \mathbf{1}_d)}$. The result is that
\begin{align*} \sum_{(i,j,\alpha) \in E} e^{2 \tilde{w} \cdot (\basis^i;\basis^j;\alpha-\sigma \mathbf{1}_d)} & \frac{|\partial^\alpha P_{ij}(0)|^2}{\alpha!} \\ & \geq \min_{(i,j,\alpha) \in E} \frac{|\partial^\alpha P_{ij}(0)|^2}{\alpha!} \max_{(i,j,\alpha) \in E} e^{2 \tilde w \cdot (\basis^i;\basis^j;\alpha-\sigma \mathbf{1}_d)} \\ & \geq \left(  \min_{(i,j,\alpha) \in E} \frac{|\partial^\alpha P_{ij}(0)|^2}{\alpha!}  \right) e^{2c ||\tilde w||}. \end{align*}
Because the right-hand side is unbounded as $||\tilde w|| \rightarrow \infty$, there is some positive radius $R$ such that the right-hand side is at least $2||P||^2$ when $||\tilde w|| > R$. As a consequence,
\[  \sum_{(i,j,\alpha) \in E} e^{2 \tilde{w} \cdot (\basis^i;\basis^j;\alpha-\sigma \mathbf{1}_d)} \frac{|\partial^\alpha P_{ij}(0)|^2}{\alpha!} \geq 2||P||^2\]
when $||\tilde w|| \geq R$. The set $||\tilde w|| \leq R$, on the other hand, is compact, and so the function
\[  \sum_{(i,j,\alpha) \in E} e^{2 \tilde{w} \cdot (\basis^i;\basis^j;\alpha-\sigma \mathbf{1}_d)}  \frac{|\partial^\alpha P_{ij}(0)|^2}{\alpha!} \]
attains a minimum on the ball $||\tilde w|| \leq R$. Because the value of that minimum is no greater than $||P||^2$, the minimum on $||\tilde w|| \leq R$ must be a global minimum (as the sum is at least $2||P||^2$ everywhere else). Therefore
\[ \sum_{(i,j,\alpha) \in E} e^{2 w \cdot (\basis^i;\basis^j;\alpha-\sigma \mathbf{1}_d)} \frac{|\partial^\alpha P_{ij}(0)|^2}{\alpha!} \]
attains a global minimum as $w$ ranges over all of $W$. As has already been demonstrated, this implies that the minimum $|||P|||_\sigma$ is attained for diagonal $A$, $B$, and $C$ and is strictly positive.

Since the lemma is now known to hold in the case of strict positivity, suppose that strict positivity of the $\theta_{ij\alpha}$ fails. Thanks to prior reasoning, it suffices to assume that $\tilde W$ is nontrivial and that there is no positive $c$ such that $\max_{(i,j,\alpha) \in E} \tilde w \cdot (\basis^i;\basis^j;\alpha-\sigma \mathbf{1}_d) \geq c ||\tilde w||$ for all $\tilde w \in \tilde W$. Thus it may be assumed that there exists a $\tilde w$ in the unit sphere of $\tilde W$ such that $\tilde w \cdot (\basis^i;\basis^j;\alpha - \sigma \mathbf{1}_d) \leq 0$ for all $(i,j,\alpha) \in E$. Because $\tilde w$ belongs to $\tilde W$, there is at least one triple $(i,j,\alpha) \in E$ at which this inequality must be strict. The condition \eqref{convex} implies that
\[ \sum_{(i,j,\alpha) \in E} \theta_{ij\alpha} \tilde w \cdot (\basis^i;\basis^j;\alpha - \sigma \mathbf{1}_d) = 0. \]
As previously, this forces $\tilde w \cdot (\basis^i;\basis^j;\alpha - \sigma \mathbf{1}_d) = 0$ whenever $\theta_{ij\alpha} > 0$. So in particular, every index $(i,j,\alpha) \in E$ for which $\tilde w \cdot (\basis^i;\basis^j;\alpha - \sigma \mathbf{1}_d) < 0$, of which there is at least one, has $\theta_{ij\alpha} = 0$. Decomposing $\tilde w$ into its parts $(\tilde w_;\tilde w_q;\tilde w_d)$, let $D_i$ be the matrix with diagonal entries $(\tilde w_p)_1,\ldots,(\tilde w_p)_p$, $D_2$ be the diagonal matrix with entries $(\tilde w_q)_1,\ldots(\tilde w_q)_q$, and $D_3$ be the diagonal matrix with entries $(\tilde w_d)_1,\ldots,(\tilde w_d)_d$. The matrices
\[ P_{t} :=  |\det e^{t D_3}|^{-\sigma} \rho_{(e^{tD_1},e^{tD_2},e^{tD_3})} P \]
will have the property that 
\[ \partial^{\alpha} [P_t]_{ij}(0) = e^{t \tilde w \cdot (\basis^i;\basis^j;\alpha - \sigma \mathbf{1}_d)} \partial^{\alpha} P_{ij}(0) \]
for each $i,j,\alpha$. Because $\tilde w \cdot (\basis^i;\basis^j;\alpha - \sigma \mathbf{1}_d)$ is never positive for $(i,j,\alpha) \in E$, the limit, which will be called $P_\infty$, as $t \rightarrow \infty$ of $P_t$ exists, and since $|||\cdot |||_\sigma$ is continuous and $|||P_t|||_\sigma = |||P|||_\sigma$, it follows that $|||P_\infty|||_{\sigma} = |||P|||_\sigma$ as well. 

Let $E_\infty$ be the set of indices $(i,j,\alpha)$ for which $\partial^\alpha [P_\infty]_{ij}(0) \neq 0$. Because $\tilde w \cdot (\basis^i;\basis^j;\alpha - \sigma \mathbf{1}_d) = 0$ for all $(i,j,\alpha) \in E$ such that $\theta_{ij\alpha} > 0$, every such triple $(i,j,\alpha)$ belongs to $E_\infty$. Moreover, because there is at least one $(i,j,\alpha) \in E$ such that $\tilde w \cdot (\basis^i;\basis^j;\alpha - \sigma \mathbf{1}_d) < 0$, it follows that $E_\infty$ is a proper subset of $E$.

Because $E_\infty \subset E$, hypotheses 1 and 2 of this lemma continue to hold for $P_\infty$ because they held for $P$. Hypothesis 3 continues to hold as well by simply retaining the already-given values of the constants $\theta_{ij\alpha}$. This is because every term of the sum \eqref{convex} with index $(i,j,\alpha) \in E \setminus E_\infty$ necessarily has $\theta_{ij\alpha} = 0$, meaning that its omission does not change the value of the sum. By induction on the number of indices $(i,j,\alpha)$ for which $\partial^\alpha P_{ij}(0) \neq 0$ and $\theta_{ij\alpha} = 0$, it follows that $|||P_\infty|||_\sigma > 0$ since this number has decreased by at least one in passing from $P$ to $P_\infty$. Since $|||P|||_\sigma = |||P_\infty|||_\sigma$, the lemma is complete.
\end{proof}

\subsection{Nash Functions}
\label{nash}

The working category of functions for Theorem \ref{bigtheorem1} is the Nash functions. For the present purposes, this category is defined as follows.
\begin{definition}
Suppose $\domain \subset \R^d$ is open. A real analytic function $f : \domain \rightarrow \R$ will be called a Nash function when there exists a polynomial $p(x,y) := (p_1(x,y),\ldots,p_N(x,y))$ mapping points $(x,y) \in \R^d \times \R^N$ to vectors $\R^N$ and a real analytic map $F(x) := (f_1(x),\ldots,f_N(x))$ defined for all $x \in \domain$ such that $p(x,F(x))$ is identically zero on $\domain$, $(\det \frac{\partial p}{\partial y})(x,F(x))$ is \textit{never} identically zero on any open subset of $\domain$, and $f(x) = F_i(x)$ for some fixed $i \in \{1,\ldots,N\}$ and all $x \in \domain$. The complexity $K$ of $f$ is the minimum of $\deg p_1 \cdots \deg p_N$ over all maps $p$ satisfying the definition for this $f$.
\end{definition}

The definition just given differs slightly from that appearing in \cite{gressmancurvature}, but the two definitions are completely equivalent. By the earlier definition, $F$ would be regarded as an algebraic lifting map and a Nash function would be one which equals $q(x,F(x))$ for some polynomial $q$. This structure can easily be directly incorporated into $F$ itself, however, by augmenting $y$ with an additional variable $y_{N+1}$, fixing $F_{N+1}(x) = f(x)$, and augmenting $p$ with the polynomial $y_{N+1} - q(x,y_1,\ldots,y_N)$. The Jacobian condition for $p$ still holds for this augmented system because $F_1,\ldots,F_N$ are independent of $y_{N+1}$; likewise the complexity of $f$ in the old definition will be the same as the complexity via this new definition using the augmentation process just described. Conversely, given a function which is Nash according to the new definition, $F$ can act as the algebraic lifting map in the older definition and $f$ can simply equal $y_i$ for the index $i$ such that $F_i = f$. Thus the definitions are equivalent and the complexities agree in both formulations.

While the main interest will be to study functions and forms which are genuinely algebraic, the category of Nash functions provides some important tools.
It is not a difficult exercise to prove that
\begin{itemize}
\item The sum of Nash functions $f$ and $g$ on $\domain$ with complexities $K_1$ and $K_2$ is also Nash with complexity at most $K_1 K_2$.
\item The product of Nash functions $f$ and $g$ on $\domain$ with complexities $K_1$ and $K_2$ is also Nash with complexity at most $2 K_1 K_2$.
\item The ratio of Nash functions $f$ and $g$ on $\domain$ with complexities $K_1$ and $K_2$ is Nash on the open set where $g \neq 0$ with complexity at most $2 K_1 K_2$.
\item The coordinate partial derivative of a Nash function $f$ on $\domain$ with complexity $K$ is a Nash function of complexity at most $K^2$.
\end{itemize}
The most challenging of these is the last one. Assuming that $f$ is Nash, let $p(x,y)$ and $F$ be as given by the definition. To show that $\partial_{t_i} f(t)$ is Nash, we will augment the variables $y$ with additional variables $z \in \R^N$ and take the polynomial map $(p(x,y),\tilde p(x,y,z))$ with
\[ \tilde p(x,y,z) := \left(\frac{\partial p}{\partial x_i} \right)(x, y) + \sum_{j=1}^N z_j \frac{\partial p}{\partial y_j} (x, y). \]
An immediate consequence of the chain rule is that
\[ p(t,F(t)) \equiv \tilde p(t,F(t),\partial_{t_i} F(t)) \equiv 0, \]
and furthermore
\[ \det \frac{\partial (p,\tilde p)}{\partial(y,z)}(x,y) = \left( \det \frac{\partial p}{\partial y} (x,y) \right)^2 \]
which never vanishes on open sets when $x = t, y = F(t)$. Since the degree of each component of $\tilde p$ is never greater than the degree of the corresponding component of $p$, the product of degrees of the augmented system will be at most $K^2$ when the original system has product $K^2$.

As mentioned in the introduction, when $K$ is a fixed parameter bounding the complexity of Nash functions, it will always be assumed that $K$ has been chosen so that $K \geq 2$. The advantage of doing so is that when $f$ and $g$ are Nash of complexity at most $K^{N}$, then each of the fundamental operations above acting on $f$ and $g$ will be Nash of complexity at most $K^{2N+1}$. Thus for more complicated combinations of many of these operations, the result will always be a Nash function of complexity at most $K^{N'}$ for some exponent $N'$ which depends only on the original exponents $N$ and the total number of operations performed.

Because we are concerned with a number of objects which are more complex than functions, we must be explicit about what it means for these objects to be Nash. The following definitions will make this precise.

\begin{definition}
A smooth function $f \ : \ \domain \rightarrow \R^n$ is called a Nash function of complexity at most $K$ when $f(t) := (f_1(t),\ldots,f_n(t))$ and each $f_1,\ldots,f_n$ is Nash of complexity at most $K$.

A smooth function $u \ : \ \domain \rightarrow \Lambda^p(V^*)$ is called Nash of complexity at most $K$ when there is some basis $w^1,\ldots,w^q$ of $V^*$ such that
\[ u(t) = \sum_{i_1 < \cdots < i_p} c_{i_1 \cdots i_p}(t) w^{i_1} \wedge \cdots \wedge w^{i_p} \]
for real analytic functions $c_{i_1 \cdots i_p}(t)$ which themselves are Nash functions of complexity at most $K$ on $\domain$.
\end{definition}

\begin{definition}
For each positive integer pair $(q,r)$, fix some positive integer $N_{q,r}$. Given a smooth family $V(t)$ of $r$-dimensional subspaces of $V$, we say it is Nash of complexity at most $K$ on $\domain$ when there are at most $N_{q,r}$ local parametrizations $(\domain_m,v_m)$ of $V(t)$ such that the open subsets $\domain_m$ have union $\domain$ and on each $\domain_m$, $v_m(t)$  is Nash of complexity at most $K$.
\end{definition}

In practice, the explicit examples in Section \ref{examplesec} are compatible with simply fixing $N_{q,r} = 1$. We allow for the possibility of $N_{q,r}$ being larger (but independent of the function and complexity) to increase compatibility with the basic features of adapted factorizations which arise in Sections \ref{kernelsec} and \ref{adaptsec}. For theoretical purposes, Corollary \ref{bndcount} below establishes that all smooth families of subspaces can be defined using only boundedly many local parametrizations, and it is useful to take $N_{q,r}$ to equal the worst-case number of such local parametrizations needed so that the results below yield identical quantitative results whether or not the subspaces in question happen to be Nash.

In terms of the nature of Nash functions, most of the heavy lifting has already been accomplished in \cite{gressmancurvature}; the main result for vector field construction (Theorem 3) will be used here as a black box. 

\subsection{Other Foundational Results}

\label{kernelsec}

The basic engine behind the concept of adapted factorizations is the following proposition. Loosely speaking, it gives one a quantitative way to smoothly parametrize the null spaces of smoothly-varying linear transformations. A key feature of the proposition is that the complexity of the constructed objects depends in a direct way on the complexity of the objects being studied.
\begin{proposition}
Suppose that $V$ and $W$ are finite-dimensional real vector spaces with $\dim V \geq q$ and let $\domain \subset \R^d$ be an open set. Given smooth $e^{1}(t),\ldots,e^{q}(t)$ $V$-valued functions on $\domain$ which are linearly independent at every point and $M(t) : V \rightarrow W$ which is a smooth family of linear transformations for $t \in \domain$ which have constant rank $r < q$ when restricted to the span of $e^1(t),\ldots,e^q(t)$ at every point $t \in \domain$, there exists a collection of at most $\binom{p}{r} \binom{q}{r}$ open sets $\domain_m \subset \domain$ whose union is $\domain$ such that the following are true:
\label{kernelparam}
\begin{enumerate}
\item For each $m$, there are smooth $V$-valued functions $x^1_m(t),\ldots,x^{q-r}_m(t)$ such that $\{x^i_m(t)\}_{i=1}^{q-r}$ is a basis of $\operatorname{span}\{e^1(t),\ldots,e^q(t)\} \cap \ker M(t)$ at each $t \in \domain_m$.
\item These $\{x^i_m(t)\}_{i=1}^{q-r}$ have coefficients with respect to $e^1(t),\ldots,e^q(t)$ with magnitude at most $2$ at every point $t \in \domain_m$. Moreover, the coefficients are given on $\domain_m$ by a ratio of two homogeneous degree $r$ polynomial functions of the vectors $M(t)e^1(t),\ldots,M(t)e^q(t)$. In particular, there is some exponent $N$ depending only on $q$ and $r$ such that if the entries of $M(t)$ and the coefficients of each $e^i(t)$ are Nash functions of complexity at most $K$ for some $K \geq 2$ when expressed in some fixed (constant) basis, then each $x_m^i(t)$, $i=1,\ldots,q-r$, will also be Nash of complexity at most $K^N$.
\item For each $m$, there are indices $j_1,\ldots,j_r \in \{1,\ldots,q\}$ such that $e^{j_1}(t) \wedge \cdots \wedge e^{j_r}(t) \wedge x^1_m(t) \wedge \cdots \wedge x^{q-r}_m(t) = e^1(t) \wedge \cdots \wedge e^q(t)$ for all $t \in \domain_m$.
\end{enumerate}
\end{proposition}
\begin{proof}
Without loss of generality, it may be assumed that $V := \R^q$ and $W := \R^p$ if one applies the proposition to the matrix-valued function
\[ \overline{M}(t) := \begin{bmatrix} [M(t) e^1(t)]_1 & \cdots & [M(t) e^q(t)]_1 \\ \vdots & \ddots & \vdots \\ [M(t) e^1(t)]_p & \cdots & [M(t) e^q(t)]_p \end{bmatrix} \]
(where $[\cdot]_i$ denotes the $i$-th coordinate with respect to some chosen basis on $W$) and takes new ${e}^1(t),\ldots,{e}^q(t)$ to simply be constant and equal to the standard basis vectors $\basis^1,\ldots,\basis^q$ in $\R^q$.

The sets $\domain_m$ will be indexed by tuples $m := (i_1,\ldots,i_r,j_1,\ldots,j_r)$ such that $1 \leq i_1 < \cdots < i_r \leq p$ and $1 \leq j_1 < \cdots < j_r \leq q$. Let $m(t)$ denote the corresponding $r \times r$ minor of $M(t)$ (i.e., the minor with rows $i_1,\ldots,i_r$ and columns $j_1,\ldots,j_r$).
Because $M$ has rank exactly $r$ at every point $t \in \domain$, there is always some $m$ such that $\det m(t) \neq 0$. Let 
\[ \domain_m := \set{t \in \domain}{ \max_{m'} |\det m'(t)| < 2 |\det m(t)| }. \]
Each $\domain_m$ is open because the determinants are continuous functions. Moreover, each $t \in \domain$ belongs to at least one $\domain_m$ (namely, the $m$ which maximizes $|\det m(t)|$ at the point $t$).
Now fix any choice $m := (i_1,\ldots,i_r,j_1,\ldots,j_r)$, let $\sigma_1,\ldots,\sigma_{q-r}$ be an enumeration of $\{1,\ldots,q\} \setminus \{j_1,\ldots,j_r\}$ in increasing order, and define
\begin{equation} x^k_m(t) := \basis^{\sigma_k} - \sum_{k'=1}^r c_{k'}^k(t) \basis^{j_{k'}} \label{uniquecoeff} \end{equation}
where
\[ \begin{bmatrix} c_{1}^k(t) \\ \vdots \\ c_{r}^k(t) \end{bmatrix} := m^{-1}(t) \begin{bmatrix} M_{i_1}^{\sigma_k}(t) \\ \vdots \\ M_{i_r}^{\sigma_k}(t) \end{bmatrix} \]
(here $M_{i}^j(t)$ denotes the entry of $M(t)$ in row $i$ and column $j$).
By Cramer's rule, each $c_{k'}^{k}$ can be expressed (up to $\pm$ sign) as the ratio of two determinants of minors, namely the minor with rows $i_1,\ldots,i_r$ and columns $j_1,\ldots,\widehat{j_{k'}},\ldots,j_r,\sigma_k$ (where $\widehat{\cdot}$ denotes omission) in the numerator and the minor $m$ in the denominator. By definition of $\domain_m$, the coefficients have magnitude at most $2$ everywhere on $\domain_m$.
Apply $M$ to both sides of the definition of $x^{k}_m(t)$ and restrict the result to rows $i_1,\ldots,i_r$. It follows that
\begin{align*} \begin{bmatrix} [M(t) x^k_{m}(t)]_{i_1} \\ \vdots \\ [M(t) x^k_m(t)]_{i_r} \end{bmatrix} & = \begin{bmatrix} M_{i_1}^{\sigma_k}(t) \\ \vdots \\ M_{i_r}^{\sigma_k}(t) \end{bmatrix} - \sum_{k'=1}^r   \begin{bmatrix} M_{i_1}^{ j_{k'}}(t) \\ \vdots \\ M_{i_r}^{ j_{k'}}(t) \end{bmatrix} c_{k'}^k(t) \\
& = \begin{bmatrix} M_{i_1}^{ \sigma_k}(t) \\ \vdots \\ M_{i_r }^{\sigma_k}(t) \end{bmatrix} - m(t) \begin{bmatrix} c_{1}^{k}(t) \\ \vdots \\ c_{r}^{k}(t) \end{bmatrix} = 0.
\end{align*}
In particular, the coefficients of $M(t) x_m^k(t)$ are zero in each row $i_1,\ldots,i_r$. However, all other rows of $M$ are linear combinations of rows $i_1,\ldots,i_r$, so it follows that $M(t) x^k_m(t) = 0$. The definition of $x^k_m$ also makes it clear that $e^{j_1} \wedge \cdots \wedge e^{j_r} \wedge x^{1}_m \wedge \cdots \wedge x^{q-r}_m$ is constant and equal to $\pm e^1 \wedge \cdots \wedge e^q$ on all of $\domain_m$, so in particular the $x^k_m$ are linearly independent and therefore form a basis of the kernel of $M(t)$. By replacing one of the $x^{1}_m(t)$ with its negative if necessary (which will be the case exactly when $(j_1,\ldots,j_r,\sigma_1,\ldots,\sigma_{q-r})$ is an odd permutation of $(1,\ldots,q)$), one can assume that the wedge has plus sign rather than minus.
\end{proof}

A useful consequence of Proposition \ref{kernelparam} is that one can assume without loss of generality that any smooth family of subspaces is given by a bounded number of local parametrizations.
\begin{corollary}
Suppose $V(t)$ is a smoothly-varying family of $r$-dimensional subspaces of $V$ for $t \in \domain$. Given $e^1(t),\ldots,e^q(t)$ which are smooth and form a basis of $V$ at each point $t \in \domain$,  there exist a finite number of local parametrizations $(\domain_m,v_m)$ (with cardinality depending only on $r$ and $q$) such that the following hold: \label{bndcount}
\begin{enumerate}
\item For each $m$, there are $V$-valued functions $x^1_m(t),\ldots,x^r_m(t)$ on $\domain_m$ such that $v_m(t) = x^1_m (t) \wedge \cdots \wedge x^r_m(t)$ for all $t \in \domain_m$;
\item For each $m$, there exist $1 \leq j_1 < \cdots < j_r \leq q$ such that
\begin{equation} x_m^k(t) = \pm e^{\sigma_k}(t) + \sum_{k'=1}^r c^k_{k'}(t) e^{j_{k'}}(t) \label{vectorsdefined} \end{equation}
for some $\sigma_k \not \in \{j_1,\ldots,j_r\}$ and coefficients $c^{k}_{k'}(t)$ which are bounded in magnitude by at most $2$ on $\domain_m$. Moreover
\[ e^{j_1}(t) \wedge \cdots \wedge e^{j_r}(t) \wedge v_m(t) = e^1(t) \wedge \cdots \wedge e^q(t) \]
for all $t \in \domain_m$.
\end{enumerate}
\end{corollary}
\begin{proof}
Every smoothly varying family of $r$-dimensional subspaces admits a number (and presumably many) of local parametrizations $(\domain_0,v_0(t))$. The principal benefit of this corollary is to establish a quantitative bound on the number of local parametrizations required. For each local parametrization, apply the previous proposition on $\domain_0$ to the map $M(t)$ which sends $\omega \in V$ to $v_0(t) \wedge \omega \in \Lambda^{r+1}(V)$. The kernel of this map is exactly $V(t)$ because $v_0$ is decomposable. This allows $\domain_0$ to be decomposed into subsets $(\domain_0)_m$ for each choice of minor $m$. We now let $\domain_m$ be the union of each $(\domain_0)_m$ over all choices of $\domain_0$. In the formula \eqref{vectorsdefined}, the coefficients $c^{k'}_k(t)$ which make $x_{m}^k(t)$ belong to $V(t)$ are unique because no linear combination of $e^{j_1}(t),\ldots,e^{j_r}(t)$ maps to $0$ via the given map $M(t)$ (if it did, this would force $v_m(t) \wedge e^{j_1}(t) \wedge \cdots \wedge e^{j_r}(t) = 0$ when $v_m$ is the wedge product $x^1_m \wedge \cdots \wedge x^r_m$, which violates the third conclusion of Proposition \ref{kernelparam}). Note also that the $\pm$ sign in \eqref{vectorsdefined} is exclusively a function of $k$ and $m$, so for any fixed pair $(k,m)$, it may be assumed that the sign is the same on every $(\Omega_0)_m$ as $\Omega_0$ ranges over all possible domains. Thus if $(\domain_0)_m$ and $(\domain_0')_m$ have overlap, the corresponding formulas for $x_m^k(t)$ and $(x')_m^k(t)$ will necessarily agree on the overlap and consequently one can globally define $x^k_m(t)$ at every point in $\domain_m$ by simply computing it locally inside some $(\domain_0)_m$ which contains the point in question.
\end{proof}

We close this section with two miscellaneous elementary results which will be used in the proof of Theorem \ref{bigtheorem1}.
The first is a simple but important observation regarding the possibility of subdividing compact sets using finite open covers.
\begin{proposition}
Suppose that $\domain_1,\ldots,\domain_N \subset \R^d$ are open sets and $E \subset \bigcup_{i'=1}^N \domain_{i'}$ is compact. Then for each $i=1,\ldots,N$, there is a compact set $E_i \subset E \cap \domain_i$ such that $\bigcup_{i'=1}^N E_i = E$. \label{cutprop}
\end{proposition}
\begin{proof}
 For each $i$, the set
\[ E_i := \set{x \in E}{ \dist(x,\domain_i^c) = \max_{i'=1,\ldots,N} \dist(x,\domain_{i'}^c)} \]
is clearly a closed and therefore compact subset of $E$, and because the maximum is always attained, $\bigcup_{i'=1}^N E_{i'} = E$. Because $E$ is contained in the union of the $\domain_{i'}$, every $x \in E$ has $\dist(x,\domain_{i'}^c) > 0$ for at least one $i'$, and consequently if $x \in E_i$, then $\dist(x,\domain_i^c) > 0$, meaning that $E_i \subset \domain_i$.
\end{proof}

The final proposition of this section is a simple matrix computation which also appears in \cite{gressman2019}. While elementary to state and prove, it is intimately connected to the so-called First Fundamental Theorem of Invariant Theory (see, for example, Sturmfels \cite{sturmfelsbook}).
\begin{proposition}
Suppose $M$ is a $p \times q$ real matrix with $q \geq p$. Then
\begin{equation} \sum_{j_1,\ldots,j_p=1}^q  \left| \det \begin{bmatrix} M_{1j_1} & \cdots & M_{1j_p} \\ \vdots & \ddots & \vdots \\ M_{p j_1} & \cdots & M_{p j_p} \end{bmatrix} \right|^2 \! = \frac{p!}{p^p} \left[ \inf_{A \in \SL(p,\R)} \sum_{i=1}^p \sum_{j=1}^q |(AM)_{ij}|^2 \right]^p \! . \label{ftgit} \end{equation}
\end{proposition}
\begin{proof}
The first thing to observe is that the left-hand side is invariant under any transformation which replaces $M$ by $A M$ for any $A$ with determinant $\pm 1$. Likewise, if $O$ is any orthogonal $q \times q$ matrix, then the left-hand side is unchanged when replacing $M$ by $M O$. To see this, simply observe that, upon making this replacement, the left-hand side equals
\begin{align*}
\sum_{\substack{j_1,\ldots,j_p\\ j'_1,\ldots,j_p' \\j''_1,\ldots,j''_p}=1}^q \det \begin{bmatrix} M_{1j_1'} & \cdots & M_{1j_p'} \\ \vdots & \ddots & \vdots \\ M_{p j_1'} & \cdots & M_{p j_p'} \end{bmatrix} \det \begin{bmatrix} M_{1j_1''} & \cdots & M_{1j_p''} \\ \vdots & \ddots & \vdots \\ M_{p j_1''} & \cdots & M_{p j_p''} \end{bmatrix} \prod_{k=1}^p O_{j'_k j_k} O_{j''_k j_k},
\end{align*}
which can be seen to be independent of $O$ by performing the sums over $j_1,\ldots,j_p$ first. It follows that when computing the left-hand side, one may assume by the Singular Value Decomposition that $M$ is diagonal with nonnegative entries. In particular, this means that
\[ \sum_{j_1=1}^q \cdots \sum_{j_p=1}^q \left| \det \begin{bmatrix} M_{1j_1} & \cdots & M_{1j_p} \\ \vdots & \ddots & \vdots \\ M_{p j_1} & \cdots & M_{p j_p} \end{bmatrix} \right|^2 \]
equals $p!$ times the product of squares of the singular values of $M$ (with the $p!$ arising from the possible ways that $j_1,\ldots,j_p$ may take values $1,\ldots,p$). A consequence of the reasoning above is that the product of singular values is unchanged when $M$ is replaced by $AM$ for some $A$ with determinant $\pm 1$.  

On the other side of the identity, 
\[  \sum_{i=1}^p \sum_{j=1}^q |(AM)_{ij}|^2  \]
is simply the square of the Hilbert-Schmidt norm of $AM$, which is manifestly invariant under the action on both left and right by orthogonal matrices. This means that it must simply equal the sum of squares of the singular values of $AM$. By the arithmetic-geometric mean inequality,
\[ \prod_{j=1}^p (\sigma_j^2)^\frac{1}{p} \leq \frac{1}{p} \sum_{j=1}^p \sigma^2_j \]
where $\sigma_1,\ldots,\sigma_p$ are singular values of $AM$. In particular, this means that the left-hand side of \eqref{ftgit} is never larger than the right-hand side.

To establish equality of the two sides, let $\epsilon > 0$ and let $A$ be the matrix whose eigenvectors coincide with the eigenvectors of $MM^T$ with eigenvalue $(\sigma_j + \epsilon)^{-1} \prod_{j'=1}^p (\sigma_{j'}+\epsilon)^{1/p}$ for the eigenvector of $MM^T$ with original eigenvalue $\sigma_j^2$. The product of $A$'s eigenvalues is $1$, meaning $A \in \SL(p,\R)$. The singular values of $A M$ equal the square root of the eigenvalues of $A M M^T A^T$, which means that they equal $\sigma_j/(\sigma_j+\epsilon) \prod_{j'=1}^{p} (\sigma_j + \epsilon)^{1/p}$ for $j=1,\ldots,p$. In the limit $\epsilon \rightarrow 0^+$, each value tends to $\prod_{j'=1}^p \sigma_j^{1/p}$. Thus the infimum on the right-hand side of \eqref{ftgit} is at most $p$ times the product of singular values of $M$ raised to the power $2/p$, giving that \eqref{ftgit} must be an equality.
\end{proof}

\subsection{Properties of Block Decompositions}

\label{adaptsec} 

This section establishes a number of fundamental properties of block decompositions, including the guaranteed existence of quantitatively ``nice'' adapted factorizations (Corollary \ref{containercorollary}), information about transformations from one adapted factorization to another (Proposition \ref{scaleprop}) and the order of vanishing of pairings of adapted factorizations in a block decomposition (Proposition \ref{blockprop}).

The first key fact to be established is the following: if $V_0(t) \supset V_1(t) \supset \cdots \supset V_m(t)$ is a family of decreasing smooth subspaces of some vector space $V$ and if $(\domain,e^1(t) \wedge \cdots \wedge e^q(t))$ parametrizes $V_0$, then there is always an adapted factorization $\{y^{j,j'}(t)\}_{j,j'}$ of the wedge product $e^1(t) \wedge \cdots \wedge e^q(t)$ with bounded coefficients when expressed in terms of the $e^i(t)$ and which is Nash in a controlled way when the $V_i(t)$ and the $e^i(t)$ are all Nash.
\begin{corollary}
Given smooth families of subspaces $V_0(t) \supset V_1(t) \cdots \supset V_m(t) \supset V_{m+1}(t) = \{0\}$ for $t \in \domain$, let $q_j := \dim V_j - \dim V_{j+1}$ for each $j \in \{0,\ldots,m\}$ and suppose $q_j > 0$ for each $j$.  If $e^1(t),\ldots,e^q(t)$ are smooth $V$-valued functions which form a basis of $V_0(t)$ at each point, then there exists a finite collection of open sets $\{\domain'\}$ with cardinality bounded by some function of $q$ and $m$ such that \label{containercorollary}
\begin{enumerate}
\item There exist $V$-valued smooth functions $y^{j,j'}(t)$ for $j \in \{ 0,\ldots,m \}$ and $j' \in \{1,\ldots,q_j\}$ for $t \in \domain'$ such that each $y^{j,j'}(t)$ is a linear combination of $e^1(t),\ldots,e^q(t)$ with coefficients bounded in magnitude by some quantity depending only on $q$ and $m$.
\item For each $j \in \{0,\ldots,m\}$, the vectors $y^{j,1}(t),\ldots,y^{j,q_j}$ belong to $V_j(t)$ and are linearly independent modulo $V_{j+1}(t)$. In particular, the vectors $y^{j,j'}(t)$ with $j \geq k$ are a basis of $V_k(t)$ for each $k \in \{0,\ldots,m\}$.
\item For each $t \in \domain_m$, 
\[ \bigwedge_{j=0}^m \bigwedge_{j'=1}^{q_j} y^{j,j'}(t) = e^1(t) \wedge \cdots \wedge e^q(t). \]
\item If the spaces $V_1(t),\ldots,V_k(t)$ are Nash of complexity at most $K$ and each $e^i(t)$ is also Nash of complexity at most $K$ for some $K \geq 2$, then each $y^{j,j'}(t)$ is also Nash of complexity at most $K^N$ for some exponent $N$ depending only on dimensions and $m$.
\end{enumerate}
\end{corollary}
\begin{proof}
The proof is a rather direct consequence of Proposition \ref{kernelparam}.
Appealing to either Corollary \ref{bndcount} or to the definition of a Nash family of smooth subspaces,  begin by restricting to one of boundedly many subsets of $\domain$ such that each $V_j(t)$ is locally parametrized on this subset by some decomposable form $v_j(t)$. If the families of subspaces are Nash, the reduction allows one to assume that each $v_j(t)$ is complexity at most $K$ on each of these subsets. 

The proof follows by induction on $m$. In the case $m=0$, there is essentially nothing to prove because one may simply set $y^{0,j'}(t) = e^{j'}(t)$ for each $j'$ and observe that these $y^{0,j'}$ satisfy all the desired properties using a trivial decomposition of $\domain$ into a single piece.

One may therefore assume that $m \geq 1$.  Apply Proposition \ref{kernelparam} with map $M(t)$ which sends $\omega$ to $v_1(t) \wedge \omega$. This gives a decomposition of the domain of $v_1(t)$ into boundedly many pieces on which:
\begin{itemize}
\item There exist $V$-valued functions $x^1(t),\ldots,x^{\dim V_1}(t)$ which span the kernel of $M(t)$, which happens to be exactly $V_1(t)$.
\item The functions $x^i(t)$ are expressible in the basis $e^1(t),\ldots,e^{q}(t)$ with coefficients which are at most $2$. There is also an exponent $N$ depending only on dimensions such that if the $e^i(t)$ are Nash with complexity at most $K$ and $v_1(t)$ is as well, then the $x^i(t)$ will be Nash of complexity at most $K^N$.
\item There exist $1 < j_1 < \cdots < j_{q_0}$ such that
\[ e^{j_1}(t) \wedge \cdots \wedge e^{j_{q_0}}(t) \wedge x^1(t) \wedge \cdots \wedge x^{\dim V_1}(t) = e^1(t) \wedge \cdots \wedge e^{q}(t) \]
at all points on this piece of the domain.
\end{itemize}
Now set $y^{0,1}(t),\ldots,y^{0,q_0}(t)$ equal $e^{j_1}(t),\ldots,e^{j_{q_0}}(t)$ and (by induction) apply the current corollary to the shorter sequence $V_1(t) \supset \cdots \supset V_{m}(t)$ using the local basis $x^1(t),\ldots,x^{\dim V_1}(t)$ of $V_1(t)$.  By hypothesis, there exist $y^{j,j'}(t)$ with $j \in \{1,\ldots,m\}$ and $j' \in \{1,\ldots,q_j\}$ which are expressible as linear combinations of the $x^i$ with bounded coefficients such that $V_j(t)$ is the span of all $y^{k,k'}$ for $k \geq j$, the $y^{j,j'}(t)$ have suitably bounded Nash complexity, and
\[ \bigwedge_{j=1}^m \bigwedge_{j'=1}^{q_j} y^{j,j'}(t) = x^1(t) \wedge \cdots \wedge x^{\dim V_1}(t). \]
Since each $x^i$ is itself expressible as a linear combination of the $e^i$ with bounded coefficients (and controlled Nash complexity), each $y^{j,j'}$ will have bounded coefficients when expressed in terms of the original $e^i(t)$. The rest follows by virtue of the observation that
\begin{align*}
e^1(t) \wedge \cdots \wedge e^{q}(t) & = y^{0,1}(t) \wedge \cdots \wedge y^{0,q_0}(t) \wedge x^1(t) \wedge \cdots \wedge x^{\dim V_1}(t) \\
& = \bigwedge_{j=0}^{m} \bigwedge_{j'=1}^{q_j} y^{j,j'}(t)
\end{align*}
giving the desired factorization property.
\end{proof}

Looking at the quantity on the right-hand side of \eqref{nondegenhyp} in Theorem \ref{bigtheorem1}, a rough analogy between it and the definition \eqref{stdnorm} for polynomial-valued matrices $P$ should be clear. One important distinction, however, is that the $\overline{u}^{i,i'}$ and $\overline{v}^{j,j'}$ of \eqref{nondegenhyp} are not fully generic bases but rather are adapted factorizations. This means that the underlying group of symmetries that governs \eqref{nondegenhyp} could in principle be smaller than $\SL(q,\R) \times \SL(p,\R) \times \GL(d,\R)$.  For this reason, it is necessary to understand more precisely the relationship between different adapted factorizations of the same form.
The following proposition gives such an analysis. The key observation is that a natural sequence of scaling parameters arise which relate the vectors in one factorization to the vectors in the other. This proposition operates at a purely pointwise level, so there is no need to reference any functional parameter $t$ at this stage.
\begin{proposition} Let $V = V_0 \supset \cdots \supset V_m \supset V_{m+1} = \{0\}$ be subsets of some vector space $V$ and let $q_i := \dim V_i - \dim V_{i+1}$ for each $i \in \{0,\ldots,m\}$.
Suppose that $v^{j,j'}$ and $\overline{v}^{j,j'}$ are vectors in $V$ for $j \in \{0,\ldots,m\}$ and $j' \in \{1,\ldots,q_j\}$ having the properties that \label{scaleprop}
\begin{enumerate}
\item For each $j \in \{0,\ldots,m\}$ and each $j' \in \{1,\ldots,q_j\}$, $v^{j,j'} \in V_j$ and $\overline{v}^{j,j'} \in V_j$.
\item The vectors satisfy
\[ \bigwedge_{j=0}^m \bigwedge_{j'=1}^{q_j} v^{j,j'} = \bigwedge_{j=0}^m \bigwedge_{j'=1}^{q_j} \overline{v}^{j,j'} \neq 0. \]
\end{enumerate} 
Then there exist nonnegative real numbers $\delta_0,\ldots,\delta_m$ with $\delta_0 \cdots \delta_m = 1$ such that for each interval $I := [a,b] \subset \{0,\ldots,m\}$, there is a matrix $M^{I}$ of determinant $\pm 1$ (whose entries are denoted $M^{I,(j,j')}_{(k,k')}$) and
vectors $E^{j,j'}$ belonging to $V_{b+1}$ such that
\begin{equation} \overline{v}^{j,j'} = \left(\prod_{\ell \in I} \delta_\ell \right)^{(\sum_{\ell \in I} q_\ell)^{-1}} \sum_{k \in I} \sum_{k' = 1}^{q_k} M^{I,(j,j')}_{(k,k')} v^{k,k'} + E^{j,j'} \label{samedeltas} \end{equation}
for all $j \in I$ and all $j' \in \{1,\ldots,q_j\}$.
\end{proposition}
\begin{proof} 
Because the wedge products of the $v^{j,j'}$ and $\overline{v}^{j,j'}$ are equal and nonzero, it follows that each $v^{j,j'}$ is uniquely expressible as a linear combination of the $\overline{v}^{j,j'}$ and vice-versa. Moreover, by counting dimensions, one must have that
\[ V_j = \vecspan_{\substack{k \geq j \\ k' \in \{1,\ldots,q_k\}}} v^{k,k'} = \vecspan_{\substack{k \geq j \\ k' \in \{1,\ldots,q_k\}}} \overline{v}^{k,k'} \]
for each $j$, which means that each $\overline{v}^{j,j'}$ must be uniquely expressible as a linear combination of $v^{k,k'}$ for pairs $(k,k')$ with $k \geq j$. This implies the existence of unique matrices $A^{j}$ and constants $c_{j,j',k,k''}$ such that
\begin{equation} \overline{v}^{j,j'} = \sum_{j''=1}^{q_j} A^{j,j'}_{j''} v^{j,j''} + \sum_{k>j} \sum_{k''} c_{j,j',k,k''} v^{k,k''} \label{lot} \end{equation}
for each $j \in \{0,\ldots,m\}$ and $j' \in \{1,\ldots,q_j\}$. If any matrix $A^j$ were not invertible, then some linear combination of $\overline{v}^{j,j'}$ for $j' \in \{1,\ldots,p_j\}$ would belong to the span of $v^{k,j'}$ for $k > j$ and consequently the wedge product $\bigwedge_{k > j} \bigwedge_{k'=1}^{q_k} \overline{v}^{k,k'}$ would have to vanish (which is never the case). Thus we may define the nonzero real number $\delta_j$ to be the absolute value of the determinant of $A^j$ for each $j$.
By induction, it is straightforward to see that
\begin{equation} \bigwedge_{j=k}^m \bigwedge_{j'=1}^{q_j} \overline{v}^{j,j'}  = \left( \prod_{j=k}^m \det A^j \right) \bigwedge_{j=k}^m \bigwedge_{j'=1}^{q_j} v^{j,j'} \label{partialproduct} \end{equation}
for $k=m,m-1,\ldots,0$. The case $k=m$ is trivial, and in cases of smaller $k$, one need simply use the induction hypothesis combined with the observation that any vector belonging to the span of $v^{j,j'}$ for $j \geq k+1$ will be annihilated when wedged with
\[ \bigwedge_{j=k+1}^m \bigwedge_{j'=1}^{q_j} {v}^{j,j'} \]
or the corresponding expression in $\overline{v}^{j,j'}$. In other words, one can rely on the fact that modulo the span of the $v^{j,j'}$ for higher $j$, one can assume without loss of generality that the error term with coefficients $c_{j,j',k,j''}$ vanishes. Using \eqref{partialproduct} and the hypotheses of the proposition give that the product $\delta_0 \cdots \delta_m = 1$.

Now if $I := [a,b] \subset \{0,\ldots,m\}$ is an interval, the formula \eqref{lot} continues to apply for those $j \in [a,b]$ to give some matrix $A^I$ such that
\[ \overline{v}^{j,j'} = \sum_{k \in I} \sum_{k'=1}^{q_k} A^{I,(j,j')}_{(k,k')} v^{k,k'} + E^{j,j'} \]
for each $j \in I$ and $j' \in \{1,\ldots,q_j\}$, where $E^{j,j'}$ belongs to the span of the $v^{k,k'}$ for $k > b$. The matrix $A^I$ must have block triangular form in the sense that $A^{I,(j,j')}_{(k,k')} = 0$ when $k < j$. The blocks on the diagonal are exactly the matrices $A^a,\ldots,A^b$. Therefore the determinant of $A^I$ is equal to $\delta_a \cdots \delta_b$ up to sign. The formula \eqref{samedeltas} now can be seen to hold by taking $M^{I}$ to equal $A^I$ times $(\delta_a \cdots \delta_b)^{-1/(q_a + \cdots + q_b)}$ so that its determinant is $\pm 1$.
\end{proof}

We return now to the definition of block decompositions in order to establish an important consequence.  The definition refers to the interior product, which is not an especially common operation to encounter, but its role happens to be relatively minor and there is certain utility which comes from defining block decompositions independently of any particular choice of factorization of the underlying forms. For the present purposes, it simply suffices to know that the interior product is linear in both factors and satisfies
\begin{equation} (u^1 \wedge \cdots \wedge u^p) \cdot \omega = \sum_{i=1}^p (-1)^{i-1} (u^i \cdot \omega) u^1 \wedge \cdots \wedge \widehat{u^i} \wedge \cdots \wedge u^p \label{interiorprod} \end{equation}
when $u^1,\ldots,u^p \in V^*$ and $\omega \in V$, where $\widehat{\cdot}$ denotes omission. The formula is analogous when the roles of $V$ and $V^*$ are interchanged.

\begin{proposition}
Suppose that the pair $(U_0,V)$ admits a block decomposition. \label{blockprop}
If $\overline{u}(s)$ is any smooth $V^*$-valued function with values in $U_i(s)$ for each $s$ and $\overline{v}(t)$ is any smooth $V$-valued function with values in $V_j(t)$ for each $t$, then 
\begin{equation} (\partial^{\alpha} \overline{u})(t) \cdot (\partial^{\beta} \overline{v})(t) = 0 \label{diagest} \end{equation}
whenever $|\alpha| + |\beta| < D_{ij}$.
\end{proposition}
\begin{proof}
Using a local parametrization $(\domain_0,u_i)$ for $U_i(t)$ near any particular $t \in \domain$, suppose that $u_i(t) = u_i^1(t) \wedge \cdots \wedge u^p_i(t)$ for some smooth $u_i^1(t),\ldots,u_i^p(t)$. Because $\overline{u}$ takes values in $U_i(t)$, we may locally write $\overline{u}(t)$ as a linear combination of $u_i^1(t),\ldots,u_i^p(t)$ with smooth coefficients. By the product rule, it suffices to show that \eqref{diagest} holds with $\overline{u}(t)$ being equal to $u_i^{i'}(t)$ for some $i'$.

Suppose first that $|\beta| = 0$. For any $|\alpha| < D_{ij}$, because $\overline{v}(t)$ takes values in $V_j(t)$, $(\partial^{\alpha}  u_i)(t) \cdot \overline{v}(t) = 0$. But the product rule guarantees that
\begin{align*} (& \partial^{\alpha}  u_i)(t) \cdot \overline{v}(t) = 0  = \\ &
\sum_{\substack{\alpha_1 + \cdots + \alpha_p = \alpha\\i'=1,\ldots,p}} \frac{\alpha! (-1)^{i'-1}}{\alpha_1! \cdots \alpha_p!}  ((\partial^{\alpha_{i'}} u^{i'}_i)(t) \cdot \overline{v}(t))(\partial^{\alpha_1} u^1_i \wedge \cdots \wedge \widehat{u^{i'}_i} \wedge \cdots \wedge \partial^{\alpha_p} u^p_i)(t)
\end{align*}
for each $t$.
If $\alpha = 0$, we can use the fact that the forms $(u_i^1 \wedge \cdots \wedge \widehat{u_i^{i'}} \wedge \cdots \wedge u_{i}^p)(t)$ for $i'=1,\ldots,p$ are linearly independent at every $t$ to conclude that $u_i^{i'}(t) \cdot \overline{v}(t) = 0$ for each $i'$. Supposing that it has been shown that $(\partial^{\alpha} u_i^{i'} )(t)\cdot \overline{v}(t)$ vanishes for all $|\alpha| < k$, if $|\alpha| = k$, then
\begin{align*} (& \partial^{\alpha}  u_i)(t) \cdot \overline{v}(t) = 0  =
 \sum_{i'=1}^p (-1)^{i'-1} ((\partial^{\alpha} u^{i'}_i)(t) \cdot \overline{v}(t))(u^1_i \wedge \cdots \wedge \widehat{u^{i'}_i} \wedge \cdots \wedge u^p_i)(t)
\end{align*}
because every other term in the product formula expansion has a leading coefficient of the form $(\partial^{\alpha_{i'}} u^{i'}_i(t)) \cdot \overline{v}(t)$ for $|\alpha_i|  < k$, which is known by induction to equal zero. But now the same argument just applied also gives that $(\partial^\alpha u_i^{i'})(t) \cdot \overline{v}(t) = 0$ for this $\alpha$ as well.

It is now known that $(\partial^\alpha \overline{u})(t) \cdot \overline{v}(t)$ vanishes identically in $t$ when $|\alpha| < D_{ij}$, which means that
\[ \partial^{\beta} \left[ (\partial^\alpha \overline{u})(t) \cdot \overline{v}(t) \right] = 0 \]
when $\beta$ is any multiindex and $|\alpha| < D_{ij}$. Assuming that it is known that $(\partial^\alpha \overline{u})(t) \cdot (\partial^\beta \overline{v})(t) = 0$ when $|\alpha| + |\beta| < D_{ij}$ and $|\beta| \leq k$ (which has been established in the case $k=0$), when $|\beta| = k+1$,
$\partial^\beta \left[ (\partial^\alpha \overline{u})(t) \cdot \overline{v}(t) \right]$ equals $(\partial^\alpha \overline{u})(t) \cdot (\partial^\beta \overline{v})(t)$ plus terms which are known to be zero (because each such term has $k$ or fewer derivatives on $\overline{v}(t)$). Thus \[ (\partial^\alpha \overline{u})(t) \cdot (\partial^\beta \overline{v})(t) = \partial^\beta \left[ (\partial^\alpha \overline{u})(t) \cdot \overline{v}(t) \right] = 0 \]
as desired.
\end{proof}

A small but important remark: observe that the full force of the definition of block decompositions was not used in the above proposition. Specifically, the proposition used the condition that $(\partial^{\alpha} u_i(t)) \cdot \overline{v}(t) = 0$ for $|\alpha| < D_{ij}$ when $u_i$ parametrizes $U_i(t)$ and $\overline{v}$ belongs to $V_j(t)$, but the reverse condition, namely that $\overline{u}(t) \cdot \partial^{\alpha} v_j(t) = 0$ when $v_j$ parametrizes $V_j(t)$ and $\overline{u}$ belongs to $U_i(t)$, is not needed. This is because there is redundancy in the definition and either condition implies the other. The fact that the former implies the latter follows from the proof of Proposition \ref{blockprop} since
\begin{align*}
\overline{u}& (t)  \cdot \partial^\alpha v_j(t) =\\ & \sum_{\substack{\alpha_1+\cdots + \alpha_q = \alpha \\ i'=1,\ldots,q}} \frac{\alpha!}{\alpha_1! \cdots \alpha_q!}  (-1)^{i'-1} (\overline{u}(t) \cdot (\partial^{\alpha_{i'}} v^{i'}_j)(t) )(v^1_j \wedge \cdots \wedge \widehat{v^{i'}_j} \wedge \cdots \wedge v^q_j)(t)
\end{align*}
and every coefficient $\overline{u}(t) \cdot (\partial^{\alpha_{i'}} v^{i'}_j)(t) $ vanishes when $|\alpha| < D_{ij}$. The reverse implication follows by symmetry in $V^*$ and $V$.

An important corollary of Proposition \ref{blockprop} is the following:
\begin{corollary}
Suppose that the pair $(U_0,V)$ admits a block decomposition. \label{commutecorollary}
If $\overline{u}(s)$ is any smooth $V^*$-valued function with values in $U_i(s)$ for each $s$ and $\overline{v}^{1}(t),\ldots,v^{q'}(t)$ are smooth $V$-valued functions with values in $V_j(t)$ for each $t$, then 
\begin{equation} (\partial^{\alpha} \overline{u})(t) \cdot \left(\partial^{\beta} \sum_{j'=1}^{q'} B_{j'j''} \overline{v}^{j''} \right)(t) = \sum_{j''=1}^{q'} B_{j'j''}(t) (\partial^{\alpha} \overline{u})(t) \cdot (\partial^{\beta}  \overline{v}^{j''})(t)\label{diagest2} \end{equation}
whenever $|\alpha| + |\beta| = D_{ij}$ and $B_{j'j''}(t)$ are any smooth functions.
\end{corollary}
\begin{proof}
By the product rule, the difference between the left- and right-hand sides of \eqref{diagest2} can be written as a smooth linear combination of terms $(\partial^{\alpha} \overline{u})(t) \cdot (\partial^{\beta'}  \overline{v}^{j''})(t)$ where $|\beta'| < |\beta|$, meaning by \eqref{diagest} that all such terms vanish.
\end{proof}

For later reference, the combined contents of Corollary \ref{containercorollary} and Proposition \ref{blockprop} as they apply in the specific context of block decompositions are recorded in the following lemma.
\begin{lemma}
Suppose $u^1(t),\ldots,u^p(t)$ are smooth $V^*$-valued functions on some domain $\domain$ such that $u(t) = u^1(t) \wedge \cdots \wedge u^p(t)$ is nonvanishing at all points of $\domain$.  Let $U_0(t)$ be the smooth family of subspaces of $V^*$ which is locally parametrized by $(\domain,u)$. If $\omega^1,\ldots,\omega^q$ is a volume-normalized basis of $V$ and if $(U_0,V)$ admits a block decomposition $U_0(t) \supset \cdots \supset U_{\mstar}(t) \supset U_{\mstar+1}(t) = \{0\}$ and $V = V_0(t) \supset \cdots \supset V_{m}(t) \supset V_{m+1}(t) = \{0\}$ with formal degrees $D_{ij}$ on $\domain \subset \R^d$, there exist boundedly many subsets of $\Omega$ (with the number being controlled by a function of $p$ and $q$ only) on which there are smooth $V^*$-valued functions $\{\overline{u}^{i,i'}(s)\}_{i,i'}$ for $i \in \{0,\ldots,\mstar\}$ and $i' \in \{1,\ldots,p_i\}$ and smooth $V$-valued functions $\{\overline{v}^{j,j'}(t)\}_{j,j'}$ for $j \in \{0,\ldots,m\}$, $j' \in \{1,\ldots,q_j\}$ such that \label{adaptedfactorization}
\begin{enumerate}
\item The $\overline{u}^{i,i'}(t)$ equal smooth linear combinations of the $u^1(t),\ldots,u^p(t)$ with coefficient functions bounded in magnitude by some constant depending only on $p$ and $q$. Likewise the $\overline{v}^{j,j'}(t)$ are expressible as smooth linear combinations of the $\omega^1,\ldots,\omega^q$ with coefficients bounded in magnitude by some constant depending on $p$ and $q$.
\item If $u(s)$ is Nash of complexity at most $K$ (for $K \geq 2$) on $\domain$, then each $\overline{u}^{i,i'}(t)$ and $\overline{v}^{j,j'}(t)$ is Nash of complexity at most $K^N$ for some $N$ depending only on $p$ and $q$.
\item For each $i \in \{0,\ldots,\mstar\}$ and each $i' \in \{1,\ldots, p_i\}$, $\overline{u}^{i,i'}(s)$ takes values in $U_i(s)$ at each $s$. Similarly, for each $j \in \{0,\ldots,m\}$ and $j' \in \{1,\ldots,q_j \}$, $\overline{v}^{j,j'}(t)$ takes values in $V_j(t)$ for each $t$. At all points $t$ in the domain of the $\overline{u}^{i,i'}$ and $\overline{v}^{j,j'}$,
\begin{equation} (\partial^\alpha \overline{u}^{i,i'})(t) \cdot (\partial^\beta \overline{v}^{j,j'})(t) = 0 \label{annihilated} \end{equation}
when $|\alpha| + |\beta| < D_{ij}$.
\item For each $t$ in this restricted domain, 
\begin{align*}
\bigwedge_{i=0}^{\mstar} \bigwedge_{i'=1}^{p_i} \overline{u}^{i,i'}(t) & = u(t), \\
\bigwedge_{j=0}^{m} \bigwedge_{j'=1}^{q_j} \overline{v}^{j,j'}(t) & = \omega^1 \wedge \cdots \wedge \omega^q,
\end{align*}
i.e., $\{\overline{u}^{i,i'}(t)\}_{i,i'}$ is an adapted factorization of $u(t)$ and $\{\overline{v}^{j,j'}\}_{j,j'}$ is an adapted factorization of $V$.
\end{enumerate}
\end{lemma}

\section{Proof of Theorem \ref{bigtheorem1}}
\label{proofsec}

\subsection{Geometric Differential Inequalities}
\label{gdineqsec}

The proof of Theorem \ref{bigtheorem1} is based on the geometric vector field framework which was first developed to study the Oberlin affine curvature condition \cite{gressman2019} and has subsequently proved useful in a number of other contexts. As several proofs of the main results have already appeared in the literature and as no fundamental modifications of the result are needed in the present case, the reader is referred to Theorem 3 of \cite{gressmancurvature} for details. The present paper will treat this construction as a black box.

Suppose that $\domain_1 \supset \domain_2 \supset \cdots$ are open sets in $\R^d$ and that for each integer $N \geq 1$,  $\{X^{(N)}_i\}_{i=1}^d$ is a family of smooth vector fields on $\domain_N$.  Any $\alpha := ((N_1,\ldots,N_\ell),(i_1,\ldots,i_\ell))$ such that $N_1 < N_2 < \cdots < N_\ell$ and $i_1,\ldots,i_\ell \in \{1,\ldots,d\}$, will be called a generalized multiindex and $X^{\alpha}$ will be defined to equal the differential operator
\begin{equation} X^\alpha := X^{{(N_\ell)}}_{i_\ell} \cdots X^{{(N_1)}}_{i_1} \label{valid} \end{equation}
which acts on smooth functions defined on $\domain_{N_\ell}$.
The the order of this operator is $\ell$ (which will also be denoted $|\alpha|$) and $N_{\ell}$ will be called the generation of the operator $X^\alpha$.  We also consider $\alpha := (\emptyset,\emptyset)$ to be a generalized multiindex of order $0$ and generation $0$ and define $X^{\alpha}$ to be the identity operator. 

\begin{theorem}[Theorem 3 of \cite{gressmancurvature}]
Suppose $\domain_0 \subset \R^d$ is open and $f : \domain_0 \rightarrow \R^m$ is real analytic and has a rank $d$ Jacobian matrix $D_x f$ at every $x \in \domain_0$. Suppose also that each component function $f^j$ is a Nash function on $\domain_0$ of complexity at most $K$ and that $W$ is some nonnegative locally integrable function on $\domain_0$.  \label{mainineq} Finally, suppose that $E_0 \subset \domain_0$ is a compact set such that $\sup_{x \in E_0} |f^j(x)| \leq 1$ for each $j$.
 Then for every integer $N \geq 1$, there exists an open set $\domain_N \subset \domain_{N-1}$, a compact set $E_N \subset E_{N-1} \cap \domain_N$, smooth vector fields $\{X^{(N)}_{i}\}_{i=1,\ldots,d}$ defined on $\domain_N$, and positive constants $c_{N,d}$ depending only on $d,N$ such that the following are true:
\begin{enumerate}
\item For each $N \geq 1$, $W(E_N) \geq c_{N,d} m^{-Nd} W(E)$, where $W$ applied to a set denotes the measure of the set with respect to $W \, dx$ (i.e., Lebesgue measure with density $W$).
\item For each $N', N$ with $1 \leq N' < N$ and each $x \in \domain_N$,
\begin{equation} \left. X^{(N)}_i \right|_{x} = \sum_{i'=1}^d c_{i}^{i'} (x) \left. X^{(N')}_{i'} \right|_x \label{basissize0} \end{equation}
for some smooth coefficients $c_{i}^{i'}$ of magnitude at most $2$.
\item For each $N \geq 1$ and each $x \in E_N$, 
\begin{equation} W(x) |\det (X^{(N)}_1,\ldots,X^{(N)}_d)| \Big|_x \geq c_{N,d} m^{-Nd} K^{-(2d+2)^N} W(E_0). \label{nondegenvecs} \end{equation}
 Here $\det (X^{(N)}_1,\ldots,X^{(N)}_d)$ indicates the determinant of the $d \times d$ matrix whose columns are given by the representations of $X^{(N)}_1,\ldots,X^{(N)}_d$ in the standard coordinates of $\R^d$.
\item For each $j \in \{1,\ldots,m\}$ and each generalized multiindex $\alpha$ of generation at most $N$,
\begin{equation} \sup_{x \in \domain_N} |X^\alpha f^j(x)| \leq 1. \label{diffineq} \end{equation}
\end{enumerate}
\end{theorem}

As was done in \cite{gressman2019} and \cite{gressman2022}, the theorem above will be adapted slightly to yield bounds for $X^{\alpha} f$ in terms of the supremum of $f$ on $E$ rather than explicitly requiring $f$ to be normalized. It will also be necessary to establish a simple multiparameter variation of \eqref{diffineq}. These objectives are accomplished by the following lemma.
\begin{lemma}
Suppose that $\mathcal F$ is an $m$-dimensional vector space of Nash functions on $\domain \subset \R^d$ such that all functions in $\mathcal F$ have complexity at most $K$ (and assume without loss of generality that $K \geq 2$). Let $W$ be a nonnegative measurable function on $\domain$ and suppose $E \subset \domain$ is compact and has positive Lebesgue measure. Then for every integer $N \geq 1$, there exists an open set $\domain_N \subset \domain_{N-1}$, a compact set $E_N \subset E_{N-1} \cap \domain_N$, smooth vector fields $\{X^{(N)}_{i}\}_{i=1,\ldots,d}$ defined on $\domain_N$, and positive constants $c_{N,d,m}$ depending only on $d,N,m$ such that the following are true: \label{multilinear}
\begin{enumerate}
\item For each $N \geq 1$, $W(E_N) \geq c_{N,d,m} W(E)$.
\item For each $N', N$ with $1 \leq N' < N$ and each $x \in \domain_N$,
\begin{equation} \left. X^{(N)}_i \right|_{x} = \sum_{i'=1}^d c_{i}^{i'} (x) \left. X^{(N')}_{i'} \right|_x \label{basissize099} \end{equation}
for some smooth coefficients $c_{i}^{i'}$ of magnitude at most $2$.
\item For each $N \geq 1$ and each $x \in E_N$, 
\begin{equation} W(x) |\det (X^{(N)}_1,\ldots,X^{(N)}_d)| \Big|_x \geq c_{N,d,m} K^{-(2d+2)^N} W(E). \label{nondegenvecs99} \end{equation}
\item For any generalized multiindices $\alpha_1,\ldots,\alpha_k$ of generation at most $N$, if $g : \domain^k \rightarrow \R$ has the property that $x_i \mapsto g(x_1,\ldots,x_k)$ belongs to $\mathcal F$ when each $x_j$ is fixed for $j \neq i$, then
\begin{equation} \begin{split}
\sup_{x_1 \in \domain_N} \cdots \sup_{x_k \in \domain_N} & |X^{\alpha_1}_1 \cdots X_{k}^{\alpha_k} g(x_1,\ldots,x_k)| \\ & \leq m^{k} \sup_{x_1 \in E} \cdots \sup_{x_k \in E} |g(x_1,\ldots,x_k)|,
\end{split} \label{diffineq99} \end{equation}
where $X^{\alpha_i}_i$ denotes the differential operator $X^{\alpha_i}$ acting on the variables $x_i$.
\end{enumerate}
\end{lemma}
\begin{proof}
The first step is to identify a suitable finite list of functions $\{f^j\}_{j}$ to which the existing theorem may be applied.
Without loss of generality, it may be  assumed that every connected component of $\domain$ intersects $E$ in a set of positive measure (by simply discarding any connected component of $\domain$ which does not and replacing $E$ by a compact subset $E'$ of equal measure contained only in those connected components which remain). Applying the lemma to this smaller $\domain$ and smaller compact set $E'$ inside $E$ yields $E_N$ and operators $X^\alpha$ which satisfy all the required relationships with $E$ itself rather than just the refined subset $E'$. Note that in discarding connected components of $\domain$, the dimension of $\mathcal F$ may decrease (if there happen to be functions $f$ which are not trivial but are identically zero on $E'$), but this poses no problems.

The quantity $||\cdot||_E$ defined by
\[ ||f||_E := \sup_{x \in E} |f(x)| \]
is a norm on $\mathcal F$; the only nonstandard step in establishing this claim is verifying that $||f||_E = 0$ implies $f = 0$. This must be the case because $||f||_E = 0$ means that $f$ vanishes on a set of positive measure within every connected component of $\domain$; since $f$ is real analytic, it must therefore vanish on the entirety of each connected component. 
Because $||\cdot||_E$ is a norm on the finite-dimensional vector space $\mathcal F$, it must be the case that the closed unit ball
\[ B := \set{f \in {\mathcal F}}{ ||f||_E \leq 1 } \]
is compact and contains a neighborhood of the zero function. If the dimension of $\mathcal F$ is $m$, let $\det$ be any nontrivial alternating $m$-linear functional on $\mathcal F$ (the set of choices is unique modulo multiplication by a nonzero real number). Compactness of $B$ and continuity of $\det$ on $B^m$ guarantees the existence of $f^1,\ldots,f^m \in W$ such that
\[ \sup_{g^1,\ldots,g^m \in B} |\det (g^1,\ldots,g^m)| = |\det (f^1,\ldots,f^m)| \]
and because $\det$ is nontrivial and $B$ contains an open ball around the origin, $|\det(f^1,\ldots,f^m)| > 0$. By Cramer's Rule, any $g \in B$ satisfies the identity
\[ g = \sum_{i=1}^m \frac{(-1)^{i-1} \det (g, f^1,\ldots,\widehat{f^i},\ldots,f^m)}{\det (f^1,\ldots,f^m)} f_i, \]
and because $g,f^1,\ldots,f^m \in B$, it follows by definition of $f^1,\ldots,f^m$ that the coefficient of $f^i$ has magnitude at most $1$ for each $i$. By rescaling, it follows that every $g \in \mathcal F$ satisfies
\begin{equation}  g = ||g||_E \sum_{i=1}^m c_i(g) f^i  \label{smallcoeffs} \end{equation}
for constants $c_i(g)$ which have magnitude at most $1$. Note also that as $f^i \in B$, it follows that $\sup_{x \in E} |f^i(x)| \leq 1$ for each $i$. Furthermore, each $f^i$ is, by assumption, a Nash function of complexity at most $K$.

Now let $f : \domain \rightarrow \R^{m+d}$ be the map with components $(f^1,\ldots,f^m,\epsilon t^1,\ldots, \epsilon t^d)$ for some fixed positive $\epsilon$ chosen sufficiently small that $|\epsilon t^i| \leq 1$ for all $t = (t^1,\ldots,t^d)$ belonging to $E$. The sets $\domain_1,\ldots,\domain_N$ and $E_1,\ldots,E_N$ and vector fields $\{X_i^{(j)}\}_{i=1}^d, j=1,\ldots,N$ are now taken to be exactly equal to the homonymous objects provided by Theorem \ref{mainineq}. The inequality \eqref{diffineq99} is the only conclusion which does not follow directly from that previous theorem. It is known, however, that for any $g \in \mathcal F$, the identity \eqref{smallcoeffs} implies that
\begin{align*}
\sup_{x \in \domain_N} |X^{\alpha} g(x)| & \leq ||g||_E \sup_{x \in \domain_N} \sum_{i=1}^m |c_i(g)| |X^\alpha f^i(x)| \\
& \leq ||g||_E \sum_{i=1}^m |c_i(g)| \sup_{x \in \domain_N} |X^\alpha f^i(x)| \leq m \sup_{x \in E} |g(x)|
\end{align*}
by virtue of \eqref{diffineq}. By induction on $k$, then,
\begin{align*}
\sup_{x_1 \in \domain_N} \cdots \sup_{x_{k-1} \in \domain_N} &  | X^{\alpha_1}_1 \cdots X^{\alpha_k}_k g(x_1,\ldots,x_k)| \\
& \leq m^{k-1} \sup_{x_1 \in E} \cdots \sup_{x_{k-1} \in E} | X^{\alpha_k}_k g(x_1,\ldots,x_k)|
\end{align*}
for each $x_k \in \domain$ because the maps $x_i \mapsto X^{\alpha_k}_k g(x_1,\ldots,x_k)$ belong to $\mathcal F$ when $i \neq k$ and $x_j$ is fixed for $j \neq i$ (simply because the maps  $x_i \mapsto g(x_1,\ldots,x_k)$ belong to $\mathcal F$, so derivatives with respect to $x_k$ must also). But furthermore
\[ |X_k^{\alpha_k} g(x_1,\ldots,x_k)| \leq m \sup_{x_k \in E} |g(x_1,\ldots,x_k)| \]
for each $x_1,\ldots,x_{k-1} \in E$ and $x_k \in \domain_N$, so 
\begin{align*}
\sup_{x_k \in \domain_N} \sup_{x_1 \in \domain_N} \cdots \sup_{x_{k-1} \in \domain_N} &  | X^{\alpha_1}_1 \cdots X^{\alpha_k}_k g(x_1,\ldots,x_k)| \\
& \leq m^{k-1} \sup_{x_1 \in E} \cdots \sup_{x_{k-1} \in E} \sup_{x_k \in \domain_N} | X^{\alpha_k}_k g(x_1,\ldots,x_k)| \\
& \leq m^{k} \sup_{x_1 \in E} \cdots \sup_{x_{k} \in E} | g(x_1,\ldots,x_k)|
\end{align*}
as asserted.
\end{proof}

\subsection{Proof of Theorem \ref{bigtheorem1}}
\label{actualproof}

\begin{proof}[Proof of Theorem \ref{bigtheorem1}]
Fix any choice of $\omega^1,\ldots,\omega^q$ with $|\det (\omega^1,\ldots,\omega^q)| = 1$ and let $W(t) := w(t) ||u(t)||_{\omega}^{-1/(\sigma p)}$. Note that the determinant condition on the $\omega^j$ implies that $\omega := \omega^1 \wedge \cdots \wedge \omega^q$ is, up to a possible choice of sign, independent of the individual $\omega^j$ which are chosen. For any compact set $E \subset \domain$,
\begin{equation} \sup_{t \in E} \frac{1}{||u(t)||_{\omega}} \left| \det \begin{bmatrix} u^1(t) \cdot \omega^{j_1} & \cdots & u^1(t) \cdot \omega^{j_p} \\ \vdots & \ddots & \vdots \\ u^p(t) \cdot \omega^{j_1} & \cdots & u^p(t) \cdot \omega^{j_p} \end{bmatrix} \right| \leq 1 \label{trivialbound} \end{equation}
for all choices of $j_1,\ldots,j_p \in \{1,\ldots,q\}$ (and in fact, the inequality simply holds for all $t \in \domain$ by definition of $||u(t)||_\omega$).

Let $U_0(t)$ be the smooth family of $p$-dimensional subspaces of $V^*$ which is parametrized by $(\domain,u)$ and suppose that $(U_0,V)$ admits a block decomposition.
By Lemma \ref{adaptedfactorization} and Proposition \ref{cutprop}, there must be some compact subset $E' \subset E$ contained in some open set $\domain'$ such that $W(E') \geq c_{p,q} W(E)$ such that
\[  \sup_{t \in E} \sup_{s_1,\ldots,s_p \in \domain'} \frac{1}{||u(t)||_{\omega}} \left| \det \begin{bmatrix} u^1(t) \cdot \overline{v}^{j_1,j'_1}(s_1) & \cdots & u^1(t) \cdot \overline{v}^{j_p,j'_p}(s_p) \\ \vdots & \ddots & \vdots \\ u^p(t) \cdot \overline{v}^{j_1,j'_1}(s_1) & \cdots & u^p(t) \cdot \overline{v}^{j_p,j'_p}(s_p) \end{bmatrix} \right| \leq c_{p,q} \]
where the $\overline{v}^{j,j'}$ form an adapted factorization of $\omega^1 \wedge \cdots \wedge \omega^q$. This follows from \eqref{trivialbound} by virtue of the simple fact that at each point $s \in \Omega'$, each $\overline{v}^{j,j'}$ is expressible as a linear combination of the $\omega^j$ with coefficients whose magnitude is bounded by some function of $p$ and $q$.

As it has been assumed that $u$ is Nash of complexity at most $K$,  Lemma \ref{adaptedfactorization} implies that the $\overline{v}^{j,j'}$ will be Nash as well with complexity at most $K^C$ for some $C$ depending only on $p$ and $q$. The component functions of the $\overline{v}^{j,j'}$ with respect to any standard basis of $V$ are then Nash functions which form a finite-dimensional vector space of functions $\mathcal F$ on $\domain$. By Lemma \ref{multilinear}, then, if $N$ is taken to be the maximum of $D_{ij}$ over all pairs $(i,j)$, then one has some compact set $E_N \subset E$ with $W(E_N) \geq c_{p,q,d,N} W(E)$ and families of smooth vector fields $X_i^{(N)}$ satisfying \eqref{basissize099} and \eqref{nondegenvecs99} such that the quantity
\[  \sup_{t,s_1,\ldots,s_p \in E_N} \frac{1}{||u(t)||_{\omega}} \left| \det \begin{bmatrix} u^1(t) \cdot X^{\alpha_1} \overline{v}^{j_1,j_1'}(s_1) & \cdots & u^1(t) \cdot X^{\alpha_p} \overline{v}^{j_p,j_p'}(s_p) \\ \vdots & \ddots & \vdots \\ u^p(t) \cdot X^{\alpha_1} \overline{v}^{j_1,j_1'}(s_1) & \cdots & u^p(t) \cdot X^{\alpha_p} \overline{v}^{j_p,j_p'}(s_p) \end{bmatrix} \right|  \]
is uniformly bounded above by some constant depending only on $p,q,d,$ and $N$
for all generalized multiindices $\alpha_j$ of order at most $N$ and generation at most $N$ (which is established by fixing any $t \in E_N$ and applying Lemma \ref{multilinear} to the variables $s_1,\ldots,s_p$).

Now by squaring and summing over all $(\alpha_k,j_k,j'_k)$ and restricting to the diagonal $s_1=\cdots=s_p=t$, the identity \eqref{ftgit} guarantees that 
\[  \sup_{t \in E_N} \frac{1}{||u(t)||_\omega} \inf_{\tilde u^1 (t) \wedge \cdots \wedge \tilde u^p(t) =u(t)} \left[ \sum_{i,\alpha,j,j'} |\tilde u^i(t) \cdot X^{\alpha} \overline{v}^{j,j'}(t)|^{2} \right]^\frac{p}{2} \leq C_{p,q,d,N} \]
for some constant $C$ depending only on $p,q,d$ and $N$. However, near any $t \in \domain$, there is always an adapted factorization $\overline{u}^{i,i'}(t)$ of $u$ which is expressible in terms of the $\tilde u^i$ with uniformly bounded coefficients. Choosing $\tilde u^1,\ldots,\tilde u^p$ which achieve the infimum at $t$ to within an error of at most $||u(t)||_\omega$ and then replacing them by the $\overline{u}^{i,i'}$ gives the following: for any $t \in E_N$, there are smooth adapted factorizations $\{\overline{u}^{i,i'}(t)\}_{i,i}$ of $u(t)$ and $\{\overline{v}^{j,j'}\}_{j,j'}$ of $V$ defined in a neighborhood of $t$ such that
\[ \frac{1}{||u(t)||_\omega}  \left[ \sum_{i,i',\alpha,j,j'} |\overline{u}^{i,i'}(t) \cdot X^{\alpha} \overline{v}^{j,j'}(t)|^{2} \right]^\frac{p}{2} \leq C_{p,q,d,N}'. \] 

Fixing any point $t \in E_N$ as above, let $M$ be any $d \times d$ matrix such that $(M \partial)_i = X^{(N)}_i$ for each $i=1,\ldots,d$. By \eqref{nondegenvecs99}, it follows that $W(t) |\det M| \geq c_{N,d,p,q} K^{-C} W(E)$ for some $C$ depending only on dimensions and $N$. Moreover, \eqref{basissize099} guarantees that the differential operators $(M \partial)^{\alpha}$ are always expressible as linear combinations of $X^{\beta}$ with uniformly bounded coefficients plus terms which are strictly lower order than $|\alpha|$. If $|\alpha| = D_{ij}$ for some given $i$ and $j$, those lower-order terms are annihilated in the sense that $\overline{u}^{i,i'}(t) \cdot [ (M \partial)^{\alpha} - X^\alpha] \overline{v}^{j,j'}(t) = 0$ (thanks to \eqref{annihilated}), and consequently
\[ |\overline{u}^{i,i'}(t) \cdot (M \partial)^\alpha \overline{v}^{j,j'}(t)|^{2} \]
is uniformly bounded above by the sum of over $\beta$ of $|\overline{u}^{i,i'}(t) \cdot X^{\beta} \overline{v}^{j,j'}(t)|^{2}$. It follows that
\begin{equation} \frac{1}{||u(t)||_\omega}  \left[ \sum_{i,i',j,j'} \sum_{|\alpha| = D_{ij}} \frac{|\overline{u}^{i,i'}(t) \cdot (M \partial)^{\alpha} \overline{v}^{j,j'}(t)|^{2}}{\alpha!} \right]^\frac{p}{2} \leq C_{p,q,d,N}''. \label{multiverse} \end{equation}
By the hypothesis \eqref{nondegenhyp}, since $W(t) |\det M| \geq c_{N,d,p,q} K^{-C} W(E)$, it must be the case that
\begin{align*} \left( c_{N,d,p,q} K^{-C} W(E) \right)^{\sigma p} & \leq  (W(t) |\det M|)^{\sigma p}   = \frac{1}{||u(t)||_\omega} \left[ |\det M|^\sigma (w(t))^\sigma \right]^p \\ & \leq \frac{1}{||u(t)||_\omega}  \left[ \sum_{i,i',j,j'} \sum_{|\alpha| = D_{ij}} \frac{|\overline{u}^{i,i'}(t) \cdot (M \partial)^{\alpha} \overline{v}^{j,j'}(t)|^{2}}{\alpha!} \right]^\frac{p}{2} \\ & \leq C''_{p,d,q,N}. \end{align*}
This implies that
\[ \int_{E} \frac{w(t)dt}{||u(t)||_\omega^{1/(\sigma p)}} \leq C'''_{p,d,q,N} K^C \]
for some constants $C'''_{p,d,q,N}$ and $C$ which depend only on $p,d,q$, and $N$. Because $E$ is an arbitrary compact set, we can exhaust $\domain$ by a sequence of such compact sets and the theorem will then hold as a consequence of the Lebesgue Dominated Convergence Theorem.
\end{proof}

\section{Useful Computational Tools}
\label{usefulsec}

In this section, we continue to assume that some smooth, nonvanishing decomposable $p$-form $u(t) \in \Lambda^p(V^*)$ has been defined for all $t \in \domain$ and that $V$ is equipped with a determinant functional. We assume that $U_0(t)$ are the $p$-dimensional subspaces parametrized by $(\domain,u)$ and that we have been given a block factorization of $(U_0,V)$ with spaces $U_0(t) \supset \cdots \supset U_{\mstar}(t) \supset U_{\mstar + 1}(t) = \{0\}$ and $V = V_0(t) \supset \cdots \supset V_m(t) \supset V_{m+1}(t) = \{0\}$ with pairings of elements in $U_i(t)$ and $V_j(t)$ having been assigned formal degree $D_{ij}$ for each pair $(i,j) \in \{0,\ldots,\mstar\} \times \{0,\ldots,m\}$. The goal now is to bridge the remaining gap between Theorem \ref{bigtheorem1} and the Geometric Invariant Theory tools from Section \ref{gitsec}.

\begin{definition}
A subset $\mathcal{T}$ of the index set $\{0,\ldots,\mstar\} \times \{0,\ldots,m\}$ is called a tile when it is a product of intervals. Given a tile $\mathcal T := [i_L,i_R] \times [j_L,j_R]$, let $\mathcal T^+$ be the tile obtained by expanding $\mathcal T$ by one index on the upper ends when possible, i.e.,
\[ \mathcal{T}^+ := [i_L,\min\{i_R+1,\mstar\}] \times [j_L,\min\{j_R+1,m\}]. \]
Index pairs in $\mathcal{T}^+ \setminus \mathcal{T}$ will be said to come immediately after $\mathcal T$.
\end{definition}
\begin{definition}
A tile $\mathcal T$ will be called a useful tile when the formal degrees of pairs $(i,j) \in \mathcal{T}^+ \setminus \mathcal{T}$ strictly dominate those of the adjacent elements in $\mathcal T$. More precisely, for every useful tile $\mathcal{T}$ and $(i,j) \in \mathcal{T}^+ \setminus \mathcal{T}$, if $(i-1,j) \in \mathcal{T}$, then $D_{ij} > D_{(i-1)j}$ and if $(i,j-1) \in \mathcal{T}$, then $D_{ij} > D_{i (j-1)}$.
\end{definition}
Note that the usefulness of a tile is a feature of both the block decomposition and the assigned formal degrees $D_{ij}$. Oftentimes it is possible to force a tile to be useful by decreasing the chosen values of $D_{ij}$ for certain index pairs. Decreasing the formal degrees $D_{ij}$ is not an entirely costless act, though, because one gets less information overall about the block decomposition.

Now suppose that, given the block decomposition, $u^{i,i'}(t)$ and $v^{j,j'}(t)$ are any adapted factorizations of $u(t)$ and $V$. For each useful tile $\mathcal T := I \times J \subset \{0,\ldots,\mstar\} \times \{0,\ldots,m\}$, consider the bilinear map
\begin{equation} P_{\mathcal T,t}(x,y,z) := \sum_{i \in I} \sum_{i' = 1}^{p_i} \sum_{j \in J} \sum_{j'=1}^{q_j} x_{i,i'} y_{j,j'} \sum_{|\alpha| = D_{ij}} \frac{z^\alpha}{\alpha!} u^{i,i'}(t) \cdot \partial^{\alpha} v^{j,j'}(t) \label{gitdef} \end{equation}
which sends pairs $x \in \R^{p_I}$ and $y \in \R^{q_J}$ (with $p_I := \sum_{i \in I} p_i$ and similarly for $q_J$) into the space of real polynomials of degree at most $\max_{(i,i) \in \mathcal{T}} D_{ij}$ in the variable $z \in \R^d$.

\begin{proposition}
Suppose that $u^{i,i'}(t)$ and $v^{j,j'}(t)$ are as given above. For any adapted factorizations $\overline{u}^{i,i'}(t)$ and $\overline{v}^{j,j'}(t)$ of $u(t)$ and $V$, there are positive real functions $\delta_0(t),\ldots,\delta_{\mstar}(t)$ and $\eta_0(t),\ldots,\eta_m(t)$ with $1 = \delta_0(t) \cdots \delta_{\mstar}(t)= \eta_0(t) \cdots \eta_m(t)$ such that every useful tile $\mathcal{T} := I \times J \subset \{0,\ldots,\mstar\} \times \{0,\ldots,m\}$ has
\begin{equation} \begin{split}  & \left( \sum_{i,i',j,j'} \sum_{|\alpha| = D_{ij}} \frac{|\overline{u}^{i,i'}(t) \cdot (M \partial)^{\alpha} \overline{v}^{j,j'}(t)|^{2}}{\alpha!} \right)^{\frac{1}{2}}  \\ & \qquad \qquad \geq |\det M|^{\sigma} \left( \prod_{i \in I} \delta_i(t) \right)^{1/p_I} \left( \prod_{j \in J} \eta_j (t) \right)^{1/q_J} |||P_{\mathcal T,t}|||_\sigma.
\end{split} \label{usefullower} 
\end{equation}
\end{proposition}
\begin{proof}
Given a tile $\mathcal T := I \times J$, restrict the sum on the left-hand side of \eqref{usefullower} to those $i \in I$ and $j \in J$. Applying Proposition \ref{scaleprop} pointwise to the $\overline{u}^{i,i'}$ and $u^{i,i'}$ and then to the $\overline{v}^{j,j'}$ and $v^{j,j'}$ gives the existence of $\delta_0(t),\ldots,\delta_{\mstar}(t)$ and $\eta_0(t),\ldots,\eta_m(t)$ which appear in this proposition. In particular, at every point $t$, there will be matrices $A(t)$ and $B(t)$ of determinant $\pm 1$ such that
\[ \overline{u}^{i,i'}(t) = \left( \prod_{i \in I} \delta_j (t) \right)^{1/p_J}  \sum_{\substack{(k,k')\\k \in I}} A^{(i,i')}_{(k,k')}(t) v^{k,k'}(t) + E_u^{i,i'}(t) \]
\[ \overline{v}^{j,j'}(t) = \left( \prod_{j \in J} \eta_j (t) \right)^{1/q_J}  \sum_{\substack{(k,k')\\k \in J}} B^{(j,j')}_{(k,k')}(t) v^{k,k'}(t) + E_v^{j,j'}(t) \]
where $(i,i')$ ranges over those pairs with $i \in I$, $(j,j')$ ranges over pairs with $j \in J$, and the unspecified higher-order terms $E_u^{i,i'}$ and $E_v^{i,i'}$ are vectors belonging to spaces $U_{i''}(t)$ and $V_{j''}(t)$ for $i''$ being greater than all elements in $I$ and $j''$ being greater than all elements of $J$. More precisely, $E^{i,i'}_u(t)$ belongs to $U_{i''}(t)$ for an $i''$ such that $(i'',j)$ comes immediately after $\mathcal T$ when $j \in J$. Likewise $E^{j,j'}_{v}(t)$ belongs to $V_{j''}(t)$ for some $j''$ such that $(i,j'')$ comes immediately after $\mathcal T$ for any $i \in I$. In particular, because $\mathcal T$ is a useful tile, 
\[ 0 = \overline{u}^{i,i'}(t) \cdot \partial^{\alpha} E_v^{j,j'}(t) = E_u^{i,i'} (t) \cdot \partial^{\alpha} \overline{v}^{j,j'}(t) \]
for all $|\alpha| = D_{ij}$. For any $d \times d$ matrix $M$, as $(M \partial)_k := \sum_{k'} M_{kk'} \partial_{k'}$, one must have that
\begin{align*}
\overline{u}^{i,i'}(t) \cdot (M \partial)^{\alpha} \overline{v}^{j,j'}(t) & = 
\sum_{k,k',\ell,\ell'} A^{(i,i')}_{(k,k')}(t) B^{(j,j')}_{(\ell,\ell')}(t) {u}^{k,k'}(t) \cdot (M \partial)^{\alpha} {v}^{\ell,\ell'}(t) \\ & = 
\sum_{k,k',\ell,\ell'} A^{(i,i')}_{(k,k')}(t) B^{(j,j')}_{(\ell,\ell')}(t) \partial_z^{\alpha} P_{\mathcal{T},t} (\basis_{k,k'},\basis_{\ell,\ell'}, M^T z)
\end{align*}
by \eqref{diagest2} (since any derivatives which fall on $B$ would yield a term which necessarily vanishes). Here $\basis_{k,k'}$ refers to the standard basis elements for $x$ in \eqref{gitdef}, i.e., $x = \sum_{k,k'} x_{k,k'} \basis_{k,k'}$. Likewise for $\basis_{\ell,\ell'}$ and the variable $y$.
It follows that the left-hand side of \eqref{usefullower} is bounded below by the scale factor
\[  \left( \prod_{i \in I} \delta_i(t) \right)^{1/p_I} \left( \prod_{j \in J} \eta_j (t) \right)^{1/q_J} \]
times the norm \eqref{stdnorm} of the map
\[ \tilde P_{\mathcal T,t}(x,y,z) = P_{\mathcal T,t} (A^T(t)x,B^T(t)y,M^Tz). \]
The proposition must then follow from the definition of $|||P_{\mathcal{T},t}|||_\sigma$.
\end{proof}

\begin{definition}
For each tile $\mathcal T := I \times J$ and any $\sigma > 0$, let $\mathbf{1}_I$ and $\mathbf{1}_J$ be the indicator functions of $I$ and $J$ and let $p_I := \sum_{i \in I} p_i$, $q_J := \sum_{j \in J} q_j$. Let $[\sigma]_{\mathcal T} \in \R^{\{0,\ldots,\mstar\}} \times \R^{\{0,\ldots,m\}} \times \R$ be the point given by $(p_I^{-1} \mathbf{1}_I; q_J^{-1} \mathbf{1}_J; \sigma)$ i.e.,
\[[\sigma]_{\mathcal T} := \left(0,\ldots,0,\frac{1}{p_I},\ldots,\frac{1}{p_I},0,\ldots,0;0,\ldots,0,\frac{1}{q_J},\ldots,\frac{1}{q_J},0,\ldots,0;\sigma \right) \]
where before the first semicolon, the nonzero entries are precisely those whose indices belong to $I$ and in the middle group separated by semicolons, the nonzero entries are those whose indices belong to $J$.
\end{definition}

By definition, the full tile $\mathcal{T} := [0,\mstar] \times [0,m]$ is always useful; often there are many subtiles which are useful as well. Theorem \ref{usefulthm} below adapts \eqref{usefullower} and relates it to Theorem \ref{bigtheorem1} to employ the Geometric Invariant Theory techniques of Section \ref{gitsec} as practical computational tools with which to study the fundamental integral \eqref{theintegral}.
\begin{theorem}
Suppose that $u^1(t),\ldots,u^p(t)$ are $V^*$-valued Nash functions on $\domain$ and that $U_0(t)$ is the smooth family of $p$-dimensional subspaces parametrized by $(\Omega,u := u^1 \wedge \cdots \wedge u^p)$. Let $V$ be $q$-dimensional and equipped with a nonvanishing alternating $q$-linear form called $\det$. Suppose that $(U_0,V)$ admits a block decomposition which is Nash and pick any $t_0 \in \domain$.  \label{usefulthm}
Using  $u^{i,i'}(t)$ and $v^{j,j'}(t)$ to denote any adapted factorizations of $u(t)$ and $V$, let $P_{\mathcal T,t_0}$ be defined as in \eqref{gitdef} for any useful tile $\mathcal T$.

Let $\mathcal{T}_1,\ldots,\mathcal{T}_k$ be some collection of useful tiles and suppose that $\mathcal{N}\subset \R^{\{0,\ldots,\mstar\}} \times \R^{\{0,\ldots,m\}} \times \R$ is the convex hull of the points $[\sigma_1]_{\mathcal{T}_1},\ldots,[\sigma_k]_{\mathcal{T}_k}$. If the point
\[ (p^{-1} \mathbf{1}_{[0,\mstar]};q^{-1} \mathbf{1}_{[0,m]} ;\sigma) \]
(i.e., the point whose coordinates are all $p^{-1}$ before the first semicolon and all $q^{-1}$ between the first and second semicolons)
belongs to $\mathcal{N}$, then there are nonnegative constants $\theta_1,\ldots,\theta_k$ summing to $1$ such that
\begin{equation} \int_{\domain} \frac{\left( \prod_{i=1}^k |||P_{\mathcal{T}_i,t}|||_{\sigma_i}^{\theta_i} \right)^{1/\sigma} dt}{\left[ ||u(t)||_{\omega} \right]^{1/(p \sigma)}} \leq C K^C \label{haveproved} \end{equation}
uniformly for all $\{\omega^j\}_{j=1}^q$ with $|\det(\omega^1,\ldots,\omega^q)| = 1$. 
The constant $C$ can be taken to depend only on $p$, $q$, $d$, and the $D_{ij}$, and $K$ is the maximal Nash complexity of the $u^i$ and the spaces $U_i(t)$ and $V_j(t)$ in the block decomposition. If, in particular, $|||P_{\mathcal{T}_i,t_0}|||_{\sigma_i} > 0$ for each $i$, then the weight function in the integral \eqref{haveproved} will be strictly positive on a neighborhood of $t_0$. 
\end{theorem}
\begin{proof}
The key observation is that if $\mathcal{T}_1,\ldots,\mathcal{T}_k$ are useful tiles such that there exists nonnegative $\theta_1,\ldots,\theta_k$ summing to $1$ for which
\[ \sum_{i=1}^k \theta_i [\sigma_i]_{\mathcal{T}_i} = (p^{-1} \mathbf{1}_{[0,\mstar]};q^{-1} \mathbf{1}_{[0,m]};\sigma),  \]
then the lower bound \eqref{usefullower} implies (after taking convex combinations) that
\begin{equation} \begin{split}  & \left( \sum_{i,i',j,j'} \sum_{|\alpha| = D_{ij}} |\overline{u}^{i,i'}(t) \cdot (M \partial)^{\alpha} \overline{v}^{j,j'}(t)|^{2} \right)^{\frac{1}{2}}  \\ & \qquad \qquad \geq |\det M|^{\sigma} \left( \prod_{i=0}^{\mstar} \delta_i(t) \right)^{1/p} \left( \prod_{j=0}^m \eta_j (t) \right)^{1/q} \prod_{i=1}^k |||P_{\mathcal{T}_i,t}|||^{\theta_i}_{\sigma_i} \\
& \qquad \qquad = |\det M|^{\sigma} \prod_{i=1}^k |||P_{\mathcal{T}_i,t}|||_{\sigma_i}^{\theta_i}.
\end{split} \label{usefullower2} 
\end{equation}
It follows that \eqref{nondegenhyp} holds with 
\[ w(t) := \left( \prod_{i=1}^k |||P_{\mathcal{T}_i,t}|||_{\sigma_i}^{\theta_i} \right)^{1/\sigma}, \]
which, by Lemma \ref{continuity}, must be a continuous function of $t$.
\end{proof}

Theorem \ref{usefulthm}, when combined with the results of \cite{testingcond}, has immediate implications for Radon-like transforms. What one obtains is a uniform bound for Radon-like transforms integrated against a measure which shares features with the affine arc length or affine hypersurface measures. This is made precise in Theorem \ref{usefulradon} below. It is natural to compare Theorem \ref{usefulradon} to the results of Christ, Dendrinos, Stovall, and Street \cite{cdss2020} for curves. Both have the property that one is able to construct a weight depending on the geometry such that the relevant bound holds uniformly across some reasonably broad range of operators. The main novelty of Theorem \ref{usefulradon} is that it holds for cases other than curves. An important limitation to note in comparison to Christ, Dendrinos, Stovall, and Street, however, is that the theorem below is a fundamentally ``one-sided'' result, meaning that it is sensitive to only certain extremal versions of the H\"{o}rmander spanning condition, while \cite{cdss2020} suffers from no such limitation. A generalization of \eqref{commutator} in Section \ref{suffsec}, for example, shows that all entries of the matrix $P$ constructed for Radon-like transforms can be expressed in terms of one-sided commutators $[Y^{\ell_1}, [ Y^{\ell_2}, \cdots [ Y^{\ell_r} , X^j] \cdots ] ]$.

\begin{theorem}
Let $\domain \subset \R^n \times \R^{n_1-k}$ be open and let $\phi(x,t)$, a map from $\domain$ into $\R^k$ (with $k < \min\{n_1,n\}$), have the property that $x \mapsto \phi(x,t)$ is Nash of complexity at most $K$ on its domain for all $t$ and $t \mapsto \nabla_x \phi(x,t)$ is Nash with complexity at most $K$ on its domain for all $x$.
For each $(x,s) \in \domain$, let \label{usefulradon}
\[ M(x,s) := \begin{bmatrix} \frac{\partial \phi^1}{\partial x_1}(x,s) & \cdots & \frac{\partial \phi^1}{\partial x_n}(x,s) \\ 
\vdots & \ddots & \vdots \\ 
\frac{\partial \phi^k}{\partial x_1}(x,s) & \cdots & \frac{\partial \phi^k}{\partial x_n}(x,s)\end{bmatrix} \]
and suppose that $M(x,s)$ has full rank $k$ at all points $(x,s) \in \domain$.
Suppose further that there exist families $A(x,s), B(x,s)$, and $P(x,s)$ of $k \times k$, $n \times n$, and $k \times n$, matrices, respectively, defined for all $(x,s) \in \domain$ such that
\begin{itemize}
\item Both $A(x,s)$ and $B(x,s)$ have $\det A(x,s) = \det B(x,s) = 1$ for all $(x,s) \in \domain$.
\item The entries $[A(x,s)]_{ij}$ and $[B(x,s)]_{ij}$ of $A(x,s)$ and $B(x,s)$ are Nash functions of complexity at most $K$.
\item The entries $[P(x,s)]_{ij}$ are homogeneous polynomials in the variable $z \in \R^{n_1-k}$ with the property that $\deg [P(x,s)]_{ij}$ is a constant function of $(x,s) \in \domain$ and is nondecreasing as a function of $i$ and $j$ separately.
\item For each $i \in \{1,\ldots,k\}$ and $j \in \{1,\ldots,n\}$,  all multiindices $\alpha,\beta$ with $|\alpha| + |\beta| \leq \deg [P(x,s)]_{ij}$ satisfy
\begin{equation} \partial^{\alpha}_s \partial^\beta_t \left[ A(x,s) M(x,s) B(x,t) - P(x,s)|_{z = t-s} \right]_{ij} = 0 \label{degreecond} \end{equation}
on the diagonal $t= s$.
\end{itemize}
Then the Radon-like transform given by
\[ T f(x) := \int f(t,\phi(x,t)) \left[ |||P(x,t)|||_{\sigma} \right]^{p/(1 + p \sigma)} dt \]
(where the integration is implicitly understood to be over the open set ${}^{x}\domain$ of points $t \in \R^{k}$ such that $(x,t) \in \domain$)
which maps functions on $\R^{n_1-k} \times \R^k$ into functions on $\R^n$ is bounded from $L^{1 + p \sigma}(\R^{n_1-k} \times \R^{k})$ to $L^{n(1 + p \sigma)/k}(\R^n)$ with a constant bounded by $C K^C$ for some constant $C$ which depends only on dimensions $n,n_1,k$ and the degrees $\deg [P(x,s)]_{ij}$.
\end{theorem}
\begin{proof}
Fix $r := 1 + p \sigma$ and let $r'$ be its H\"{o}lder dual exponent. By Theorem 1 of \cite{testingcond}, if one defines $u(x,s)$ to be the wedge product of the rows of $M(x,s)$, it suffices to show that
\begin{equation} \int \frac{ \left[ |||P(x,t)|||_{\sigma} \right]^{pr'/(1 + p \sigma)} dt}{||u(x,t)||_\omega^{r'-1}} \leq C K^C \label{required1} \end{equation}
uniformly for all $x$ and all choices of $\omega$ with $|\det(\omega^1,\ldots,\omega^n)| = 1$. Here one uses Proposition 1 of \cite{gressmancurvature} to see that the coarea measure as defined in \cite{testingcond} simply equals $dt$ for parametrizations of the form used here. A simple computation gives that $r' = 1 + (p \sigma)^{-1}$, and so \eqref{required1} becomes 
\begin{equation} \int \frac{ \left[ |||P(x,t)|||_{\sigma} \right]^{1/\sigma} dt}{[||u(x,t)||_\omega]^{1/(p\sigma)}} \leq C K^C. \label{required2} \end{equation}
Regard the parameter $x$ as fixed but arbitrary. The matrices $A$ and $B$ can be used to build a block decomposition as follows:  Let $\mstar := k-1$ and $m := n-1$. For each $i \in \{0,\ldots,\mstar\}$, let $U_i(s)$ be the span of rows $i+1,\ldots,k$ of the matrix $A(x,s) M(x,s)$ and for each $j \in \{0,\ldots,m\}$, let $V_j(t)$ be the span of columns $j+1,\ldots,n$ of the matrix $B(x,t)$. The dimensions of the spaces $U_i(s)$ and $V_j(t)$ decrease by exactly one when $i$ and $j$ increase. These families of subspaces are Nash of complexity at most $K^C$ simply because one can parametrize them by wedge products of rows of $A(x,s) M(x,s)$ and columns of $B(x,t)$ (and restricting the variable $x$ to be fixed can easily be seen to yield a Nash function of just the variables $t$ with complexity no greater than its complexity as a function of both $x$ and $t$). The condition \eqref{degreecond} guarantees that the spaces $U_i(s)$ and $V_j(t)$ form a block decomposition with the formal degree $D_{ij}$ simply taken to equal $\deg [P(x,s)]_{ij}$. This is because the $(i,j)$ entry of $A(x,s) M(x,s) B(x,t)$ is exactly equal to the pairing of row $i$ of $A(x,s) M(x,s)$ (which takes values in $U_i(s)$) with column $j$ of $B(x,t)$ (which takes values in $V_j(t)$). The condition \eqref{degreecond} guarantees that the pairing vanishes to the appropriate order on the diagonal $t=s$, and moreover it guarantees that $P$ must exactly equal the object \eqref{gitdef} defined above for the tile $[0,\mstar] \times [0,m]$. Applying Theorem \ref{usefulthm} for this single tile guarantees the required uniform bound \eqref{required2}.
\end{proof}

\section{Examples}
\label{examplesec}

\subsection{A concrete example}
Consider the $4$-form $u(s)$ on $\R^9$ for $s := (s_1,s_2) \in \R^2$ which is given by the wedge product of the rows of the following matrix:
\[ M(s) := \begin{bmatrix}
1 & 0 & 0 & 0 & s_1 & 0    & 0     & 0    & s_1^2 \\
0 & 1 & 0 & 0 & 0    & s_2 & 0     & 0    & s_2^2 \\
0 & 0 & 1 & 0 & 0    & 0     & s_1 & 0    & s_1^3 \\
0 & 0 & 0 & 1 & 0    & 0     & 0     & s_2 & s_2^3
\end{bmatrix}. \]
To construct a block decomposition compatible with this $u$, one can perform a series of elementary row and column operations on $M(s)$. The rules are the following:
\begin{enumerate}
\item Any row $i'$ may be multiplied by any (Nash) function of $s$ and the result may be added to any row $i \neq i'$.
\item Any column $j$ may be multiplied by any (Nash) function of $t$ and the result may be added to any column $j \neq j'$.
\end{enumerate}
All matrices attainable in this way are expressible in the form $A(s) M(s) B(t)$ for (Nash) matrices $A(s)$ and $B(t)$ with determinant $1$. The goal is to perform operations which eliminate monomials in the variable $(s-t)$ to obtain a matrix whose entries vanish to maximum possible order on the diagonal $t=s$. The rows of $A(s) M(s)$ can then be understood as the vectors $u^{i,i'}$ in an adapted factorization of $u$ and the columns of $B(t)$ represent the vectors $v^{j,j'}$ in an adapted factorization of $V$.

To that end, let $z := s-t$ in the computations below. 
The first series of steps begins by multiplying column 1 by $-t_1$ and adding the result to column 5. Next, multiply column $1$ by $-t_1^2$ and add this to column $9$. In each case, choose coefficients to force the entries of the new matrix to vanish on the diagonal. Repeating in the analogous way using columns 2, 3, and 4 to modify columns 6 through 9 gives the result
\[ \begin{bmatrix}
1 & 0 & 0 & 0 & z_1 & 0    & 0     & 0     & (s_1+t_1) z_1\\
0 & 1 & 0 & 0 & 0    & z_2 & 0     & 0     & (s_2+t_2) z_2\\
0 & 0 & 1 & 0 & 0    & 0     & z_1 & 0     & (s_1^2+s_1t_1+t_1^2) z_1\\
0 & 0 & 0 & 1 & 0    & 0     & 0     & z_2 & (s_2^2+s_2t_2 + t_2^2) z_2
\end{bmatrix}. \]
Now continue: the entries of column columns 5 through 9 vanish to first order on the diagonal. We may use columns 5 through 8 to eliminate the first-order terms in $z$ appearing in column 9: columns 5, 6, 7, and 8 should be multiplied by $-2t_1,-2t_2,-3t_1^2,$ and $-3t_2^2$, respectively, and the results added to column 9 to obtain
\[ \begin{bmatrix}
1 & 0 & 0 & 0 & z_1 & 0    & 0     & 0     & z_1^2\\
0 & 1 & 0 & 0 & 0    & z_2 & 0     & 0     & z_2^2\\
0 & 0 & 1 & 0 & 0    & 0     & z_1 & 0     & (s_1 + 2t_1)z_1^2\\
0 & 0 & 0 & 1 & 0    & 0     & 0     & z_2 & (s_2 + 2 t_2) z_2^2
\end{bmatrix}. \]
One can further eliminate $z$ using row operations: one should multiply rows 1 and 2 by $-3s_1$ and $-3s_2$, and add the results to rows $3$ and $4$. The result is that
\[ A(s) M(s) B(t) = \begin{bmatrix}
1 & 0 & 0 & 0 & z_1 & 0    & 0     & 0     & z_1^2\\
0 & 1 & 0 & 0 & 0    & z_2 & 0     & 0     & z_2^2\\
-3s_1 & 0 & 1 & 0 & -3s_1z_1    & 0     & z_1 & 0     & -2z_1^3\\
0 & -3s_2 & 0 & 1 & 0    & -3s_2z_2   & 0     & z_2 & -2z_2^3
\end{bmatrix} \]
for some algebraic, determinant $1$ matrices $A(s)$ and $B(t)$. Now let $U_0(s)$ be the span of the rows of $A(s) M(s)$ and let $U_1(s)$ be the span of the final two rows of $A(s) M(s)$. Let $V_0(t) := \R^{9}$, let $V_1(t)$ be the span of columns 5 through 9 of $B(t)$, and let $V_2(t)$ be the span of column 9 of $B(t)$. The computation just completed establishes that this gives a block decomposition of $(U_0,\R^9)$ when $D_{00} := D_{10} := 0$, $D_{01} := D_{11}: = 1$, $D_{02} := 2$, and $D_{12} := 3$. There are four useful tiles which are homogeneous as functions of $z$: $\mathcal{T}_{0} := [0,1] \times [0,0], \mathcal{T}_1 := [0,1] \times [1,1], \mathcal{T}_2 := [0,0] \times [2,2]$, and $\mathcal{T}_3 := [1,1] \times [2,2]$. The maps associated to these tiles are exactly
\[
P_{\mathcal{T}_0,t}(z)  := \begin{bmatrix} 1 & 0 & 0 & 0 \\ 0 & 1 & 0 & 0 \\ -3t_1 & 0 & 1 & 0 \\ 0 & -3t_2 & 0 & 1 \end{bmatrix} , \ \ \ P_{\mathcal{T}_1,t}(z)  := \begin{bmatrix} z_1 & 0 & 0 & 0 \\ 0 & z_2 & 0 & 0 \\ -3t_1 z_1 & 0 & z_1 & 0 \\ 0 & -3t_2 z_2 & 0 & z_2 \end{bmatrix},  \]
\[P_{\mathcal{T}_2,t}(z)  := \begin{bmatrix} z_1^2 \\ z_2^2 \end{bmatrix}, \ \ 
P_{\mathcal{T}_3,t}(z)  := \begin{bmatrix} -2 z_1^3 \\ -2 z_2^3 \end{bmatrix}.
\]
Now notice that $P_{\mathcal{T}_0,t}$ differs from the $4 \times 4$ identity
by the action of a determinant one matrix on the row space, which means for the purposes that follow, $P_{\mathcal{T}_1,t}$ may be replaced by an identity matrix without changing the magnitude of $|||P_{\mathcal{T}_1,t}|||_\sigma$. The same argument applies to $P_{\mathcal{T}_1,t}$ as well and establishes that $|||P_{\mathcal{T}_1,t}|||_\sigma = |||P_{\mathcal{T}_1,0}|||_\sigma$ for all $t$. With this reduction, it is easy to check that each map is in the form identified by Lemma \ref{standardform}, and as a consequence $|||P_{\mathcal{T}_i,t}|||_{i / 2}$ is nonzero and constant as a function of $t$ for each $i\in \{0,1,2,3\}$. By Theorem \ref{usefulthm}, we construct the convex hull $\mathcal{N}$ of the points
\begin{align*}
[0]_{\mathcal{T}_0} & := \left( \frac{1}{4}, \frac{1}{4}; \frac{1}{4}, 0, 0 ; 0 \right), &
\left[\frac{1}{2} \right]_{\mathcal{T}_1} & := \left( \frac{1}{4}, \frac{1}{4}; 0 , \frac{1}{4}, 0 ; \frac{1}{2} \right), \\
\left[1\right]_{\mathcal{T}_2} &  := \left( \frac{1}{2}, 0 ; 0 , 0, 1 ; \frac{2}{2} \right), &
\left[\frac{3}{2}\right]_{\mathcal{T}_3} & := \left( 0 , \frac{1}{2} ; 0 , 0, 1 ; \frac{3}{2} \right),
\end{align*}
and observe that it contains the point
\[ \left( \frac{1}{4}, \frac{1}{4} ; \frac{1}{9}, \frac{1}{9}, \frac{1}{9}; \frac{13}{36} \right) \]
by multiplying the points $[i/2]_{\mathcal{T}_i}$ by $\frac{4}{9}, \frac{4}{9}, \frac{1}{18}$, and $\frac{1}{18}$, respectively. Thus Theorem \ref{usefulthm} establishes that there exists a finite constant $C$ such that
\[ \int_{\R^2} \frac{dt}{\left[ ||u(t)||_{\omega} \right]^{9/13}} \leq C \]
for all bases $\omega := \{\omega^1,\ldots,\omega^9\}$ of $\R^9$ with determinant $1$.

\subsection{Radon-like Transforms}

We consider the testing conditions for two closely-related families of Radon-like transforms. 
\begin{definition}
Let $\mathcal A$ be some finite collection of nonzero multiindices on $\R^d$.  Let $N$ denote the cardinality of $\mathcal{A}$. The collection $\mathcal{A}$ will be called balanced of type 1 when it satisfies the following two criteria:
\begin{itemize}
\item For any two nonzero multiindices $\alpha',\alpha''$, if $\alpha'+\alpha'' \in \mathcal{A}$, then $\alpha',\alpha'' \in \mathcal{A}$. 
\item The collection $\mathcal{A}$ satisfies
\[ \frac{1}{N} \sum_{\alpha \in \mathcal{A}} \alpha = \sigma \mathbf{1}_d \]
for some $\sigma > 0$.
\end{itemize}
If $\mathcal A$ contains no monomials $\alpha$ with $|\alpha| = 1$, then $\mathcal{A}$ will be called balanced of type 2 when it satisfies both of the following criteria:
\begin{itemize}
\item For any two nonzero multiindices $\alpha',\alpha''$, if $\alpha'+\alpha'' \in \mathcal{A}$ and $|\alpha'| \geq 2$, then $\alpha'\in \mathcal{A}$.
\item The collection $\mathcal{A}$ satisfies
\[ \frac{1}{N} \sum_{\alpha \in \mathcal{A}} \frac{\alpha}{|\alpha|} = d^{-1} \mathbf{1}_d \text{ and } \frac{1}{N} \sum_{\alpha \in \mathcal{A}} \alpha= (\sigma + d^{-1}) \mathbf{1}_d \]
for some $\sigma > 0$.
\end{itemize}
\end{definition}

For any positive integer $k$, let $\Lambda := \mathcal{A} \times \{1,\ldots,k\}$ and let $\{x_{\alpha,i}\}_{(\alpha,i) \in \Lambda}$ denote elements of $\R^{\Lambda}$. Consider the (non-translation-invariant) Radon-like transform $T_1$ given by
\[ T_1 f(x,x') := \int_{\R^d} f\left(\left\{x_{\alpha,i} + \frac{(y')^{\alpha} x'_i}{\alpha!}\right\}_{(\alpha,i) \in \Lambda},y'\right) dy' \]
which sends functions $f(\{y_{\alpha,i}\}_{(\alpha,i) \in \Lambda},y')$ defined on $\R^{\Lambda} \times \R^d$ to functions defined on $\R^{\Lambda} \times \R^{k}$ and the (translation-invariant) Radon-like transform $T_2$ given by
\[ T_2 f(x,x') := \int_{\R^d} f \left( \left\{x_{\alpha} + \frac{(x'-y')^{\alpha}}{\alpha!} \right\}_{\alpha \in \mathcal{A}},y'\right) dy' \]
which sends functions on $\R^{\mathcal A} \times \R^d$ to functions on $\R^{\mathcal A} \times \R^d$. The main result is as follows.
\begin{theorem}
If $\mathcal A$ is balanced of type 1, then the Radon-like transform $T_1$ maps $L^r(\R^{\Lambda} \times \R^{d})$ to $L^{r (N+1)/N} (\R^{\Lambda} \times \R^k)$ for $r := \frac{N k \sigma}{N+1}+ 1$.
If $\mathcal A$ is balanced of type 2, then $T_2$ maps $L^r(\R^{\mathcal A} \times \R^{d})$ to $L^{r (N+d)/N}(\R^{\mathcal A} \times \R^{d})$ for $r := \frac{N \sigma d}{N+d} + 1$. \label{radonexample}
\end{theorem}

To apply the testing condition from \cite{testingcond} to $T_1$, one must consider the form $u_1(s)$ given by the wedge of the rows of the matrix
\[ M_1(s) :=  \begin{bmatrix} I & 0 & \cdots & 0 & \frac{s^{\alpha_1}}{\alpha_1!} I \\ 0 & I & \ddots & \vdots & \frac{s^{\alpha_2}}{\alpha_2!} I \\ \vdots & \ddots & \ddots & 0 & \vdots \\ 0 & \cdots & 0 & I & \frac{s^{\alpha_N}}{\alpha_N!} I \end{bmatrix} \]
where the entries written above are to be understood as $k \times k$ blocks and $\alpha_1,\ldots,\alpha_N$ is an enumeration of $\mathcal{A}$. For $T_2$, the form $u_2(s)$ is the wedge of rows of the matrix $M_2(s)$ given by
\[ M_2(-s) :=  \begin{bmatrix} 1 & 0 & \cdots & 0 & \frac{\partial}{\partial s_1} \frac{s^{\alpha_1}}{\alpha_1!} & \cdots & \frac{\partial}{\partial s_d} \frac{s^{\alpha_1}}{\alpha_1!}  \\ 0 & 1 & \ddots & \vdots & \frac{\partial}{\partial s_1} \frac{s^{\alpha_2}}{\alpha_2!} & \cdots & \frac{\partial}{\partial s_d} \frac{s^{\alpha_2}}{\alpha_2!}  \\ \vdots & \ddots & \ddots & 0 & \vdots & \ddots & \vdots  \\ 0 & \cdots & 0 & 1 & \frac{\partial}{\partial s_1} \frac{s^{\alpha_N}}{\alpha_N!} & \cdots & \frac{\partial}{\partial s_d} \frac{s^{\alpha_N}}{\alpha_N!}   \end{bmatrix}. \]
By an elimination procedure of the same sort used in the previous example, there exist upper triangular matrices $B_1(t)$ and $B_2(t)$ whose diagonal entries are all $1$ and whose non-diagonal entries are simply monomials in $t$ (with various constant coefficients) such that
\[ M_1(s) B_1(t)  = \begin{bmatrix} I & 0 & \cdots & 0 & \left(\frac{s^{\alpha_1}}{\alpha_1!} - \frac{t^{\alpha_1}}{\alpha_1!} \right) I \\ 0 & I & \ddots & \vdots & \vdots \\ \vdots & \ddots & \ddots & 0 & \vdots \\ 0 & \cdots & 0 & I & \left( \frac{s^{\alpha_N}}{\alpha_N!} - \frac{t^{\alpha_N}}{\alpha_N!} \right) I \end{bmatrix} \]
and such that $M_2(s) B_2(t)$ consists of an $N \times N$ identity block on the left (where $|\mathcal A| = N$) followed by an $N \times d$ block on the right whose $(i,\ell)$ entry is $(-1)^{|\alpha|-1} \left[ \frac{\partial}{\partial s_\ell} \frac{s^{\alpha_i}}{\alpha_i!} - \frac{\partial}{\partial t_\ell} \frac{t^{\alpha_i}}{\alpha_i!} \right]$.

Now for any homogeneous polynomial $q$ of degree $m > 0$, Taylor's Theorem guarantees that
\begin{equation} \sum_{|\beta| \leq m} \frac{(-s)^{\beta}}{\beta!} \left[ \partial^{\beta} q (s) - \partial^\beta q(t) \right] = q(s-s) - q(t-s) = - q(t-s). \label{taylor} \end{equation}
Moreover, since $\partial^\beta q(s) - \partial^\beta q(t)$ is identically zero when $|\beta| = m$ (as it is a difference of two equal constants), one may omit terms $\beta$ with $|\beta| = m$ and sum only over those $\beta$ with $|\beta| < m$. For a balanced set $\mathcal{A}$ of type 1, every $\alpha \in \mathcal{A}$ has the property that any nonzero $\alpha'$ which is entrywise less than or equal to $\alpha$ belongs to $\mathcal{A}$. This means that when $\alpha \in \mathcal{A}$ and $\partial^{\beta} (s^{\alpha} /\alpha!)$ is not constant, $\alpha - \beta$ belongs to $\mathcal{A}$ as well. For a balanced set $\mathcal{A}$ of type $2$, this means that if $\alpha \in \mathcal{A}$ and $\partial^{\beta} \left[ \nabla (s^{\alpha}/\alpha!)\right]$ is not constant (where $\nabla$ is the standard gradient), then $\alpha - \beta \in \mathcal{A}$. In both cases, this means that when one ignores the initial identity block of $M_1(s) B_1(t)$ and $M_2(s) B_2(t)$, every row of the remaining block has the form $r(s) - r(t)$ (where $r(u) := u^\alpha / \alpha!$ for $M_1(s) B_1(t)$ and $r(u) = (-1)^{|\alpha|-1} \nabla_u (u^\alpha/\alpha!)$ for $M_2(s) B_2(t)$) and furthermore the block has another row which equals $\partial^{\beta} r(s) - \partial^{\beta} r(t)$ (up to a factor of $\pm 1$) provided that $\partial^{\beta} r(s) - \partial^{\beta} r(t)$ isn't itself identically zero. Thus the identity \eqref{taylor} applied to the rows of $M_1(s) B_1(t)$ and $M_2(s) B_2(t)$ implies that there must be lower triangular matrices $A_1(s)$ and $A_2(s)$ whose diagonal entries are all ones and whose non-diagonal entries are all monomials in $s$ such that
\begin{itemize}
\item The matrix $A_1(s) M_1(s) B_1(t)$ consists of an $N k \times N k$ block on the left which equals $A_1(s)$ and an $Nk \times Nk$ block on the right which equals 
\[ \begin{bmatrix} - \frac{(t-s)^{\alpha_1}}{\alpha_1!} I_{k\times k} \\ \vdots \\ - \frac{(t-s)^{\alpha_N}}{\alpha_N!} I_{k \times k} \end{bmatrix}, \]
\item The matrix $A_2(s) M_2(s) B_2(t)$ consists of an $N \times N$ block on the left which equals $A_2(s)$ and an $N \times d$ block on the right whose $(\ell,i)$ entry is $-\partial_{s_i} (s-t)^{\alpha_\ell} / {\alpha_\ell}!$.
\end{itemize}
We arrive at two polynomial-valued matrices: 
\[ P_1(z) := \begin{bmatrix} - \frac{z^{\alpha}_1}{\alpha_1!} I_{k\times k} \\ \vdots \\ - \frac{z^{\alpha_N}}{\alpha_N!} I_{k \times k} \end{bmatrix} \]
and $P_2(z)$, whose entries are given by
\[ [P_2(z)]_{ij} := -\frac{\partial}{\partial z_j} \frac{z^{\alpha_i}}{\alpha_i!}. \]
For both matrices, no monomial appears more than once in any row or column and no two monomials appear together in the same entry.  If $A_1$ denotes the triples $(i,j,\alpha)$ for which $[\partial^{\alpha} P_1]_{ij}(0) \neq 0$, then $|A_1| = Nk$ and
\[ \sum_{(i,j,\alpha) \in A_1} \frac{1}{Nk} (\basis^i; \basis^j; \alpha) = \left( (Nk)^{-1} \mathbf{1}_{Nk}; k^{-1} \mathbf{1}_k; \sigma \mathbf{1}_d \right) \]
because $\mathcal A$ is balanced of type 1. For $P_2$,
\begin{align*} \sum_{i=1}^N \sum_{j=1}^d \frac{(\alpha_i)_j}{N |\alpha_i|} \left( \basis^i; \basis^j; \alpha - \basis^j\right) & = \left(N^{-1} \mathbf{1}_N; \frac{1}{N} \sum_{\alpha \in \mathcal{A}} \frac{\alpha}{|\alpha|} ; \frac{1}{N} \sum_{\alpha \in \mathcal{A}} \frac{(|\alpha|-1) \alpha}{|\alpha|} \right)\\
& = \left(N^{-1} \mathbf{1}_N; d^{-1} \mathbf{1}_d ; \sigma \mathbf{1}_d \right)
\end{align*}
because $\mathcal A$ is balanced of type 2. (Note the subtle but important point that if $\partial^{\alpha} [P_2]_{ij} = 0$ then the corresponding coefficient in the linear combination above is zero so that it does not in reality need to be included in the summation.) It follows from Lemma \ref{standardform} that $||| P_1 |||_\sigma$ and $|||P_2 |||_\sigma$ are strictly positive.

To apply Theorem \ref{usefulthm}, in both cases one splits the block decomposition implicitly constructed into one useful tile $\mathcal{T}_0$, of size $Nk \times Nk$ in the first case and $N \times N$ in the second case, which has pairings of formal degree $0$ and consists of a lower triangular matrix which equals the identity on the diagonal in our default adapted factorizations and a second useful tile $\mathcal{T}_1$ consisting of the parts corresponding to $P_1$ and $P_2$ in the standard adapted factorizations. The observations that 
\begin{align*}
\frac{Nk}{Nk+k} & \left( (Nk)^{-1}; (Nk)^{-1}, 0; 0 \right)  + \frac{k}{Nk+k} \left( (Nk)^{-1}; 0, k^{-1}; \sigma \right) \\ & = \left( (Nk)^{-1}; (Nk+k)^{-1}, (Nk+k)^{-1}; \frac{\sigma}{N+1} \right), \\
\frac{N}{N+d} & \left( N^{-1}; N^{-1}, 0; 0 \right)  + \frac{d}{N+d} \left( N^{-1}; 0, d^{-1}; \sigma\right) \\
& = \left( N^{-1}; (N+d)^{-1}, (N+d)^{-1}; \frac{\sigma d}{N+d} \right),
\end{align*}
allow one to apply Theorem \ref{usefulthm} to conclude that
\[ \int_{\R^d} \frac{dt}{[ ||u_1(t)||_{\omega}]^{(N+1)/(N k \sigma)}} \leq C \text{ and } \int_{\R^d} \frac{dt}{ [ ||u_2(t)||_\omega]^{(N+d)/(N \sigma d)}} \leq C \]
uniformly in $\omega$. The testing condition of \cite{testingcond} implies the bounds presented in Theorem \ref{radonexample} (after a variation of Proposition 1 from \cite{gressmancurvature} establishes that $dt$ agrees with the coarea measure in both cases).

It is worth noting that the family of translation-invariant examples $T_2$ includes the moment curve, for which restricted weak-type bounds in all dimensions were first established by Christ \cite{christ1998}. Later work of Stovall \cite{stovall2010}, building on results of Dendrinos and Wright \cite{dw2010}, achieved endpoint $L^p$-improving inequalities as a byproduct of a more general theorem concerning polynomial curves equipped with affine arc length measure. Recent work of Christ, Dendrinos, Stovall, and Street \cite{cdss2020} pushes these ideas even further.

Using these same techniques, it is also possible to establish $L^p$-improving properties for the overdetermined Radon-like transform
\[ T_3 f(x_0,\{x_\alpha\}_{\alpha\in \mathcal{A}}) := \int_{\R^d} f \left(x_0 + \sum_{\alpha \in \mathcal{A}} x_\alpha \frac{y^\alpha}{\alpha!}, y \right) dy \]
where $x_0,x_\alpha \in \R^k$ for sets of monomials $\mathcal{A}$ that are balanced of type 1. Operators of this sort previously appeared in \cite{gressman2007}, which used a tensor power technique due to Carbery \cite{carbery2004} to achieve Lebesgue space endpoint results.

\subsection{A degenerate example} 
\label{degensec}

Consider the $4$-form $u(s)$ on $\R^8$ for $s := (s_1,s_2,s_3) \in \R^3$ which equals the wedge product of the rows of the following matrix:
\[ M(s) := \begin{bmatrix} 1 & 0 & 0 & 0 & 0 & s_2 & s_3 & s_1 + \frac{1}{2} (s_2^2+s_3^2) \\ 0 & 1 & 0 & 0 & s_2 & s_1 &0 & s_1 s_2 \\ 0 & 0 & 1 & 0 & s_1 & s_3 & s_2 & s_2 s_3 + \frac{1}{2} s_1^2 \\ 0 & 0 & 0 & 1 & s_3 & 0 & s_1 & s_1 s_3\end{bmatrix}. \]
Observe that each column has the property that its partial derivatives with respect to any one of $s_1,s_2,s_3$ are expressible as constant linear combinations of columns appearing to the left of it. It follows by Taylor's Theorem that
\[ M(s) \begin{bmatrix}
1 & 0 & 0 & 0 & 0 & -t_2 & -t_3 & -t_1 + \frac{1}{2}(t_2^2+t_3^2)\\
0 & 1 & 0 & 0 & -t_2 & -t_1 & 0 & t_1 t_2\\
0 & 0 & 1 & 0 & -t_1 & -t_3 & -t_2 & t_2 t_3 + \frac{1}{2} t_1^2\\
0 & 0 & 0 & 1 & -t_3 & 0 & -t_1 & t_1 t_3\\
0 & 0 & 0 & 0 & 1 & 0 & 0 & -t_1\\
0 & 0 &0 & 0 & 0 & 1 & 0 & -t_2\\
0 & 0& 0 & 0 & 0 & 0 & 1 & -t_3\\
0 & 0 & 0 &0 & 0 & 0 & 0 & 1
\end{bmatrix} = M(s-t) \]
for each $s, t \in \R^3$. Letting $U_0(s)$ be the span of the rows of $M(s)$ in the usual way, one therefore obtains a block decomposition of $(U_0(s),\R^8)$ by letting $U_i(s)$ be the span of rows $i+1,\ldots,4$ for $i=0,\ldots,3$ and letting $V_j(t)$ be the span of columns $j+1,\ldots,8$ for $j=0,\ldots,7$. The formal degree $D_{ij}$ of a row and column pair can be defined to equal the degree of the lowest-order nonzero term in the polynomial $M_{ij}(s)$ when one exists and then handling the trivially zero entries by defining $D_{ij} = 0$ when $j \leq 4$ and $D_{ij} = 1$ when $j \in \{5,6,7\}$. Now let 
\[ P(z) := \begin{bmatrix} 
1 & 0 & 0 & 0 & 0 & -z_2 & -z_3 & -z_1  \\ 0 & 1 & 0 & 0 & -z_2 & -z_1 &0 & z_1 z_2 \\ 0 & 0 & 1 & 0 & -z_1 & -z_3 & -z_2 & z_2 z_3 + \frac{1}{2} z_1^2 \\ 0 & 0 & 0 & 1 & -z_3 & 0 & -z_1 & z_1 z_3\end{bmatrix}. \]
This $P(z)$ agrees with $M(-z)$ in all places except the last entry of the first row, where one sees that the second-order terms of $M$ for that entry have been omitted in $P$. When Theorem \ref{usefulthm} is applied to this $u$ using the block decomposition just identified, the right-hand side of \eqref{nondegenhyp} is bounded below by $|||P|||_{\sigma}$ because $[0,3] \times [0,7]$ is a useful tile.
For appropriate values of $\sigma$, one can use Lemma \ref{standardform} to show that $|||P|||_{\sigma} > 0$. This $P$ clearly satisfies the required structural properties: no monomial appears more than once in any row or column, and the monomials $z_2 z_3$ and $\frac{1}{2} z_1^2$ have multiindices $(0,1,1)$ and $(2,0,0)$ such that $(0,1,1) + \basis^{k'} \neq (2,0,0) + \basis^{k}$ for any $k,k'$. To find points of the form \eqref{convex} in the convex hull of the vectors $(\basis^i;\basis^j;\alpha)$ for $\partial^{\alpha}P_{ij}(0) \neq 0$, it is helpful to first define four auxiliary points:
\begin{align*}
I_1 := \frac{1}{4} & \sum_{i=1}^4 (\basis^i;\basis^i;0,0,0) = \left( \frac{1}{4} \mathbf{1}_4; \frac{1}{4}, \frac{1}{4}, \frac{1}{4}, \frac{1}{4},0,0,0,0; 0,0,0 \right), \\
I_2 := \frac{1}{8} & \Big((\basis^1;\basis^6;0,1,0) + (\basis^1;\basis^7;0,0,1) + (\basis^2;\basis^5;0,1,0) \\
 & + (\basis^2;\basis^6;1,0,0) + (\basis^4;\basis^5;0,0,1) + (\basis^4;\basis^7;1,0,0) \Big) \\ & + \frac{1}{4} (\basis^3;\basis^8;2,0,0) = \left(\frac{1}{4} \mathbf{1}_4; 0,0,0,0, \frac{1}{4}, \frac{1}{4}, \frac{1}{4}, \frac{1}{4}; \frac{3}{4}, \frac{1}{4}, \frac{1}{4} \right), \\
 I_3 := \frac{1}{8} & \Big((\basis^1;\basis^6;0,1,0) + (\basis^1;\basis^7;0,0,1) + (\basis^2;\basis^5;0,1,0) \\
 & + (\basis^2;\basis^6;1,0,0) + (\basis^4;\basis^5;0,0,1) + (\basis^4;\basis^7;1,0,0) \Big) \\ & + \frac{1}{4} (\basis^3;\basis^8;0,1,1) = \left(\frac{1}{4} \mathbf{1}_4; 0,0,0,0, \frac{1}{4}, \frac{1}{4}, \frac{1}{4}, \frac{1}{4}; \frac{1}{4}, \frac{1}{2}, \frac{1}{2} \right), \\
 I_4 :=  \frac{1}{4} & \Big( (\basis^1;\basis^8;1,0,0)  + (\basis^2;\basis^5;0,1,0) + (\basis^3;\basis^6,0,0,1) + (\basis^4;\basis^7;1,0,0) \Big) \\ & = \left( \frac{1}{4} \mathbf{1}_4; 0,0,0,0, \frac{1}{4}, \frac{1}{4}, \frac{1}{4}, \frac{1}{4}; \frac{1}{2}, \frac{1}{4}, \frac{1}{4} \right).
\end{align*}
Now
\[ \frac{1}{2} I_1 + \frac{1}{6} I_2 + \frac{1}{3} I_3 = \left( \frac{1}{4} \mathbf{1}_4; \frac{1}{8} \mathbf{1}_8; \frac{5}{24} \mathbf{1}_{3} \right) \]
and
\[ \frac{1}{2} I_1 + \frac{1}{4} I_3 + \frac{1}{4} I_4 =  \left( \frac{1}{4} \mathbf{1}_4; \frac{1}{8} \mathbf{1}_8; \frac{3}{16} \mathbf{1}_{3} \right), \]
so by \eqref{convex} it must be the case that $|||P|||_{\sigma} > 0$ for all $\sigma \in [\frac{3}{16},\frac{5}{24}]$.
As a consequence, for each $\tau \in [\frac{6}{5},\frac{4}{3}]$, there is a finite constant $C_\tau$ such that
\[ \int_{\R^3} \frac{dt}{\left[ ||u(t)||_\omega \right]^{\tau}} \leq C_\tau \]
uniformly over all determinant one choices of $\omega := \{\omega^j\}_{j=1}^8$.  There are two reasons why this result is noteworthy: the first is that there is a nontrivial range of $\tau$ for which uniform bounds hold despite the fact that the result is global rather than local. The second is that the degree one part of $P$ is degenerate, and so Theorem \ref{usefulthm} cannot be used to break $P$ into smaller homogeneous pieces as was done for the first concrete example. To verify degeneracy of the matrix
\[ \begin{bmatrix} 0  & -z_2 & -z_3 & -z_1 \\ -z_2 & -z_1 & 0 & 0 \\ -z_1 & -z_3 & -z_2 & 0 \\ -z_3 & 0 & -z_1 & 0 \end{bmatrix}, \]
one can scale $z_1$ by $\epsilon^4$ and $z_2, z_3$ by $\epsilon^{-2}$ and then scale the rows by factors $\epsilon^3, \epsilon^{-3},\epsilon^{3},\epsilon^{-3}$, respectively, and the columns by factors $\epsilon^6,1,1,\epsilon^{-6}$. The result is
\[ \begin{bmatrix} 0  & - \epsilon z_2 & - \epsilon z_3 & -\epsilon  z_1 \\ -\epsilon z_2 & -\epsilon z_1 & 0 & 0 \\ - \epsilon^{13} z_1 & - \epsilon z_3 & - \epsilon z_2 & 0 \\ - \epsilon z_3 & 0 & - \epsilon z_1 & 0 \end{bmatrix} \]
which clearly goes to zero as $\epsilon \rightarrow \infty$. As the scalings of rows, columns, and variables each preserves volume, it must be the case that this tile $\mathcal{T}$ has $|||P_{\mathcal T,t}|||_{\sigma} = 0$ for all $\sigma$ by the Hilbert-Mumford Criterion.

\section{Model Operators Revisited} 
\label{modelsec} 

This section is devoted to the proof of Theorem \ref{characterize}, which (modulo regularity assumptions) characterizes the so-called ``model operators'' exhibiting best-possible $L^p$-improving properties depending only on the first- and second-order properties of the mappings $\pi_1$ and $\pi_2$. In this sense, they represent those Radon-like transforms which are governed by curvature-like properties rather than higher-order torsion-like quantities. Theorem \ref{characterize} is the second such characterization of these operators, the first one being found in \cite{gressmancurvature}. Because the characterization in \cite{gressmancurvature} is also local and algebraic in nature, it should in principle be possible to deduce Theorem \ref{characterize} directly from the original result; however, the complexity of the original characterizing condition is quite forbidding, and in this case it appears that the simplest proof is to independently establish boundedness of \eqref{thebnd} using Theorem \ref{bigtheorem1} and the general characterization of nonlinear Radon-Brascamp-Lieb functionals established in \cite{testingcond}. The equivalence of the two characterizing criteria follows as a corollary of the $L^p$-improving estimates established rather than vice-versa.

The proof begins with the proposition below, which establishes that every pair $(U_0,V)$ admits a particularly simple block decomposition.
\begin{proposition}
Suppose that $u(t) := u^1(t) \wedge \cdots \wedge u^p(t) \in \Lambda^p(V^*)$ is nonvanishing and is Nash of complexity at most $K$. Let $U_0(t)$ be the family of $p$-dimensional subspaces parametrized by $(\domain,u)$. Then $(U_0,V)$ always admits a block decomposition with $\mstar = 0$ and $m = 1$ with $D_{00} = 0$ and $D_{01} = 1$. This decomposition is Nash with complexity at most $K^N$ for some constant $N$ depending only on dimensions. If $\dim V = q$, then $\dim V_1 = q-p$. \label{blockprop2}
\end{proposition}
\begin{proof}
 Let $e^1,\ldots,e^q$ be any basis of $V$.  For each $t$, consider the map $M(t)$ which sends $\omega \in V$ to $u(t) \cdot \omega \in \Lambda^{p-1}(V^*)$, where $\cdot$ is interior multiplication. By virtue of the formula \eqref{interiorprod} and the fact that the forms $u^1 \wedge \cdots \wedge \widehat{u^i} \wedge \cdots \wedge u^p$ are linearly independent, $u(t) \cdot \omega = 0$ if any only if $u^i(t) \cdot \omega = 0$ for each $i = 1,\ldots,p$. Because $u$ is nonvanishing, we know that the kernel of $M$ must have dimension $q - p$ at every point $t$ and in particular $M(t)$ must have rank $p$ at all points. Proposition \ref{kernelparam} thus guarantees that there always exist $V =: V_0(t) \supset V_1(t)$ with $\dim V_1 = q-p$ such that $u(t) \cdot \overline{v} = 0$ for all $\overline{v} \in V_1(t)$. Moreover, since we have established the stronger property that $u^i(t) \cdot \overline{v} = 0$, it follows that any $\overline{u} \in U_0(t)$ and any $v(t)$ parametrizing $V_1(t)$ satisfy $\overline{u} \cdot v(t) = 0$ as well.
 \end{proof}

In the block decomposition constructed in Proposition \ref{blockprop2}, the tiles $\mathcal{T}_0 := [0,0] \times [0,0]$ and $\mathcal{T}_1 := [0,0] \times [1,1]$ are useful and by the construction of Theorem \ref{usefulthm} admit points
\begin{align*}
[0]_{\mathcal{T}_0} & := \left( \frac{1}{p}; \frac{1}{p}, 0; 0 \right) \\
[d^{-1}]_{\mathcal{T}_1} & := \left( \frac{1}{p}; 0, \frac{1}{q-p}; \frac{1}{d} \right)
\end{align*}
which contain $(\frac{1}{p}; \frac{1}{q}, \frac{1}{q}; \frac{q-p}{d q})$ in their convex hull for coefficients $\theta_0 := \frac{p}{q}$ and $\theta_1 := \frac{q-p}{q}$.
Therefore \eqref{haveproved} implies 
\begin{equation} \int_{\domain} \frac{|||P_{\mathcal{T}_0,t}|||_0^{\frac{d p}{q-p}}  |||P_{\mathcal{T}_1,t}|||_{1/d}^d dt}{\left[ ||u(t)||_\omega \right]^{\frac{dq}{(q-p)p}}} \leq C K^C \label{curvaturecase} \end{equation}
for some constants $C$ depending only on dimension when $u$ is Nash of complexity at most $K$. Here (as in Theorem \ref{usefulthm}), the objects $P_{\mathcal{T}_0,t}$ and $P_{\mathcal{T}_1,t}$ may be initially defined using any factorization $u^1 \wedge \cdots \wedge u^p$ of $u$ (since all factorizations with $u^1 \wedge \cdots \wedge u^p = u$ will be adapted) and any adapted factorization $v^{j,j'}$ for $j \in \{0,1\}, j' \in \{1,\ldots,q_j\}$ and $(q_0,q_1) = (p,q-p)$ such that $\bigwedge_{j=0}^1 \bigwedge_{j'=1}^{q_j} v^{j,j'} = e^1 \wedge \cdots \wedge e^q$ for any volume-normalized basis $\{e^j\}$ of $V$ and such that $\{v^{1,j'}(t)\}_{j'=1}^{q-p}$ is a basis of $V_1(t)$. The goal of the present section will be to establish that the numerator in the integrand of \eqref{curvaturecase} is nonvanishing when the formula is applied to the integrals governing the bound \eqref{thebnd} of the Radon-like transform given by \eqref{doublefibration} when the canonical map $Q_z$ of the form \eqref{qdescribed} given by the Lie bracket \eqref{theqmap} is semistable.

\subsection{Sufficiency of semistability}
\label{suffsec}
This section establishes that \eqref{thebnd} holds for a suitable $a$ and constant $C$ when the map $Q_{z_0}$ given by \eqref{theqmap} is semistable in the GIT sense and the kernels $\ker d \pi_1$ and $\ker d \pi_2$ do not intersect. We will in fact show something slightly stronger, which is that the form of $a$ and the constant $C$ can be taken so that the arguments of \cite{gressmancurvature} also imply near-optimal $L^p$-improving inequalities when the manifold $\mathcal M$ and projections $\pi_1,\pi_2$ are merely smooth and lack Nash regularity.

Recall the setup: $\mathcal M$ is some manifold of dimension $n_1 + n - k$ equipped with projections $\pi_1 : \mathcal{M} \rightarrow \R^{n_1}$ and $\pi_2 : \mathcal{M} \rightarrow \R^{n}$, both of which have surjective differentials at each point $z \in \mathcal{M}$ and transverse kernels (i.e., $\ker d \pi_1 \cap \ker d \pi_2 = \{0\}$). As in the introduction, this allows (after coordinate transformations) one to assume that $\mathcal M$ is parametrized by coordinates $(x,t)$ belonging to some neighborhood of $z_0 = (x_0,t_0) \in \R^n \times \R^{n_1 -k}$. In these coordinates, $\pi_2 (x,t) = x$ and $\pi_1(x,t) = (t,\phi(x,t))$ for some smooth $\phi$.
For each $i = 1,\ldots,n_1-k$, one can define $Y^i := \partial_{t_i}$; these $Y^{1},\ldots,Y^{n_1-k}$ smoothly span the kernel of $d \pi_2$. Similarly, if $X^1,\ldots,X^{n-k}$ are smooth vector fields on $\mathcal M$ which form a basis of $\ker d \pi_1$ at all $z$ near $z_0$, then clearly each $X^i$ must belong to the span of the coordinate vector fields $\partial_{x_1},\ldots,\partial_{x_m}$.

Given $\phi$, one can build a defining function $\pi(x,y)$ for the Radon-like transform \eqref{doublefibration} by letting $y = (y',y'')$ for $y' \in \R^{n_1-k}$ and $y'' \in \R^{k}$ and fixing $\pi(x,y) := -y'' + \phi(x,y')$. If $x \mapsto \phi(x,t)$ is Nash of complexity at most $K$ for each $t$, then Proposition 1 of \cite{gressmancurvature} and the main theorem of \cite{testingcond} imply that
\begin{equation} \begin{split} \left| \int g(x) f(t,\phi(x,t)) a(x,t) dx dt \right| & \\ \leq C K^N \sup_{x,\omega} & \left[ \int \frac{|a(x,t)|^{r'_1} dt}{|| u(x,t)||_{\omega}^{r_1'-1}} \right]^{1/r_1'}||g||_{L^{r_2}}  ||f||_{L^{r_1}} \end{split} \label{applytesting} \end{equation}
for some constants $C$ and $N$ depending only on dimensions and codimensions. Here $u = u^1 \wedge \cdots \wedge u^k$, where \begin{equation} u^i (x,t) := \left( \frac{\partial \phi^i}{\partial x_1}(x,t),\ldots, \frac{\partial \phi^i}{\partial x_n}(x,t) \right), \ i=1,\ldots,k. \label{udefined} \end{equation}
The exponent $r_1$ equals $1 + \frac{k(n-k)}{n(n_1-k)}$ and $r_2 = \frac{k(n_1-k)}{n_1 (n-k)} + 1$. As usual, $\omega$ ranges over all volume-normalized bases of $\R^n$.  

Because $d \pi_1$ has surjective differential, the Jacobian matrix $\frac{\partial \phi}{\partial x}(x,t)$ must have full rank, which means that $u(x,t)$ is nonvanishing. Applying Proposition \ref{blockprop2} to this $u(x,t)$, the resulting spaces $V_1(x,t)$ are spanned exactly those vectors which are zero when paired with any $u^i(x,t)$, which means that $V_1(x,t)$ can be identified with the span of $\{ X^i \}_{i=1}^{n-k}$ at $(x,t)$.  More precisely, as Proposition \ref{blockprop2} yields smooth vectors $v^{1,j'}(x,t)$ (and, in fact, Nash) spanning $V_1$ at each $(x,t)$, it may be assumed without loss of generality that the $X^{j'}$ have been chosen so that
\[ X^{j'} = \sum_{\ell} (v^{1,j'}(x,t))_\ell \partial_{x_\ell} \]
for each $j'=1,\ldots,n-k$. Now because $u^i(x,s) \cdot v^{1,j'}(x,t)$ vanishes identically on the diagonal $s=t$, 
\begin{align*} \left[ \partial_{t_\ell} u^i(x,s) \cdot v^{1,j'}(x,t) \right] \Big|_{s = t} & = - \left[ \partial_{s_\ell} u^i(x,s) \cdot v^{1,j'}(x,t) \right] \Big|_{s=t} \\
& = -X^{j'} \left[ \partial_{t_\ell} \phi^i \right] \Big|_{(x,t)} = - X^{j'} Y^{\ell} \phi^i \Big|_{(x,t)}. \end{align*}
But it is also true that $X^{j'} \phi$ is identically zero, so it follows that $Y^{\ell} X^{j'} \phi$ is identically zero and that
\begin{equation} \left[ \partial_{t_\ell} u^i(x,s) \cdot v^{1,j'}(x,t) \right] \Big|_{s = t} = [Y^{\ell},X^{j'}] \phi^i \Big|_{(x,t)}, \label{commutator} \end{equation}
where $[\cdot, \cdot]$ denotes the Lie bracket. One therefore has by the definition \eqref{gitdef} that when $\mathcal{T}_1$ is the useful tile $[0,0] \times [1,1]$ 
\[ P_{\mathcal{T}_1,(x,t)} (w_1,w_2,w_3) := \sum_{i,j',\ell} (w_1)_i (w_2)_{j'} (w_3)_\ell [Y^{\ell},X^{j'}] \phi^i \Big|_{(x,t)}.\]
Because $[Y^\ell,X^{j'}]$ is known to belong to the span of $\{\partial_{x_j}\}_{j=1}^n$, one must have that
\[ [Y^\ell, X^{j'}] \phi^i \Big|_{(x,t)} = \left[ [Y^{\ell}, X^{j'}] \phi^i( \pi_2(x,t), s) \right] \Big|_{s = t} \]
for any $\ell$ and $j'$ at any point $(x,t)$. The map
\[ \Phi_{x,t} (Z) := \left[ Z \phi(\pi_2(x,t),s) \right] \Big|_{s=t} \]
is well-defined for any vector $Z$ tangent to $\mathcal M$ at $(x,t)$. It annihilates vectors in the kernel of $d \pi_2$ (namely, vectors in the $t$-directions) and vectors in the kernel of $d \pi_1$ (namely, vectors in the $x$ directions which annihilate $\phi(x,t)$). Thus one can regard $\Phi_{x,t}$ as a linear map from $T_{(x,t)}(\mathcal M) / (\ker d \pi_1 + \ker d \pi_2)$ into $\R^k$. Because $d \pi_1$ has surjective differential, this map $\Phi_{x,t}$ must be an isomorphism. Therefore
\[ P_{\mathcal{T}_1,(x,t)} (w_1,w_2,w_3) = - w_1 \cdot \Phi_{(x,t)} \left( Q_{(x,t)} \left( \sum_{j'} (w_2)_{j'} X^{j'} , \sum_{\ell} (w_3 )_\ell Y^{\ell} \right) \right), \]
where $Q_{(x,t)}$ is exactly the canonical bilinear map \eqref{qdescribed}. If $Q_{(x,t)}$ is semisimple, then it must follow that $|||P_{\mathcal{T}_1,(x,t)}|||_{1/(n_1-k)}$ is strictly positive, since $P_{\mathcal{T}_1,(x,t)}$ differs from $Q_{(x,t)}$ by the action of invertible linear transformations in each factor and because
\[ |\det M|^{-1/(n_1-k)} P_{\mathcal{T}_1,(x_0,t_0)} (w_1,w_2,M^t w_3) \]
is homogeneous of degree $0$ in $M$ (which means that the infimum over all $M \in \operatorname{SL}(n_1-k,\R)$ will equal the infimum over $\operatorname{GL}(n_1-k,\R)$).

As for the useful tile $\mathcal{T}_0$, one has
\[ P_{\mathcal{T}_0,(x,t)} (w_1,w_2,w_3) := \sum_{i,j} (w_1)_i (w_2)_{j} \left[ v^{0,j}(x,t) \cdot \nabla_x \right] \phi^i(x,t). \]
(Note in particular that this expression is homogeneous of degree $0$ as a function of $w_3$.) 
The quantity $|||P_{\mathcal{T}_0,(x,t)}|||_0^k$ is bounded below by the absolute value of the determinant of the form $P_{\mathcal{T}_0,(x,t)} (\cdot,\cdot)$ times some constant depending only on $k$ by virtue of \eqref{ftgit}. This determinant can't be zero, because if it were, it would force the wedge product $u(x,t)$ to vanish. It follows that when $Q_{(x,t)}$ is semistable, the numerator of the integrand in \eqref{curvaturecase} is nonzero at $(x,t)$. Here $d := n_1 - k$, $p := k$, and $q = n$. This numerology gives the identity $dq/((q-p) p) = r_1' - 1$. As long as the bump function $a$ is chosen with small support so that $|a(x,t)|^{r'_1}$ is dominated by the numerator in \eqref{curvaturecase}, the right-hand side of \eqref{applytesting} will be bounded above by $C K^{N'} ||g||_{L^{r_2}} ||f||_{L^{r_1}}$ for constants $C$ and $N'$ which depend only on the parameters ($k, n,n_1$). Because the numerator in \eqref{curvaturecase} depends continuously on the degree $2$ Taylor jets of $\phi$ as $(x,t)$ varies, one can conclude that the constant $C$ in the main inequality, i.e., \eqref{thebnd} in Theorem \ref{characterize}, can be taken to be some uniformly bounded constant times $K^{N'}$ provided the support of $a$ is sufficiently small and $a$ is uniformly bounded. The constant will also remain bounded if $\phi$ is altered by a sufficiently small $C^3$ perturbation.

\subsection{Necessity of semistability}
\label{neccsec}

The argument in this section has two closely-related parts. The first is to establish that transversality of $\ker d \pi_1$ and $\ker d \pi_2$ is a necessary consequence of \eqref{thebnd}. As a side note, it is natural to ask whether transversality is itself a consequence of the semistability of $Q_{z_0}$. The following simple example shows that this is not the case: for the manifold $\mathcal M := \R^{5}$ and with $n := n_1 := 3, k := 2$, let
\[ \pi_1(x) := ( x_1 x_3 - x_4, x_3, x_5) \text{ and } \pi_2(x) := (x_1 x_2 - x_5, x_2, x_4). \]
The kernel of $d \pi_1$ is spanned by
\[ X^1 := \frac{\partial}{\partial x_2} \text{ and } X^2 := \frac{\partial}{\partial x_1} + x_3 \frac{\partial}{\partial x_4}, \]
while $\ker d \pi_2$ is spanned by
\[ Y^1 := \frac{\partial}{\partial x_3} \text{ and } Y^2 := \frac{\partial}{\partial x_1} + x_2 \frac{\partial}{\partial x_5}. \]
Generically the kernels $\ker d \pi_1$ and $\ker d \pi_2$ are transverse, but both contain $\frac{\partial}{\partial x_1}$ at $x = 0$. The map $Q_z$ is well-defined at $0$, but it maps into a two-dimensional space at the origin (while at a generic point its image space is only one dimensional). Despite the anomalous dimensionality, $Q_z$ is indeed semistable at the origin. However, because the dimensionality changes discontinuously, knowing $Q_z$'s behavior at just the origin is inadequate to understand $Q_z$ at other nearby points (and, as it happens, there are points in every neighborhood of the origin at which $Q_z$ is not semistable).

Returning now to the general case, because $\pi_2(z)$ is assumed to have surjective differential, it is possible near any $z_0 \in \mathcal{M}$ to construct a local coordinate system in the variables $(x,t)$ for $x \in \R^n$ and $t \in \R^{n_1-k}$ such that $z_0$ has coordinates $(0,0)$ and $\pi_2(x,t) = x$ for all $(x,t)$ near the origin. The kernel of $d \pi_2$ is then spanned by $\{\partial_{t_\ell}\}_{\ell=1}^{n_1-k}$. Let $W$ be the linear subspace of $\R^{n+n_1-k}$ which passes through $(x,t) = (0,0)$ and has tangent space spanned by $\ker d \pi_1|_{(0,0)} + \ker d \pi_2|_{(0,0)}$. Finally, for all $\delta > 0$, let $E_\delta$ be the ellipsoid in $\R^{n+n_1-k}$ which intersects $W$ in a ball of radius $\delta$ and extends an amount $\delta^2$ in directions orthogonal to $W$ (i.e., the projection of $E_\delta$ onto the orthogonal complement of $W$ is a ball of radius $\delta^2$).

Because $d \pi_1$ is surjective, its kernel must have dimension $n-k$ at every point. Suppose that the intersection of $\ker d \pi_1$ and the span of $\{\partial_{t_\ell}\}_{\ell=1}^{n_1-k}$ has dimension $r \geq 0$. The volume of $E_\delta$ will be exactly $c \delta^{(n-k)+(n_1-k)-r} (\delta^2)^{k+r} = c\delta^{n + n_1 + r}$ for some constant $c$ depending only on dimensions. Let $X^1,\ldots,X^{n-k-r}$ be vectors in $\ker d \pi_1$ which are linearly independent modulo the span of $\{\partial_{t_\ell}\}_{\ell=1}^{n_1-k}$. This makes their projections onto the $x$ coordinate directions linearly independent, and after a linear change of variables in $x$, one may assume without loss of generality that $d \pi_2(X^j) = \partial/\partial x_j$ for each $j=1,\ldots,n-k-r$ at the point $(x,t) = (0,0)$. It follows that $\pi_2(E_\delta)$ has volume exactly $c' \delta^{n-k-r} (\delta^2)^{k+r} = c' \delta^{n+k+r}$ for some fixed constant $c'$ independent of $\delta$.

The projection $\pi_1(E_\delta)$ is not an exact ellipsoid, but as the projection of $W$ via $d \pi_1$ has dimension $n_1 - k - r$ at the origin $(x,t) = (0,0)$, Taylor's theorem guarantees that all points in $\pi_1(E_\delta)$ will be distance at most $C \delta^2$ away from a point in $d \pi_1|_{(0,0)} (E_\delta \cap W)$, which is itself a ball of diameter $C' \delta$ in a subspace of dimension $n_1 - k - r$. Thus $\pi_1(E_\delta)$ will, for all sufficiently small $\delta$, be contained in a measurable set of volume comparable to $C\delta^{n_1 + k + r}$. Letting $f_\delta$ be the characteristic function of this measurable set and letting $g_\delta$ be the characteristic function of the projection $\pi_2(E_\delta)$, it follows when $a(0,0) > 0$ that
\[ \int g_\delta(\pi_2(x,t)) f_\delta(\pi_1(x,t)) a(x,t) dx dt \geq C'' |E_\delta| = C''' \delta^{n + n_1 + r} \]
for constants $C''$ and $C'''$ independent of $\delta$, provided $\delta$ is sufficiently small. Boundedness of the form \eqref{thebnd} requires that
\[ C''' \delta^{n + n_1 + r} \leq C^{iv} \left( \delta^{n + k + r}\right)^{\frac{1}{p_2}} \left( \delta^{n_1 + k + r} \right)^{\frac{1}{p_1}} \] 
for all sufficiently small positive $\delta$ and some positive constants $C''',C^{iv}$, where
\begin{align*} \frac{1}{p_2} & := \frac{n_1(n-k)}{k(n_1-k) + n_1(n-k)} = \frac{n_1(n-k)}{nn_1 - k^2}, \\
 \frac{1}{p_1} & := \frac{n(n_1-k)}{k(n-k) + n(n_1-k)} = \frac{n(n_1-k)}{nn_1 - k^2}.
\end{align*}
However, it is easy to check that 
\[ n_1 + n = (n+k) \frac{1}{p_2} + (n_1+k) \frac{1}{p_1} \]
and
\[ \frac{1}{p_1} + \frac{1}{p_2} = 1 + \frac{(n-k)(n_1-k)}{nn_1 -k^2} > 1, \]
so letting $\delta \rightarrow 0^+$ forces the conclusion that $r = 0$, namely, that $\ker d \pi_1 \cap \ker d \pi_2 = \{0\}$.

Knowing now that $\ker d \pi_1 \cap \ker d \pi_2 = \{0\}$, one can use a slightly more refined version of the coordinates $(x,t)$ just constructed. More specifically, after a linear transformation of $\R^{n_1}$, i.e., replacing $\pi_1$ by $A \pi_1$ if necessary (where $A \in \operatorname{GL}(n_1,\R)$) so that $d \pi_1(\ker d \pi_2)$ lies in the first $n_1-k$ coordinate directions in $\R^{n_1}$, coordinates $(x,t)$ may be chosen so that $x = \pi_2(x,t)$ and $(t,\phi(x,t)) = \pi_1(x,t)$ for all $(x,t)$ near the origin, where $\phi$ is some mapping into $\R^k$ which has $\phi = \partial \phi /\partial t_\ell = 0$ at the origin for all $\ell=1,\ldots,n_1-k$.  To be precise, we define $x := \pi_2(z)$ and $t := (\pi_1^{1}(z),\ldots\pi_1^{n_1-k}(z))$. Here $(\pi_1^1,\ldots,\pi_1^{n_1})$ is the coordinate representation of $\pi_1$. Let $\psi(x,t)$ be the inverse of the map $z \mapsto (x,t) = (\pi_2(z), (\pi_1^1(z),\ldots,\pi_1^{n_1-k}(z)))$.
When $x$ is fixed, $\psi(x,t)$ lies on the $(n_1-k)$-dimensional submanifold $\pi_2^{-1}(\{x\})$, and having chosen the coordinate directions of $\pi_1$, it is possible to assume that $d \pi^{n_1-k+1}_1,\ldots,d \pi_1^{n_1}$ vanish when restricted to the tangent space of $\pi_2^{-1}(\{x\})$ at $z_0$. This is what guarantees that $\phi(x,t) := (\pi^{n_1-k+1}_1(\psi(x,t)),\ldots,\pi^{n_1}_1(\psi(x,t)))$ has $\phi = \partial \phi /\partial t_\ell = 0$ vanishing at the origin (since derivatives of $\psi$ in the $t$ direction are tangent to $\pi_2^{-1}(\{x\})$). (As a side note, observe that this means that every Radon-like transform can be locally parametrized in such a way that it has the same form as the operators in Theorem \ref{usefulradon}.)

Now consider the transformation $\Phi(t,y) := (t,y - \phi(0,t))$ for $t \in \R^{n_1-k}$ and $y \in \R^k$. Near the origin, this is a diffeomorphism. Moreover,
in the $(x,t)$ coordinates constructed above, $\Phi \circ \pi_1(x,t) = (t, \phi(x,t) - \phi(0,t))$, which means that $\Phi \circ \pi_1(x,t) = (t, \tilde \phi(x,t))$ for some map $\tilde \phi(x,t)$ whose $t$ derivatives of \textit{all} orders vanish when $x=0$. Thus, after making a nonlinear change of variables on $\R^{n_1}$, one may assume without loss of generality that derivatives of $\phi$ of all orders with respect to $t$ variables vanish at the origin. Lastly, since $d \pi_1$ is surjective, it may be assumed by rotation of the $x$ coordinates that $d \phi(\partial_{x_{n-k+1}}),\ldots,d \phi(\partial_{x_{n}})$ span the image of $d \phi$ and $d \phi(\partial_{x_1}) = \cdots = d \pi(\partial_{x_{n-k}}) = 0$ at the origin. As a result, in these coordinates one has vector fields
\[ X^j = \frac{\partial}{\partial x_j} + \sum_{j'=n-k+1}^{n} c_{jj'}(x,t) \frac{\partial}{\partial x_{j'}} \]
for $j=1,\ldots,n-k$ which span the kernel of $d \pi_1$ at all $(x,t)$ near $(0,0)$, where $c_{jj'}$ vanishes at the origin for each pair $(j,j')$.

Consider a point $(x,t)$ such that $x_1,\ldots,x_{n-k}$ and $t_1,\ldots,t_{n_1-k}$ are of size $\delta$ and $x_{n-k+1},\ldots,x_{n}$ are of size $\delta^2$. It must be the case that
\begin{equation} \phi(x,t) = \sum_{j=n-k+1}^{n} x_{j} e^{j} + \sum_{j,j'=1}^{n-k} f^{j,j'} x_j x_{j'} + \sum_{j=1}^{n-k} \sum_{\ell=1}^{n_1-k} g^{j,\ell} x_j t_{\ell}  + o(\delta^2) \label{newdef} \end{equation}
as $\delta \rightarrow 0^+$ for some vectors $e^j$, $f^{j,j'}$ and $g^{j,\ell}$. Without loss of generality, the $e^j$ can be assumed to belong to the standard basis of $\R^k$. Moreover,
\[ g^{j,\ell} = \frac{\partial}{\partial x_j} \frac{\partial}{\partial t_\ell} \phi |_{(0,0)} = X^j \frac{\partial}{\partial t_{\ell}} \phi |_{(0,0)}  = \left[ X^j ,\frac{\partial}{\partial t_\ell} \right] \phi |_{(0,0)} \]
because $X^j \phi$ is identically zero and therefore its derivative with respect to $t$ is as well.  Now $\partial_{\partial t_\ell} \phi(x,t) = 0$ at the origin, and likewise $X^j \phi(x,t) = 0$ at the origin because the $X^j$ belong to the kernel of $d \pi_1$ and consequently of $d \phi$ as well. Thus if $W^{j}_{\ell}$ is any differential operator which equals $[X^j,\partial/\partial t_{\ell}]$ modulo $\ker d \pi_1 + \ker d \pi_2$ at the origin, then
\[ g^{j,\ell} = W^j_\ell \phi  |_{(0,0)}. \]
 If the bilinear curvature map $Q_{z_0}$ is degenerate, this means that given any $\epsilon > 0$, there are volume-preserving linear coordinate transformations of $x_1,\ldots,x_{n-k}$ (acting as a matrix on the index $j$), $t_1,\ldots,t_{n-k}$ (acting as a matrix on the index $\ell$), and the space $\R^k$ into which $\phi$ maps so that, in the transformed coordinates, the coefficients $g^{j,\ell}$ of the terms $x_j t_\ell$ in \eqref{newdef} have magnitude at most $\epsilon$. Making a compensatory linear transformation of $x_{n-k+1},\ldots,x_{n}$ which is also volume-preserving (which inverts the transformation of $\R^k$ just applied), it may be further assumed that $e^{n-k+1},\ldots,e^{n}$ are the standard basis of $\R^k$ regardless of the chosen value of $\epsilon$.

Now consider the set $\tilde E_\delta$ of points $(x,t)$ in $\mathcal M$ given by $|x_j| \leq \delta$, $j=1,\ldots,n-k$, $|t_\ell| \leq \delta$, $\ell = 1,\ldots,n_1-k$, and 
\[ \left|\left| \sum_{j=n-k+1}^{n} x_{j} e^j + \sum_{j,j'=1}^{n-k} f^{j,j'}_{\epsilon} x_j x_{j'} \right| \right| \leq \epsilon \delta^2, \]
where the norm $||\cdot||$ is the $\ell^\infty$ norm on $\R^k$. The subscript $\epsilon$ on $f_\epsilon^{j,j'}$ is meant to indicate that in this $\epsilon$-dependent coordinate system, the $f^{j,j'}_\epsilon$ could in principle be quite large. However, by Fubini's theorem, 
the set $\tilde E_{\delta}$ has volume exactly $2^{n + n_1 - k} \epsilon^k \delta^{n + n_1}$ for all $\delta$ (sufficiently small that it belongs to the neighborhood on which the coordinates are defined). Its projection $\pi_2(\tilde E_\delta)$ likewise has volume $2^{n} \epsilon^k \delta^{n+k}$ by Fubini's Theorem. The projection $\pi_1(\tilde E_\delta)$ will be contained in a set of volume $C' \delta^{n_1-k} (\epsilon \delta^2)^{k} = C' \epsilon^{k} \delta^{n_1+k}$ for some constant $C'$ independent of $\epsilon$ and $\delta$, provided that $\delta$ is sufficiently small (depending now on $\epsilon$ as well). This is because $\phi(x,t)$ differs from 
\[ \sum_{j=n-k+1}^{n} x_{j} e^j + \sum_{j,j'=1}^{n-k} f^{j,j'}_{\epsilon} x_j x_{j'} \] 
by a vector of size $C \epsilon \delta^2 + o(\delta^2)$ thanks to \eqref{newdef}, which means that $||\phi(x,t)|| \leq (C+1) \epsilon \delta^2 + o(\delta^2)$ (so one can see that $\pi_1(\tilde E_\delta)$ belongs to the product of an $(n_1-k)$-dimensional ball of radius $\delta$ and a $k$-dimensional ball of radius $(C+2) \epsilon \delta^2$ for all sufficiently small $\delta$).
 As above, this means that when $a(0,0) > 0$ and $\delta$ is sufficiently small,
\[ \epsilon^k \delta^{n+n_1} \leq C''' \left( \epsilon^k \delta^{n+k} \right)^{\frac{1}{p_2}} \left(\epsilon^{k} \delta^{n_1+k} \right)^\frac{1}{p_1}, \]
 where $C'''$ is independent of both $\epsilon$ and $\delta$.  Because $\ker d \pi_1 \cap \ker d \pi_2 = \{0\}$, we know already that the powers of $\delta$ on both sides of the inequality are equal. Therefore it must be the case that
\[ \epsilon^k \leq C''' \epsilon^{k \left( \frac{1}{p_2} + \frac{1}{p_1} \right)}, \]
but as $p_2^{-1} + p_1^{-1} > 1$, this inequality clearly cannot hold for sufficiently small $\epsilon$. This means that the inequality \eqref{thebnd} cannot hold for any finite constant $C$.

\section{Concluding Remarks}
\label{remarksec}
\begin{itemize}
\item One interesting application of Theorems \ref{bigtheorem1} and \ref{usefulthm} not included here concerns the Oberlin affine curvature condition. The paper \cite{gressman2019} characterizes the submanifolds of codimension $k$ in $\R^n$ with best-possible affine Hausdorff dimension for each pair $(k,n)$, but there are many situations of high symmetry (e.g., submanifolds with product structure) in which one would expect submanifolds to have affine Hausdorff dimension which is not trivial yet not equal to the extremal value. Theorem \ref{bigtheorem1} and \ref{usefulthm} can be applied in a number of such cases; if $\gamma(t)$ is a parametrization of such a submanifold, one takes $p=1$ and $u(t) = \gamma(t)$. To construct block decompositions, simply let $V_j(t)$ be all vectors which are orthogonal to $\partial^\alpha \gamma(t)$ for all $|\alpha| \leq j$. In the extremal case, it is easy to prove that the dimension of $V_j(t)$ does not depend on $t$, but this will also be the case in many other situations as well (and moreover, one can always find a dense open set on which $\dim V_j(t)$ is constant). This allows one to define weights as in \eqref{haveproved}; the resulting inequalities imply the Oberlin affine curvature condition for some exponent depending on $p \sigma$ and $n$.
\item Many open questions remain in the case when the weights given by \eqref{haveproved} turn out to be identically zero. For certain combinations of dimensions, there are structural reasons why this might indeed happen more often than one would otherwise expect. Just as an example, if $Q : \R^a \times \R^b \rightarrow \R^c$ is bilinear, it is relatively straightforward to show that it can only be semistable if $a \leq bc, b \leq ac$, and $c \leq ab$. However, these are not the only necessary conditions and there are curious situations (e.g., $a = 2, b = 3, c = 5$) in which every $Q$ fails to be semistable. In terms of Theorem \ref{bigtheorem1}, one question is whether the example of Section \ref{degensec} is representative of a generic phenomenon in cases where lower-order parts are necessarily degenerate.
\end{itemize}

\bibliography{mybib}

\end{document}